\documentclass[reqno,12pt]{amsart}
\usepackage[latin1]{inputenc}
\usepackage[english]{babel}
\usepackage{amsmath, amsthm, amssymb, amsopn, amsfonts, amstext, stmaryrd, enumerate, color, mathtools, mathrsfs, enumitem, url, accents}

\usepackage{xcolor}

\usepackage[final]{showkeys}

\usepackage{hyperref}

\usepackage[margin=1.2in]{geometry}

\hypersetup{
	colorlinks=true,
	linkcolor=blue,
	citecolor=red,
	urlcolor=black,
	linktoc=all
}

\newcommand{\B}{\mathcal{B}}
\newcommand{\C}{\mathcal{C}}

\renewcommand{\H}{\mathbb{H}}

\newcommand{\J}{\mathcal{J}}

\renewcommand{\L}{\mathcal{L}}
\newcommand{\N}{\mathbb{N}}
\newcommand{\R}{\mathbb{R}}

\newcommand{\loc}{{\rm loc}}
\newcommand{\dive}{{\mbox{\normalfont div}}}

\newcommand{\dist}{{\mbox{\normalfont dist}}}
\newcommand{\PV}{\mbox{\normalfont P.V.}}
\newcommand{\Haus}{\mathcal{H}}
\newcommand{\red}{{\partial^* \!}}
\renewcommand{\Cap}{\mbox{\normalfont Cap}}

\DeclareMathOperator{\diam}{diam}
\DeclareMathOperator{\supp}{supp}

\DeclareMathOperator{\osc}{osc}

\DeclareMathOperator{\Per}{Per}

\def\Xint#1{\mathchoice
{\XXint\displaystyle\textstyle{#1}}%
{\XXint\textstyle\scriptstyle{#1}}%
{\XXint\scriptstyle\scriptscriptstyle{#1}}%
{\XXint\scriptscriptstyle\scriptscriptstyle{#1}}%
\!\int}
\def\XXint#1#2#3{{\setbox0=\hbox{$#1{#2#3}{\int}$ }
\vcenter{\hbox{$#2#3$ }}\kern-.6\wd0}}

\def\dashint{\Xint-}

\newlength{\dhatheight}

\numberwithin{equation}{section}

\theoremstyle{plain}
\newtheorem{definition}{Definition}[section]
\newtheorem{theorem}[definition]{Theorem}
\newtheorem{proposition}[definition]{Proposition}
\newtheorem{lemma}[definition]{Lemma}
\newtheorem{corollary}[definition]{Corollary}

\theoremstyle{definition}
\newtheorem{remark}[definition]{Remark}

\renewcommand{\le}{\leqslant}

\renewcommand{\ge}{\geqslant}

\title{A gradient estimate for nonlocal minimal graphs}

\author{Xavier Cabr\'e}

\author{Matteo Cozzi}

\address{
\newline
\textit{Xavier Cabr\'e \textsuperscript{1,2}}
\newline
\textsuperscript{1} Universitat Polit\`ecnica de Catalunya, Departament de Matem\`{a}tiques, Diagonal 647, 08028 Barcelona, Spain
\newline
\textsuperscript{2} ICREA, Pg. Lluis Companys 23, 08010 Barcelona, Spain
\newline
\textit{E-mail address}: \textit{\tt xavier.cabre@upc.edu}
}

\address{\vspace{-6pt}
\newline
\textit{Matteo Cozzi \textsuperscript{3}}
\newline
\textsuperscript{3} University of Bath, Department of Mathematical Sciences, Claverton Down, Bath BA2 7AY, UK
\newline
\textit{E-mail address}: \textit{\tt m.cozzi@bath.ac.uk}
}

\thanks{Both authors are members of the Barcelona Graduate School of Mathematics, are part of the Catalan research group 2017 SGR 1392, and are supported by the MINECO grants~MTM2014-52402-C3-1-P and~MTM2017-84214-C2-1-P. The second author has also been supported by the~``Mar\'ia de Maeztu'' MINECO grant~MDM-2014-0445 and is currently supported by a Royal Society Newton International Fellowship}



\begin{document}

\begin{abstract}
We consider the class of measurable functions defined in all of~$\R^n$ that give rise to a nonlocal minimal graph over a ball of~$\R^n$. We establish that the gradient of any such function is bounded in the interior of the ball by a power of its oscillation. This estimate, together with previously known results, leads to the~$C^\infty$ regularity of the function in the ball. While the smoothness of nonlocal minimal graphs was known for~$n = 1, 2$---but without a quantitative bound---, in higher dimensions only their continuity had been established.

To prove the gradient bound, we show that the normal to a nonlocal minimal graph is a supersolution of a truncated fractional Jacobi operator, for which we prove a weak Harnack inequality. To this end, we establish a new universal fractional Sobolev inequality on nonlocal minimal surfaces.

Our estimate provides an extension to the fractional setting of the celebrated gradient bounds of~Finn and of~Bombieri,~De Giorgi~\& Miranda for solutions of the classical mean curvature equation.
\end{abstract}

\maketitle



\section{Introduction}

\noindent
Nonlocal minimal surfaces were first introduced and studied by Caffarelli, Roquejoffre \& Savin in the seminal 
work~\cite{CRS10}, where the authors defined them as the minimizers of a fractional perimeter functional. They were motivated by the work of Caffarelli \& Souganidis~\cite{CS10} on the asymptotic configurations of a threshold dynamics scheme governed by a L\'evy-type jump diffusion. They appeared also in~\cite{I09}, where Imbert studied related nonlocal geometric flows arising in dislocation dynamics in crystals. Later, Savin \& Valdinoci~\cite{SV12} showed the relevance of nonlocal minimal surfaces by proving, through~$\Gamma$-convergence techniques, that they are the limiting configurations for Ginzburg-Landau energies modeling phase-separation phenomena in the presence of strongly nonlocal interactions.

Since their introduction, nonlocal minimal surfaces have attracted much attention, first and foremost to understand their regularity and to make progresses towards their classification. We refer the reader to~\cite{DV16},~\cite{L15},~\cite[Section~7]{CF17}, and~\cite[Chapter~6]{BV16} for general introductions to this topic.

\subsection{The gradient estimates} \label{mainressub}

Our main result establishes the regularity, through a gradient estimate, of a particular type of nonlocal 
minimal surfaces:~$\alpha$-minimal graphs. To state it, let~$\alpha \in (0, 1)$ and~$n \ge 1$ be an integer. Given a bounded open set~$\Omega \subset \R^{n + 1}$ and a measurable set~$E \subseteq \R^{n + 1}$, we define the~\emph{$\alpha$-perimeter} of~$E$ inside~$\Omega$ as the quantity
$$
\Per_{\alpha}(E; \Omega) := I_\alpha(E \cap \Omega, \R^{n + 1} \setminus E) + I_\alpha(E \setminus \Omega, \Omega \setminus E),
$$
where
$$
I_\alpha(A, B) := \int_A \int_B \frac{dx dy}{|x - y|^{n + 1 + \alpha}}
$$
for any two disjoint measurable sets~$A, B \subseteq \R^{n + 1}$. If it happens that~$\Per_\alpha(E; \Omega)$ is 
finite and~$\Per_{\alpha}(E; \Omega) \le \Per_{\alpha}(F; \Omega)$
for every measurable set~$F \subseteq \R^{n + 1}$ such that~$F \setminus \Omega = E \setminus \Omega$, we then call~$E$ a~\emph{minimizer} of the~$\alpha$-perimeter inside~$\Omega$, and its boundary~$\partial E$ a \emph{nonlocal minimal surface of order~$\alpha$} in~$\Omega$---or simply an~\emph{$\alpha$-minimal surface} in~$\Omega$. For an unbounded open set~$\Omega$, we extend this definition by saying that~$\partial E$ is an~$\alpha$-minimal surface in~$\Omega$ if~$E$ minimizes the~$\alpha$-perimeter in every open set compactly contained in~$\Omega$.

Our main result deals with a particular class of nonlocal minimal surfaces, namely those sets that locally minimize 
the~$\alpha$-perimeter inside the infinite vertical cylinder~$B'_R \times \R$ over an open 
ball~$B'_R = \{ x' \in \R^n : |x'| < R \}$ and that are globally the subgraph of some function~$u$ defined in
all of~$\R^n$. It states that~$u$ is smooth inside~$B'_R$ and that, locally, the gradient of~$u$ is controlled by its oscillation.

\begin{theorem} \label{localgradestthm}
Let~$n \ge 1$ and~$\alpha \in (0, 1)$. Let~$E \subset \R^{n + 1}$ be the global subgraph
$$
E = \Big\{ (x', x_{n + 1}) \in \R^{n} \times \R : x_{n + 1} < u(x') \Big\}
$$
of a measurable function~$u: \R^n \to \R$, bounded in~$B'_r$ for some~$r > 0$, and
assume that~$\partial E$ is an~$\alpha$-minimal surface in the cylinder~$B'_{2 r} \times \R$.

Then,~$u \in C^\infty(B'_r)$ and its gradient satisfies
\begin{equation} \label{gradbound}
\| \nabla_{\! x'} u \|_{L^\infty(B'_r)} \le C \left( 1 + \frac{\osc_{B'_r} u}{r} \right)^{n + 1 + \alpha},
\end{equation}
for some constant~$C$ depending only on~$n$ and~$\alpha$.
\end{theorem}

Note that, in~\eqref{gradbound}, the supremum of the gradient of~$u$ in~$B_r'$ is controlled by its oscillation in the same ball~$B_r'$. However, this is not an estimate up to the boundary, since~$\partial E$ is assumed to be~$\alpha$-minimal in the cylinder over the larger ball~$B_{2 r}'$. In fact, a gradient estimate up to the boundary cannot hold, by the boundary~\emph{stickiness} phenomenon commented in Subsection~\ref{plateau}.

The gradient bound~\eqref{gradbound} is the main novelty of Theorem~\ref{localgradestthm}. Once it is established, the smoothness of the
surface~$\partial E$ follows from the results of~Caffarelli, Roquejoffre~\&~Savin~\cite{CRS10}, Figalli~\&~Valdinoci~\cite{FV17}, and Barrios, Figalli~\&~Valdinoci~\cite{BFV14}. Prior to this work, no gradient estimate was available,
even in dimension~$n=1$. On the other hand, the smoothness of such graphs was known for~$n = 1$ and~$n = 2$---but without a quantitative bound---by the results of Savin~\&~Valdinoci~\cite{SV13} for~$n = 1$, and of~\cite{FV17} and Dipierro, Savin~\&~Valdinoci~\cite{DSV16} for~$n = 2$. In higher dimensions they were only known to be continuous---with no modulus of continuity being established---by a result of~\cite{DSV16}.

That the gradient is controlled by a power of the oscillation (instead of its exponential as in classical minimal 
graphs) is a consequence of the possibility of having an additional weighted~$L^1$
term in one of our main results: a weak Harnack inequality for elliptic integro-differential equations on nonlocal minimal surfaces. Such a term does not appear in the local case, and, to our knowledge, its presence was first clearly observed, in the nonlocal flat Euclidean case, by~Ros-Oton~\&~Serra~\cite[Theorem~2.2]{RS16}. If one ignores this term, it is still possible to control the gradient by an exponential of the oscillation (as in the local case) via a covering argument analogous to that of~\cite[Corollary~3.2]{DSJ11}---see the end of Section~\ref{weakharsec} for more comments on this.

As we will see in the next subsection, a significant application of 
Theorem~\ref{localgradestthm} regards the Dirichlet or Plateau problem for nonlocal minimal surfaces
in a bounded domain when the exterior datum is a locally bounded graph. This problem is known to enjoy existence and uniqueness. Now, by our result, it also has
interior~$C^\infty$ regularity. Note that the exterior datum may be
discontinuous---it needs only to be the graph of a locally bounded function. In addition, a maximum principle from~\cite[Section~3]{DSV16} allows to control the right-hand side of~\eqref{gradbound} by the oscillation of the exterior datum~$g$ in a large enough annulus. This leads to the gradient estimate~\eqref{estdatum} below, in which the right-hand side depends only on the exterior datum.


Our proof of Theorem~\ref{localgradestthm} relies, in its essential strategy, on two new fundamental ingredients:
\begin{enumerate}[label=(\alph*),leftmargin=25pt]
\item the superharmonicity of the vertical component of the normal to a nonlocal minimal graph with respect to a truncated 
fractional Jacobi operator, and
\item a universal fractional Sobolev inequality on nonlocal minimal surfaces.
\end{enumerate}
We will use as well three other known results:
\begin{enumerate}[label=(\alph*),leftmargin=25pt]
\setcounter{enumi}{2}
\item the density estimates of Caffarelli, Roquejoffre \& Savin~\cite{CRS10},
\item the perimeter bound of Cinti, Serra \& Valdinoci~\cite{CSV16}, and
\item the estimate of~Savin \& Valdinoci~\cite{SV13} on the Hausdorff dimension of the singular set of a nonlocal minimal surface.
\end{enumerate}

In the following subsections we will state the results~(a) and~(b), and we will 
outline the proof of Theorem~\ref{localgradestthm}. Briefly, point~(c) will be essential to establish~(b), while~(d) will play an important role in the proof of~(a). Next,~(b)-(c)-(d) will be used to establish, through a Moser iteration, a new weak Harnack inequality for fractional equations on~$\alpha$-minimal surfaces. This inequality, applied to the vertical component of the normal, thanks to~(a), will lead easily to our gradient estimate. Point~(e) will be important to obtain, through a capacitary argument, the same bounds when~$u$ is not a priori known to be smooth.

It is worth noting that the Jacobi operator in~(a) will be of order~$2 s = 1 + \alpha > 1$. As we will comment more extensively in Subsection~\ref{proofoutlinesub}, this fact prevents us from using a simple method of~\cite{S06,RS16} to prove the fractional weak Harnack inequality. The issue here is the lack of information available a priori on the local geometry of nonlocal minimal surfaces, needed to control the Jacobi operator when applied to smooth barrier functions. Thus, we are forced to run a Moser iteration, which is based on the implementation of this technique in the nonlocal setting as first accomplished by Kassmann~\cite{Kas09,Kas11} in the flat case.

It is known that~$\alpha$-minimal surfaces converge to classical minimal surfaces as~$\mbox{$\alpha \uparrow 1$}$. This can be deduced from the results of~\cite{BBM01,D02,P04}---see also the more recent references~\cite{ADM11,CV11} for the statement in the exact same terminology as ours.
For classical minimal graphs, an estimate similar to~\eqref{gradbound} was established by~Finn~\cite{F63} for~$n = 2$, and by~Bombieri, De~Giorgi~\& Miranda~\cite{BDM69} in higher dimensions (the case~$n = 1$ is clearly trivial for classical minimal surfaces, while this is not the case in the fractional setting). They showed that the gradient of any solution~$u$ to the minimal graph equation in a ball~$B'_{2 r}$ of~$\R^n$ satisfies
\begin{equation} \label{BDMest}
\| \nabla_{\! x'} u \|_{L^\infty(B_r')} \le \exp \left\{ C \left( 1 + \frac{\sup_{B'_{2 r}} u - u(0)}{r} \right) \right\},
\end{equation}
for some dimensional constant~$C > 0$. 
After their works, several new proofs of gradient bounds similar to~\eqref{BDMest} were obtained, 
most notably in~\cite{T72,BG72,S76,K86,W98,DSJ11}.

Estimate \eqref{BDMest} for classical minimal graphs was shown to be optimal in~\cite{F63}. 
That is, the gradient cannot be controlled by a function of the oscillation growing slower 
than an exponential. We do not know whether inequality \eqref{gradbound} is optimal in the fractional case. 
When trying to adapt the example of \cite{F63}, it seems necessary to have a better understanding of
a delicate issue for nonlocal minimal graphs discovered in~\cite{DSV17}: the so-called boundary 
\emph{stickiness}. In Subsection~\ref{plateau} we will comment further on this issue.

As a consequence of the sharpness of \eqref{BDMest} in the classical setting, the constant~$C$ 
in~\eqref{gradbound} must blow-up as~$\alpha \uparrow 1$. To get a gradient bound uniform 
in~$\alpha$---of the form~\eqref{BDMest} for instance---, one should primarily obtain a Sobolev 
inequality on nonlocal minimal graphs, such as the one that we establish in Theorem~\ref{sobineintrothm}, 
but governed by a constant displaying the right dependence in~$\alpha$ as~$\alpha \uparrow 1$ 
when~$s = (1 + \alpha) / 2$ and~$p = 2$. This seems to be a non-trivial task, as some of the arguments of 
\cite{BBM02,MS02} in the flat Euclidean case do not extend to surfaces.

When the graph~$\partial E$ is~$\alpha$-minimal in the entire space~$\R^{n + 1}$, we establish a better gradient bound. It differs from~\eqref{gradbound} in that a lower power of the oscillation of~$u$ appears on the right-hand side of the estimate---the new power is~$n$. Interestingly, the exact same inequality holds for classical entire minimal graphs, as found by~Bombieri~\&~Giusti~\cite{BG72}. Still, our proof gives a constant that may blow up as~$\alpha \uparrow 1$.

\begin{theorem} \label{globalgradestthm}
Let~$n \ge 1$ and~$\alpha \in (0, 1)$. Let~$E$ be the global subgraph
$$
E = \Big\{ (x', x_{n + 1}) \in \R^{n} \times \R : x_{n + 1} < u(x') \Big\}
$$
of a locally bounded function~$u: \R^n \to \R$, and assume that~$\partial E$ is an~$\alpha$-minimal surface in all of~$\R^{n + 1}$.

Then,~$u$ is of class~$C^\infty$ and there exists a constant~$C$ depending only on~$n$ and~$\alpha$, such that
\begin{equation} \label{bettergradbound}
\| \nabla_{\! x'} u \|_{L^\infty(B'_r)} \le C \left( 1 + \frac{\osc_{B'_r} u}{r} \right)^{n}
\end{equation}
for every~$r > 0$.
\end{theorem}


\subsection{The Dirichlet or Plateau problem}\label{plateau}

The existence of a solution to the fractional Plateau problem in a bounded Lipschitz domain~$\Omega \subset \R^{n + 1}$ (i.e., a nonlocal minimal surface~$\partial E$ in~$\Omega$ with datum prescribed in~$\R^{n + 1} \setminus \Omega$) was established in~\cite{CRS10}. When~$\Omega$ is unbounded, this result can be generalized through a diagonal compactness argument---see~\cite[Corollary~1.10]{L18} for the details. See also~\cite{CSV16} for existence results for nonlocal perimeter-type functionals with other kernels.

In~\cite{CRS10} it is also proved that a minimizer~$E$ of~$\Per_\alpha$ in~$\Omega$ satisfies (in a suitable viscosity sense) the Euler-Lagrange equation
\begin{equation} \label{H=0}
H_\alpha[E](x) = 0
\end{equation}
at any point~$x \in \Omega \cap \partial E$. Here,~$H_\alpha[E](x)$ denotes the so-called~\emph{$\alpha$-mean curvature} (or~\emph{nonlocal mean curvature of order~$\alpha$}) of~$E$ at a point~$x$ of its boundary, and is formally defined by
\begin{equation} \label{Hdef}
H_\alpha[E](x) := \, \PV \int_{\R^{n + 1}} \frac{\chi_{\R^{n + 1} \setminus E}(y) - \chi_{E}(y)}{|x - y|^{n + 1 + \alpha}} \, dy,
\end{equation}
where the integral is meant in the standard Cauchy principal value sense.

When~$\Omega$ is a cylinder of the form~$\Omega' \times \R$, with~$\Omega' \subset \R^n$ smooth and bounded, and the outside datum is the subgraph of a continuous function~$g: \R^n \setminus \Omega' \to \R$, it has been proved in~\cite{DSV16} that minimizers are also subgraphs inside~$\Omega'$. Thus, any minimizer~$E \subset \R^{n + 1}$ is globally the subgraph of a function~$u: \R^n \to \R$ and, as a consequence, its~$\alpha$-mean curvature can be written as
$$
H_\alpha[E](x', u(x')) = 2 \, \mathfrak{H}_\alpha u(x')
$$
for every~$x' \in \R^n$ around which~$u$ is of class~$C^2$. Here,~$\mathfrak{H}_\alpha$ is the operator
\begin{equation} \label{Halphadef}
\mathfrak{H}_\alpha u(x') := \PV \int_{\R^n} G_\alpha \left( \frac{u(x') - u(y')}{|x' - y'|} \right) \frac{dy'}{|x' - y'|^{n + \alpha}},
\end{equation}
with~$G_\alpha(t):= \int_0^t (1 + \tau^2)^{- (n + 1 + \alpha)/2} \, d\tau$. See,~e.g.,~\cite[Section~2]{CV13},~\cite[Section~3]{BFV14}, or~\cite{AV14} for proofs of this and of more general identities. Note that~$\mathfrak{H}_\alpha u(x')$ is well-defined and finite whenever~$u$ is~$C^2$ in a neighborhood of~$x'$. In particular, no growth assumption on~$u$ at infinity is needed, since~$G_\alpha$ is bounded.

As it will be shown in~\cite{CL17}, the subgraph~$E$ of~$u$ is a minimizer of~$\Per_\alpha$ in the cylinder~$\Omega' \times \R$ if and only if~$u$ solves, in an appropriate weak variational sense, the nonlinear equation~$\mathfrak{H}_\alpha u = 0$ in~$\Omega'$. Therefore, the Plateau problem for~$\alpha$-minimal graphs in~$\Omega' \times \R$ is equivalent to the Dirichlet problem
\begin{equation} \label{Dirprob}
\begin{cases}
\mathfrak{H}_\alpha u = 0 & \quad \mbox{in } \Omega',\\
u = g & \quad \mbox{in } \R^n \setminus \Omega',
\end{cases}
\end{equation}
for a given measurable function~$g: \R^n \setminus \Omega' \to \R$.

Problem~\eqref{Dirprob} enjoys existence and uniqueness of solution in a suitable weak variational sense. While the existence follows from~\cite{CRS10}, the uniqueness will be proved in~\cite{CL17} when~$\Omega'$ is a bounded open set with Lipschitz boundary and under the very mild assumption~$g \in L^\infty_\loc(\R^n \setminus \Omega')$---or even milder ones---on the exterior datum. Theorem~\ref{localgradestthm} of the present paper provides full interior regularity for the Dirichlet problem~\eqref{Dirprob}, giving that the solution~$u$ is of class~$C^\infty$ in~$\Omega'$. Moreover, by combining estimate~\eqref{gradbound} with the boundedness results of~\cite[Section~3]{DSV16} (see also~\cite{CL17}), we can control the gradient of~$u$ locally inside~$\Omega'$ by a power of the oscillation of the outside datum~$g$ in a sufficiently large neighborhood of~$\Omega'$. More precisely, we have the bound
\begin{equation} \label{estdatum}
\| \nabla_{\! x'} u \|_{L^\infty(U')} \le \frac{C}{\dist(U', \partial \Omega')^{n + 1 + \alpha}} \left( R_1 + \osc_{B'_{(1 + C) R_1} \setminus \Omega'} g \right)^{n + 1 + \alpha},
\end{equation}
where~$R_1 > 0$ is any radius for which~$\Omega' \subseteq B_{R_1}'$,~$U'$ is any open set compactly contained in~$\Omega'$, and~$C$ is a positive constant depending only on~$n$ and~$\alpha$. For the validity of~\eqref{estdatum}, the fact that the function~$G_\alpha$ in~\eqref{Halphadef} is bounded plays again an important role.

In our work we do not use expression~\eqref{Halphadef} to establish the gradient estimate. Instead, we work with identity~\eqref{Hdef} in~$\R^{n + 1}$ and with the~``ambient'' metric on~$\partial E$ inherited from balls in~$\R^{n + 1}$.

For~$n = 1, 2$ the smoothness of the solution~$u$ to~\eqref{Dirprob} was already known (although without a quantitative estimate) by the results of~\cite{SV13,FV17,DSV16}. In~\cite{DSV16} it was also established that, in any dimension,~$u$ is continuous up to the boundary of~$\Omega'$ from the inside, provided~$g$ is itself continuous up to the boundary of~$\R^n \setminus \Omega'$.
However, as noticed in~\cite{DSV17}, a solution~$u$ is rarely continuous~\emph{across} the boundary---presenting instead the so-called~\emph{boundary stickiness}
phenomenon---and may not have bounded gradient in all of~$\Omega'$, even if~$g$
is smooth and bounded. In this respect, the interior bound of Theorem~\ref{localgradestthm} is optimal---it cannot be extended up to the boundary. In fact, one may have whole vertical portions of the boundary~$\partial E$ of the subgraph of~$u$ lying on~$\partial \Omega' \times \R$---see also the forthcoming~\cite{BL17} for a more radical sticking phenomenon when~$\alpha$ is small and the datum~$g$ is unbounded.

\subsection{A truncated fractional Jacobi operator}\label{jacobi}

Let~$\alpha \in (0, 1)$ and~$\Sigma = \red E$ be the reduced boundary of a subset~$E$ of~$\R^{n + 1}$ with locally finite perimeter (see,~e.g.,~\cite{G84,M12,EG92} for the definition of reduced boundary and its main properties). Let~$\nu_E$ be the unit normal vector to~$\Sigma$ pointing outwards from~$E$. One defines the~\emph{fractional} (or~\emph{nonlocal})~\emph{$\alpha$-Jacobi operator}~$\J_{\Sigma, \alpha}$ at a point~$x \in \Sigma$, acting on a sufficiently smooth and bounded function~$w: \Sigma \to \R$,~by
\begin{equation} \label{Jacdef}
\J_{\Sigma, \alpha} w(x) := \L_{\Sigma, \frac{1 + \alpha}{2}} w(x) + c_{\Sigma, \frac{1 + \alpha}{2}}^2(x) \, w(x),
\end{equation}
where, for~$s \in (0, 1)$, we set
\begin{equation} \label{LSigmadef}
\L_{\Sigma, s} w(x) := \PV \int_{\Sigma} \frac{w(y) - w(x)}{|y - x|^{n + 2 s}} \, d\Haus^n(y)
\end{equation}
and
\begin{equation} \label{c2def}
c^2_{\Sigma, s}(x) := \int_{\Sigma} \frac{\langle \nu_E(x) - \nu_E(y), \nu_E(x) \rangle}{|y - x|^{n + 2 s}} \, d\Haus^n(y).
\end{equation}
See Section~\ref{notsec} for the definition of principal value surface
integrals as the one in~\eqref{LSigmadef}.

For the integrals in~\eqref{Jacdef}-\eqref{c2def} to converge (with~$w$ bounded), one needs an assumption on the behavior of~$\Sigma$ at infinity---e.g., that
\begin{equation} \label{FFMMMfinite}
\int_{\Sigma} \frac{d\Haus^n(y)}{(1 + |y|)^{n + 1 + \alpha}} < +\infty,
\end{equation}
as required in~\cite[Theorem~6.1]{FFMMM15}. We point out that, when~$\Sigma$ is an~$\alpha$-minimal surface in all of~$\R^{n + 1}$, condition~\eqref{FFMMMfinite} is satisfied thanks to a deep result: the perimeter estimate of~\cite{CSV16}---see Lemma~\ref{kerL1overredElem} of Section~\ref{prelisec} below.

The fractional Jacobi operator~$\J_{\Sigma, \alpha}$ was found in~\cite{DDPW14,FFMMM15} while computing the second variation of the fractional perimeter. Under the assumptions that~$\Sigma = \partial E$ is smooth, fulfills~\eqref{FFMMMfinite}, and has zero~$\alpha$-mean curvature inside a bounded open set~$\Omega \subset \R^{n + 1}$, it is proved in~\cite{DDPW14,FFMMM15} that
\begin{equation} \label{Jacsecvar}
\left. \frac{d^2}{dt^2} \Per_{\alpha}(E_{t \psi}; \Omega) \right\rvert_{t = 0} = 2 \int_{\partial E} \psi(x) \left( - \J_{\Sigma, \alpha} \psi(x) \right) d\Haus^n(x),
\end{equation}
for every smooth function~$\psi$ supported inside~$\partial E \cap \Omega$ and with~$\{ E_{t \psi} \}_{t > 0}$ being the family of normal perturbations of~$E$ induced by~$\psi$.

In order to deal with sets that minimize the~$\alpha$-perimeter only inside proper subsets of~$\R^{n + 1}$, as in our Theorem~\ref{localgradestthm}, and not to impose any restriction---such as~\eqref{FFMMMfinite}---on their global geometry or on the exterior datum, in this paper we introduce a truncated version of the Jacobi operator. Given an open set~$\Omega \subseteq \R^{n + 1}$, we define
\begin{equation}
 \label{truncJac}
\J_{\Sigma, \alpha}^\Omega w(x) := \PV \int_{\Sigma \cap \Omega} \frac{w(y) - w(x) + \langle \nu_E(x) - \nu_E(y), \nu_E(x) \rangle \, w(x)}{|y - x|^{n + 1 + \alpha}} \, d\Haus^n(y).
\end{equation}
Observe that for~$\Omega = \R^{n + 1}$ we recover the~$\alpha$-Jacobi operator of~\eqref{Jacdef}-\eqref{c2def}. We also consider the truncated fractional Laplace-type operator
\begin{equation} \label{truncLap}
\L_{\Sigma, \frac{1 + \alpha}{2}}^\Omega w(x) := \PV \int_{\Sigma \cap \Omega} \frac{w(y) - w(x)}{|y - x|^{n + 1 + \alpha}} \, d\Haus^n(y),
\end{equation}
which will be the fundamental object of Section~\ref{weakharsec}, where we establish a weak Harnack inequality for non-negative supersolutions of equations having~\eqref{truncLap} as leading term. In that section, we will consider more general equations involving truncated kernels, not only the geometric ones driven by the Jacobi operator. For this analysis to be carried out, the perimeter estimates of~\cite{CSV16} will be crucial, in particular as they allow the truncated integrals to be well defined. Regardless of this, in Remark~\ref{defJac} we point out that the global Jacobi operator~\eqref{Jacdef}-\eqref{c2def}---with no truncation---can still be defined on hypersurfaces that minimize the~$\alpha$-perimeter only in a proper cylinder (at least when acting on functions supported within the cylinder). This can be done with no restriction on their global geometry or exterior datum.

The next result shows that, as a consequence of the translation invariance of equation~\eqref{H=0}-\eqref{Hdef}, 
the normal~$\nu_E$ to an~$\alpha$-minimal surface~$\partial E$ 
in all of~$\R^{n + 1}$ solves the fractional Jacobi equation~$\J_{\red E, \alpha} \nu_E = 0$. This fact, already observed in~\cite[Appendix~B]{DDPW14} for graphs, will be established in Section~\ref{jacsec} in all details, addressing in particular the convergence of the integrals at infinity. To do this, it will be crucial to use the perimeter estimates of~\cite{CSV16}.

In addition, when~$E$ is a global subgraph but only a minimizer in a vertical cylinder, we establish that the 
last component~$\nu_E^{n + 1}$ is a supersolution of an equation driven by the truncated 
fractional Jacobi operator \eqref{truncJac}, with~$\Omega$ being a cylinder. 
This result is new and has the advantage of requiring the~$\alpha$-minimality of~$E$ just inside a proper subset of~$\R^{n + 1}$---a feature that will be of key importance to obtain Theorem~\ref{localgradestthm} in its full generality and thus to cover virtually every outside datum~$g$ in the Dirichlet problem~\eqref{Dirprob}.

Here are the precise statements of these facts.

\begin{theorem} \label{Jacintrothm}
Let~$n \ge 1$,~$\alpha \in (0, 1)$, and~$E \subset \R^{n + 1}$. The following facts hold true.
\begin{enumerate}[label=$(\roman*)$,leftmargin=*]
\item If~$\partial E$ is an~$\alpha$-minimal surface~in all of~$\R^{n + 1}$, then the fractional Jacobi equation
$$
- \J_{\red E, \alpha} \, \nu_E(\bar{x}) = 0
$$
holds at every point~$\bar{x} \in \partial E$ around which~$\partial E$ is of class~$C^3$.
\item There exists a constant~$C > 0$ depending only on~$n$ and~$\alpha$, such that, if~$E \subset \R^{n + 1}$ is the global subgraph
$$
E = \Big\{ (x', x_{n + 1}) \in \R^n \times \R : x_{n + 1} < u(x') \Big\}
$$
of a measurable function~$u: \R^n \to \R$ and~$\partial E$ is~$\alpha$-minimal in~$B_{2 R}' \times \R$ for some~$R > 0$, then the truncated fractional Jacobi inequality
\begin{equation} \label{Jacle0}
- \J_{\red E, \alpha}^{B_{R}' \times \R} \, \nu_E^{n + 1}(\bar{x}) \ge - \frac{C}{R^{1 + \alpha}} \, \nu_E^{n + 1}(\bar{x})
\end{equation}
holds at every point~$\bar{x} \in \partial E \cap \big(B_{R/2}' \times \R \big)$ around which~$\partial E$ is of class~$C^3$.
\end{enumerate}
\end{theorem}

We stress that the constant~$C$ in~\eqref{Jacle0} does not depend on the regularity of~$\partial E$ around~$\bar{x}$, which is assumed only to perform computations in the proof. 
This assumption will not prevent us from using 
Theorem~\ref{Jacintrothm} to establish the smoothness of nonlocal minimal graphs. 
Indeed, by the results of Savin~\&~Valdinoci~\cite{SV13}---point~(e) in the beginning of the 
introduction---, the set of regular points of~$\partial E$ is ``large'' within~$\partial E$. A capacitary argument then allows us to extend~\eqref{Jacle0} to the whole~$\partial E$ in a weak sense.

Notice that point~$(ii)$ of Theorem~\ref{Jacintrothm} requires~$\partial E$ to be smooth at~$\bar{x}$ as a manifold, but does not impose a priori regularity on the function~$u$ defining~$\partial E$ as the graph~$\{ x_{n + 1} = u(x')\}$. That is,~$|\nabla_{\! x'} u|$ could be infinite at a point~$\bar{x}'$, but~$\partial E$ still be~$C^\infty$ around~$(\bar{x}', u(\bar{x}'))$. Now, it is interesting to realize that our truncated fractional Jacobi inequality~\eqref{Jacle0} provides qualitative information on the slope of~$\partial E$, ruling out the possibility that the normal~$\nu_E$ could be horizontal at a point~$\bar{x} = (\bar{x}', \bar{x}_{n + 1})$ in~$\partial E \cap \big( B_{R/2}' \times \R \big)$---i.e., ruling out that~$|\nabla_{\! x'} u(\bar{x}')| = +\infty$. Indeed, suppose by contradiction that~$\nu_E^{n + 1}(\bar{x}) = 0$. By Theorem~\ref{Jacintrothm}$\,(ii)$ we would have
$$
0 \ge \J_{\red E, \alpha}^{B_R' \times \R} \, \nu_E^{n + 1}(\bar{x}) = \PV \int_{\red E \cap \left( B'_R \times \R \right)} \frac{\nu_E^{n + 1}(y)}{|y - \bar{x}|^{n + 1 + \alpha}} \, d\Haus^n(y).
$$
But since~$E$ is a global subgraph in the vertical direction, we have~$\nu_E^{n + 1}(y) \ge 0$ for every~$y \in \red E$. Therefore, by the above inequality,~$\nu_E^{n + 1} \equiv 0$ on~$\red E \cap \left( B'_R \times \R \right)$, which clearly contradicts the fact that~$\partial E$ is a graph. Notice that, for this argument to work, it is crucial that the right-hand side of~\eqref{Jacle0} is a zeroth order term.

From this remark (already made in~\cite[Theorem~C.2]{DV16} for globally regular subgraphs), we see that Theorem~\ref{Jacintrothm} ``suggests'' the regularity of nonlocal minimal graphs. This is established in the current paper, and its final output---Theorem~\ref{localgradestthm}---is a precise quantitative version of the observation above.

\begin{remark} \label{defJac}
Being a geometric object, one may wonder whether it is possible to define the global Jacobi operator~\eqref{Jacdef}-\eqref{c2def} also on a hypersurface~$\Sigma = \partial E$ that minimizes the~$\alpha$-perimeter only in a proper cylinder~$B_r' \times \R$. The answer is that this can be done, through the divergence theorem, at least when~$\Sigma \cap (\overline{B_r'} \times \R)$ is smooth and bounded, and when the operator acts on smooth functions~$w$ with compact support inside~$\Sigma \cap (B_r' \times \R)$.  Indeed, for such a~$w$, because of a simplification of the terms involving~$w(x)$ in (1.9), one sees that it suffices to give a meaning to the integrals
\begin{equation} \label{inttodefine}
\int_{\Sigma \setminus (B_r' \times \R)} \frac{\nu_E^i(y)}{|y - x|^{n + 2 s}} \, d\Haus^n(y),
\end{equation}
for~$i = 1, \ldots, n + 1$. Observe that the part of the integral over~$\Sigma \cap (B_r' \times \R)$ converges in the standard sense, since~$\Sigma$ is regular inside~$\overline{B_r'} \times \R$. On the other hand, the quantity in~\eqref{inttodefine} can be defined, using formally the divergence theorem, as
$$
- (n + 2 s) \int_{E \setminus (B_r' \times \R)} \frac{y_i - x_i}{|y - x|^{n + 2 + 2s}} \, dy + \int_{\partial (B_r' \times \R) \cap E} \frac{\nu_{B_r' \times \R}^i(y)}{|y - x|^{n + 2 s}} \, d\Haus^n(y).
$$
Note that the two integrals above always converge, no matter how rough~$\Sigma = \partial E$ is.

On the other hand, it is not clear if the global fractional Laplace-type operator~\eqref{LSigmadef} can be understood in a similar way. Except for this remark, in no other place we will need such an argument to define a surface integral---they will be understood in the standard way.
\end{remark}

\subsection{A universal Sobolev inequality on nonlocal minimal surfaces}\label{sobolev}

Miranda obtained in~\cite{M67}, for the first time, a universal Sobolev inequality for~$W^{1, p}$ functions over a minimal surface~$\Sigma$. His inequality is~\emph{universal} in the sense that it does not depend on the geometry or structure of~$\Sigma$---as long as~$\Sigma$ is a minimal surface. To prove it, he took advantage of the isoperimetric inequality for integral currents established by Federer~\& Fleming in~\cite{FF60}. Later on, Allard~\cite{A72} and Michael~\& Simon~\cite{MS73} independently extended the result of Miranda to general hypersurfaces~$\Sigma$ of Euclidean space. Their Sobolev inequality encapsulates the geometry of~$\Sigma$ only through an additional term depending on the mean curvature of~$\Sigma$.

In the next result, we establish a universal Sobolev inequality of fractional order for functions defined on a nonlocal minimal surface~$\Sigma = \partial E$. We point out that, when~$\Sigma$ is the Euclidean space (or an open subset of it), the fractional Sobolev inequality is well-known and several different proofs of it are 
available---see,~e.g.,~\cite{A75,BBM02, MS02,L09,DPV12}.

\begin{theorem} \label{sobineintrothm}
Let~$n \ge 1$,~$\alpha \in (0, 1)$, and~$\partial E$ be an~$\alpha$-minimal surface in all of~$\R^{n + 1}$. Let~$s \in (0, 1)$ and~$p \ge 1$ be such that~$n > s p$.

Then, there exists a constant~$C$ depending only on~$n$,~$\alpha$,~$s$, and~$p$, such that
\begin{equation} \label{sobineintro}
\| v \|_{L^{\frac{n p}{n - s p}}(\red E)} \le C \, [v]_{W^{s, p}(\red E)}
\end{equation}
for every~$v \in W^{s, p}(\red E)$.
\end{theorem}

The quantity~$[\, \cdot \,]_{W^{s, p}(\red E)}$ is the seminorm of the fractional Sobolev space~$W^{s, p}(\red E)$ over~$\red E$---see Section~\ref{notsec} for the precise definition. Note that, in the theorem,~$\alpha$ and~$s$ are arbitrary parameters in~$(0, 1)$---no relation between them is assumed.

To prove Theorem~\ref{sobineintrothm} we do not rely on an isoperimetric-type
inequality, but we follow instead a beautiful proof of the fractional Sobolev inequality in~$\R^n$ that we learned from~H.~Brezis~\cite{B01}. A fractional isoperimetric inequality on~$\Sigma$ can be obtained a posteriori, by applying~\eqref{sobineintro} with~$p = 1$ to characteristic functions---see Corollary~\ref{isopinecor}. The only property of~$\Sigma = \red E$ needed to implement the argument of~\cite{B01} is the lower bound 
\begin{equation} \label{lowerbound}
\Haus^n( \Sigma \cap B_R(x)) \ge c_\star R^n \quad \mbox{for all } R > 0 \mbox{ and } x \in \Sigma,
\end{equation}
for the~$\Haus^n$-measure of~$\Sigma$ inside ambient balls~$B_R(x) := \left\{ y \in \R^{n + 1} : |y - x| < R \right\}$, where~$c_\star > 0$ is a constant. Estimate~\eqref{lowerbound} follows from the density estimates of~\cite{CRS10} and the relative isoperimetric inequality. Hence, we obtain the Sobolev inequality~\eqref{sobineintro} as a particular case of more general results valid on all hypersurfaces of~$\R^{n + 1}$ satisfying~\eqref{lowerbound} or even milder assumptions---see Proposition~\ref{sobineprop1}, and Corollaries~\ref{restrictedsobinecor} and~\ref{isopinecor}. We point out that the same lower bound~\eqref{lowerbound} was assumed in~\cite{H96} to deduce Sobolev inequalities of integer order on general metric measure spaces.

In the proof of our gradient estimate, we will need a localized version of the Sobolev inequality~\eqref{sobineintro}. This will require the use of the perimeter estimate of~\cite{CSV16}---i.e., the reverse inequality in~\eqref{lowerbound}.

Our proof of~\eqref{sobineintro} requires~$E$ to be a minimizer of~$\Per_\alpha$, since we use the density estimates for minimizers to deduce~\eqref{lowerbound}. In a future work, we will establish~\eqref{lowerbound} also for stationary surfaces of the~$\alpha$-perimeter. On the other hand, it would be very interesting to obtain nonlocal versions of the Michael-Simon and Allard inequality---an important open problem.

\subsection{A flatness result for entire nonlocal minimal graphs}

The classification of~$\alpha$-minimal surfaces is a central and challenging problem in the current research on nonlocal PDEs. Recall that~\cite{SV12} established their connection with phase transitions for strongly nonlocal Allen-Cahn energies. Savin \& Valdinoci~\cite{SV13} proved that in~$\R^2$ there exist no global minimizing~$\alpha$-minimal surfaces apart from straight lines. For this, they showed that the only minimizing~$\alpha$-minimal cones in~$\R^2$ are the half-planes. We added here the word~``minimizing'' (not used, but implicit, up to now) to distinguish minimizers from stable~$\alpha$-minimal surfaces or from stationary~$\alpha$-minimal surfaces, described next. The word~``minimizing'' will be added only in this subsection, being implicit elsewhere.

For cones in higher dimensions, no rigidity result nor a counterexample are known at the moment, with the exception of the works~\cite{CV13,CCS17}. Caffarelli \& Valdinoci~\cite{CV13} proved the flatness of minimizing~$\alpha$-minimal cones, when the parameter~$\alpha$ is close to~$1$, in dimension less than or equal to~$7$. 
In~\cite{CCS17}, the first author, Cinti~\&~Serra established the flatness of stable~$\alpha$-minimal cones (and stable~$\alpha$-minimal surfaces, as well) in~$\R^3$ when~$\alpha$ is close to~$1$. Stable~$\alpha$-minimal surfaces are stationary points of the~$\alpha$-perimeter at which~$\Per_\alpha$ has non-negative second variation---as given by~\eqref{Jacsecvar}. Therefore, stable~$\alpha$-minimal surfaces are strong candidates to minimize the~$\alpha$-perimeter.

On the other hand, non-planar stationary~$\alpha$-minimal cones (i.e.,~cones that have zero~$\alpha$-mean curvature, but which are not necessarily minimizers of~$\Per_\alpha$, nor stable) are known to exist in all dimensions, thanks to the work~\cite{DDPW14} by~D\'avila, del Pino~\&~Wei. These are Lawson-type cones and, somewhat surprisingly, they are stable already in dimension~$7$ when~$\alpha$ is sufficiently small. We mention here an important open problem in this direction. The classical Simons cone is a stationary~$\alpha$-minimal surface in~$\R^{2 m}$ for all~$m \ge 1$. However, if it is a minimizer in large enough dimensions, as expected, remains unknown.

In~\cite{FV17}, Figalli~\& Valdinoci obtained a fractional version of a celebrated theorem of~De Giorgi, stating that the non-existence of singular minimizing~$\alpha$-minimal cones in~$\R^n$ yields the validity of a Bernstein-type theorem for~$\alpha$-minimal graphs in one dimension more. By applying this in combination with the results of~\cite{SV13}, they deduced that~$\alpha$-minimal graphs of functions defined in all of~$\R$ or in all of~$\R^2$ are flat. The validity of this result is not known at the moment in any higher dimension. Recall that for classical minimal graphs, Bernstein's theorem holds up to~$\R^7 \times \R = \R^8$, with counterexamples in higher dimensions.

Very recently,~Farina~\& Valdinoci~\cite{FarV17} noted that globally Lipschitz~$\alpha$-minimal graphs are affine in any dimension. We state this result in Theorem~\ref{Liprigthm}, and we give an alternative proof of it carried out independently of their work. In addition, as an application of the gradient estimate of Theorem~\ref{localgradestthm} or Theorem~\ref{globalgradestthm}, we improve~\cite[Theorem~4]{FarV17} by replacing the uniform Lipschitz hypothesis with a linear growth assumption. The precise statement is the following.

\begin{theorem} \label{growthrigthm}
Let~$n \ge 1$ and~$\alpha \in (0, 1)$. Let~$E$ be the global subgraph
$$
E = \Big\{ (x', x_{n + 1}) \in \R^{n} \times \R : x_{n + 1} < u(x') \Big\}
$$
of a measurable function~$u: \R^n \to \R$ satisfying
\begin{equation} \label{growthass}
|u(x')| \le C (1 + |x'|) \quad \mbox{for a.e.~} x' \in \R^n,
\end{equation}
for some constant~$C$. Assume also that~$\partial E$ is an~$\alpha$-minimal surface in all of~$\R^{n + 1}$.

Then,~$u$ is affine, or equivalently~$\partial E$ is a hyperplane.
\end{theorem}

\subsection{Outline of the proof of the gradient estimate. A weak Harnack inequality on nonlocal
minimal surfaces} \label{proofoutlinesub}

The following are the fundamental steps in the proof of our main result, 
Theorem~\ref{localgradestthm}. To simplify the exposition, here we restrict ourselves to 
$\alpha$-minimal graphs~$\Sigma = \partial E$ in all of~$\R^{n + 1}$.

The last component~$\nu_E^{n + 1}$ of the upward pointing normal to~$\Sigma$ is a non-negative function. 
Since the~``nonlocal second fundamental form''~\eqref{c2def} is non-negative as well, we deduce from 
Theorem~\ref{Jacintrothm}$\,(i)$ that~$\nu_E^{n + 1}$ is a non-negative superharmonic function for 
the fractional Laplace-type operator on~$\Sigma$ defined in~\eqref{LSigmadef} with~$s = (1 + \alpha) / 2$.

The natural next step in the proof of Theorem~\ref{localgradestthm} consists in establishing a \emph{weak Harnack inequality} for non-negative supersolutions of linear integral equations on~$\alpha$-minimal surfaces. This is the content of the next result. See Section~\ref{weakharsec} for a much broader statement, valid for a larger class of equations posed on rather general hypersurfaces~$\Sigma \subset \R^{n + 1}$, and which does not assume the smoothness of~$\Sigma$, nor of the supersolution. Here,~$B_R := B_R(0) = \left\{ x \in \R^{n + 1} : |x| < R \right\}$.

\begin{theorem} \label{wHthm}
Let~$n \ge 1$,~$\alpha \in (0, 1)$,~$s \in (1/2, 1)$,~$\Sigma = \partial E$ be a smooth~$\alpha$-minimal surface in all of~$\R^{n + 1}$, and assume that~$0 \in \Sigma$. Let~$\L_{\Sigma, s}$ be as in~\eqref{LSigmadef} and~$w$ be a smooth, bounded, non-negative function on~$\Sigma$ satisfying
$$
- \L_{\Sigma, s} w \ge 0 \quad \mbox{in } \Sigma \cap B_{2}.
$$

Then,
\begin{equation} \label{wHine}
\inf_{\Sigma \cap B_1} w \ge c \left( \int_{\Sigma \cap B_1} w(x) \, d\Haus^n(x) + \int_{\Sigma \setminus B_1} \frac{w(y)}{|y|^{n + 2 s}} \, d\Haus^n(y) \right)
\end{equation}
for some constant~$c \in (0, 1]$ depending only on~$n$,~$\alpha$, and~$s$.
\end{theorem}

Before commenting on Theorem~\ref{wHthm}, we briefly show how the gradient estimate~\eqref{gradbound} of Theorem~\ref{localgradestthm} easily follows from it. Indeed, let~$u$ be the function defining~$\Sigma$ as a graph and pick a point~$\bar{x} \in \Sigma \cap (B_1' \times \R)$. As mentioned at the beginning of this subsection, by Theorem~\ref{Jacintrothm}$\,(i)$, the last component~$\nu_E^{n + 1}$ of the upward pointing normal is a non-negative supersolution of the fractional Laplace-type equation~$- \L_{\Sigma, (1 + \alpha) / 2} w = 0$ on the whole~$\Sigma$. Hence we may apply Theorem~\ref{wHthm} to it. By translating~$\bar{x}$ to the origin, taking advantage of~\eqref{wHine}, and going back to the original coordinates, we get
$$
\nu_E^{n + 1}(\bar{x}) \ge c \int_{\Sigma} \frac{\nu_E^{n + 1}(y)}{1 + |y - \bar{x}|^{n + 1 + \alpha}} \, d\Haus^n(y) \ge c \int_{\Sigma \cap (B'_1 \times \R)} \frac{\nu_E^{n + 1}(y)}{1 + |y - \bar{x}|^{n + 1 + \alpha}} \, d\Haus^n(y).
$$
Set now~$M := \| u \|_{L^\infty(B'_1)}$. In view of the inclusion~$\Sigma \cap (B'_1 \times \R) \subset B_1' \times [-M, M]$ and the fact that~$\nu_E^{n + 1} = (1 + |\nabla_{\! x'} u|^2)^{-1/2}$, we infer from the above inequality that
$$
\frac{1}{\sqrt{1 + |\nabla_{\! x'} u(\bar{x}')|^2}} \ge \frac{\tilde{c}}{(1 + M)^{n + 1 + \alpha}} \int_{\Sigma \cap (B'_1 \times \R)} \nu_E^{n + 1}(y) \, d\Haus^n(y) = \frac{\tilde{c} \, |B_1'|}{(1 + M)^{n + 1 + \alpha}},
$$
with~$\tilde{c} > 0$ depending only on~$n$ and~$\alpha$. Since~$\bar{x}'$ is an arbitrary point of~$B_1'$, this gives estimate~\eqref{gradbound} with~$r = 1$. The case of a general~$r > 0$ follows by scaling.

Of course, to perform the above computation we overlooked several non-trivial details---most importantly, that, according to the statement of Theorem~\ref{localgradestthm}, the set~$\Sigma = \partial E$ is not known to be smooth a priori and that~$E$ is assumed to minimize the~$\alpha$-perimeter only inside a vertical cylinder in~$\R^{n + 1}$.

\begin{remark} \label{s>12rmk}
Theorem~\ref{wHthm} is limited to operators of fractional order~$2 s$ strictly greater than~$1$. 
We are forced to this hypothesis only because of the corresponding assumption~$\beta > 1$ made 
in Lemma~\ref{kerL1overredElem}. As we have already seen, for applications to~$\alpha$-minimal 
surfaces this hypothesis is not restrictive, since we need to take~$s = (1 + \alpha) / 2 > 1/2$. 
Nevertheless, it would be interesting to extend Theorem~\ref{wHthm} 
to operators with~$s \le 1 / 2$, 
possibly requiring stronger assumptions on~$\Sigma$.
\end{remark}

Harnack-type inequalities have been established by many authors for linear and nonlinear singular integral operators in Euclidean space, e.g.,~\cite{L72,BL02,SV04,CS09,Kas07,Kas11,DCKP14,C17}. For second-order elliptic PDEs with bounded measurable coefficients, the Harnack inequality was first established by Moser in his groundbreaking work~\cite{M61}. The idea that the same result could be extended to solutions of equations posed on minimal surfaces is, to the best of our knowledge, due to~Bombieri~\&~Giusti~\cite{BG72}. 
Theorem~\ref{wHthm} here is therefore an extension of their result to singular integral equations on 
nonlocal minimal surfaces.

Our proof of the weak Harnack inequality~\eqref{wHine} is based on a delicate application of the Moser 
iteration technique, implemented along the lines of the works of Kassmann~\cite{Kas09,Kas11}. We stress that 
having~$s > 1/2$ in Theorem~\ref{wHthm} prevents the possibility of obtaining~\eqref{wHine} through 
the simpler method of~\cite{S06,RS16} or~\cite[Section~3]{BV16}, which works well in a wide class of flat Euclidean nonlocal settings. Indeed, in their technique it is crucial for the operator to be bounded when 
applied to smooth barrier functions, a fact that a priori does not necessarily hold for the operator~$\L_{\Sigma, s}$ 
with estimates independent of the geometry of~$\Sigma$ when~$s \ge 1/2$.

To carry out the Moser iteration, two ingredients are needed: that the measure of the underlying space is doubling and a Sobolev-type inequality. In Subsection~\ref{sobolev} we already addressed the validity of a fractional Sobolev inequality on nonlocal minimal surfaces. On the other hand, the doubling property follows from the fact that
\begin{equation} \label{HnAhlreg}
c_\star r^n \le \Haus^n(\Sigma \cap B_r(x)) \le C_\star r^n
\end{equation}
holds for every point~$x \in \Sigma$ and radius~$r > 0$, for some constants~$C_\star \ge c_\star > 0$, when~$\Sigma = \red E$ is the reduced boundary of a minimizer of~$\Per_\alpha$ in all of~$\R^{n + 1}$.

While, as already noted in Subsection~\ref{sobolev}, the lower bound in~\eqref{HnAhlreg} follows from the density estimates of~\cite{CRS10}, the validity of the upper bound was established in the recent paper~\cite{CSV16}. This last result is deep and somewhat surprising, since it provides, for minimizers of the~$\alpha$-perimeter, a bound for their classical perimeter---a higher order quantity. In addition, and remarkably, the result of~\cite{CSV16} is true in any dimension and also for stable nonlocal minimal surfaces, not just for minimizers. However, it is not uniform as~$\alpha \uparrow 1$. In fact, for classical stable minimal surfaces in~$\R^{n + 1}$, the validity of a perimeter bound is known for~$n = 2$, but is still a famous open problem when~$n \ge 3$. Recall that, in light of the works~\cite{BBM01,D02,P04}---see also~\cite{ADM11,CV11}---,~$\alpha$-minimal surfaces converge to classical minimal surfaces as~$\alpha \uparrow 1$.

\subsection{Organization of the paper}

Section~\ref{notsec} specifies some notation that is used throughout the paper.
In Section~\ref{prelisec}, besides collecting several known results needed at later stages,
we obtain estimates for integral quantities defined on hypersurfaces with controlled volume growth and we establish a result on fractional capacities.
Section~\ref{jacsec} deals with fractional Jacobi operators applied to the normal vector; we prove Theorem~\ref{Jacintrothm}.
In Section~\ref{sobinesec} we establish fractional Poincar\'e, Sobolev, and isoperimetric inequalities 
on nonlocal minimal surfaces and on more general subsets of~$\R^{n + 1}$---in particular, we give the
proof of Theorem~\ref{sobineintrothm}.
Section~\ref{weakharsec} concerns the Moser iteration that leads 
to the weak Harnack inequality of Theorem~\ref{wHthm}.
In Section~\ref{regsec} we establish the gradient 
estimates of Theorems~\ref{localgradestthm} and~\ref{globalgradestthm}, while a proof of the Liouville-type 
Theorem~\ref{growthrigthm} is given in Section~\ref{rigsec}.
We end the article with Appendix~\ref{auxapp}, which contains a few technical computations.

\section{Notation} \label{notsec}

\noindent
We will typically work in the Euclidean space~$\R^{n + 1}$, with~$n \ge 1$. Points in~$\R^{n + 1}$ are indicated by~$x, y, z$, while~$x', y', z'$ is reserved for points in~$\R^n$---which will often be identified with the subspace~$\R^n \times \{ 0 \}$ of~$\R^{n + 1}$. This distinction carries over to Euclidean balls. Hence, for~$R > 0$,~$x \in \R^{n + 1}$, and~$x' \in \R^n$ we write
\begin{align*}
B_R(x) & := \left\{ y \in \R^{n + 1} : |y - x| < R \right\}, \\
B'_R(x') & := \left\{ y' \in \R^{n} : |y' - x'| < R \right\},
\end{align*}
as well as~$B_R := B_R(0)$ and~$B_R' := B_R'(0)$.

We indicate as
$$
\C_R(x) := B'_R(x') \times \R
$$
the infinite vertical cylinder of radius~$R > 0$, centered at a point~$x = (x', x_{n + 1})$ of~$\R^{n + 1}$. Again,~$\C_R := \C_R(0)$.

For~$d \ge 0$, the symbol~$\Haus^d$ stands for the~$d$-dimensional Hausdorff measure in both~$\R^n$ and~$\R^{n + 1}$. The Hausdorff measures~$\Haus^{n + 1}$ of~$\R^{n + 1}$ and~$\Haus^n$ of~$\R^n$ are both sometimes denoted by~$|\cdot|$. Inside integrals over~$\Haus^n$-measurable subsets of~$\R^{n + 1}$, we frequently write~$d\sigma$ instead of~$d\Haus^n$.

We now set the notation for fractional Sobolev spaces defined on~$\Haus^n$-measurable subsets of~$\R^{n + 1}$. Let~$p \ge 1$,~$s \in (0, 1)$, and~$\Sigma \subset \R^{n + 1}$ be a set with locally finite~$n$-dimensional Hausdorff measure. Given~$U \subseteq \Sigma$, we say that a function~$v \in L^p(U)$ belongs to~$W^{s, p}(U)$ if and only if
$$
[v]_{W^{s, p}(U)} := \left( \int_{U} \int_{U} \frac{|v(x) - v(y)|^p}{|x - y|^{n + s p}} \, d\sigma(x) \, d\sigma(y) \right)^{\frac{1}{p}} < +\infty.
$$
The Sobolev space~$W^{s, p}(U)$ is then endowed with the norm~$\| \cdot \|_{W^{s, p}(U)}$ defined by the relation~$\| v \|_{W^{s, p}(U)}^p := \| v \|_{L^p(U)}^p + [v]_{W^{s, p}(U)}^p$. When~$p = 2$, the notation~$H^s(U)$ is favored to~$W^{s, 2}(U)$. If~$U$ has finite~$\Haus^n$ measure and~$v \in L^1(U)$, we also write
$$
(v)_{U} := \dashint_{U} v(x) \, d\sigma(x) := \frac{1}{\Haus^n(U)} \int_{U} v(x) \, d\sigma(x).
$$
A similar (standard) terminology is used for the corresponding spaces over open subsets of~$\R^n$ and of~$\R^{n + 1}$.

We will often deal with improper integrals defined over open subsets of~$\R^n$ and~$\R^{n + 1}$, as well as~$\Haus^n$-measurable subsets of~$\R^{n + 1}$. As customary, we will use the symbol~$\PV$ to indicate that an integral has to be understood in the Cauchy principal value sense, defined as follows. Let~$\Omega$ be an open subset of~$\R^{n + 1}$ and~$f: \Omega \to \R$ be a measurable function having an isolated singularity at a point~$x \in \Omega$. We define
$$
\PV \int_{\Omega} f(y) \, dy := \lim_{\delta \rightarrow 0^+} \int_{\Omega \setminus B_\delta(x)} f(y) \, dy,
$$
provided the limit converges. The same definition is adopted for integrals defined on a subset~$\Omega$ of~$\R^n$, integrating this time over~$\Omega \setminus B'_\delta(x')$. Similarly, if~$f$ is defined on a~$\Haus^n$-measurable set~$U \subset \R^{n + 1}$,~$x \in U$, and~$f$ has an isolated singularity at~$x$, we set
$$
\PV \int_{U} f(y) \, d\sigma(y) := \lim_{\delta \rightarrow 0^+} \int_{U \setminus B_\delta(x)} f(y) \, d\sigma(y).
$$
Note that, also in this case, we remove Euclidean balls of~$\R^{n + 1}$ from the domain of integration, prior to taking the limit.

\section{Some preliminary results} \label{prelisec}

\noindent
In this section, we review some known facts about the regularity of nonlocal minimal surfaces, obtain integral bounds to be used throughout the paper, and collect a few results about the fractional Sobolev capacity on general hypersurfaces of Euclidean space. The reader who wants to get to the heart of the matter may skip this part and proceed to Section~\ref{jacsec}.

\subsection{Some known regularity results for nonlocal minimal surfaces} \label{persubsec}

The following theorem provides uniform perimeter and density estimates for nonlocal minimal surfaces. They were proved in~\cite{CSV16} and~\cite{CRS10}, respectively. From now on, by nonlocal minimal surface we mean a minimizer of the fractional perimeter---as in the beginning of the introduction. Note that the symbol~$\red E$ appearing in the statement denotes the~\emph{reduced boundary} of a set~$E$ with finite perimeter (see,~e.g.,~\cite{G84} or~\cite{M12} for more details).

\begin{theorem}[\cite{CRS10,CSV16}] \label{perimeterboundsthm}
Let~$n \ge 1$,~$\alpha \in (0, 1)$, and~$\partial E$ be an~$\alpha$-minimal surface in the ball~$B_R(x)$, for some~$x \in \R^{n + 1}$ and~$R > 0$.

Then,~$E$ has locally finite perimeter inside~$B_R(x)$ and for every~$\bar{\mu} \in (0, 1)$ it holds
\begin{equation} \label{perimeterboundabove}
\hspace{6pt} \Haus^n(\red E \cap B_\rho(x)) = \Per(E; B_\rho(x)) \le C_\star \rho^{n} \quad \mbox{for } \rho \in (0, \bar{\mu} R],
\end{equation}
for some constant~$C_\star$ depending only on~$n$,~$\alpha$, and~$\bar{\mu}$. Moreover, if~$x \in \partial E$, then
\begin{equation} \label{perimeterboundbelow}
\Haus^n(\red E \cap B_\rho(x)) = \Per(E; B_\rho(x)) \ge c_\star \rho^{n} \quad \hspace{2pt} \mbox{for } \rho \in (0, R],
\end{equation}
for some constant~$c_\star > 0$ depending only on~$n$ and~$\alpha$.
\end{theorem}
\begin{proof}
The identity between the perimeter of~$E$ and the Hausdorff measure of~$\red E$ is a standard fact that holds for all sets with finite perimeter in light of~De~Giorgi's structure theorem---see~\cite[Chapter~15]{M12} or~\cite[Chapter~4]{G84}.

When~$\bar{\mu} \le 1/4$, inequality~\eqref{perimeterboundabove} is simply~\cite[Corollary~1.8]{CSV16}. The case~$\bar{\mu} \in (1/4, 1)$ follows using an easy covering argument.

In order to check that~\eqref{perimeterboundbelow} is true as well, we begin by noticing that, by the density estimates of~\cite[Theorem~4.1]{CRS10}, it holds
$$
\min \Big\{ |E \cap B_\rho(x)|, |B_\rho(x) \setminus E| \Big\} \ge c \, \rho^{n + 1}
$$
for some constant~$c > 0$ depending only on~$n$ and~$\alpha$. The lower bound for the perimeter of~$E$ then follows by applying the relative isoperimetric inequality (see,~e.g..~\cite[Corollary~1.29]{G84}).
\end{proof}

The next statement puts together some results from~\cite{CRS10,BFV14,SV13}, establishing the smoothness of nonlocal minimal surfaces in~$\R^{n + 1}$ outside of a set of Hausdorff dimension~$(n + 1) - 3 = n - 2$.

\begin{theorem}[\cite{CRS10,BFV14,SV13}] \label{regsingthm}
Let~$n \ge 1$,~$\alpha \in (0, 1)$,~$\Omega \subseteq \R^{n + 1}$ be an open set, and~$\partial E \subset \R^{n + 1}$ be an~$\alpha$-minimal surface in~$\Omega$.

Then,~$\partial E \cap \Omega$ is of class~$C^\infty$ outside of a closed singular set~$S \subset \partial E \cap \Omega$. The set~$S$ has Hausdorff dimension at most~$(n + 1) - 3 = n - 2$, i.e.,~$S = \varnothing$ if~$n = 1$ and~$\Haus^d(S) = 0$ for every~$d > n - 2$ if~$n \ge 2$.
\end{theorem}

The theorem follows from the results of~\cite{CRS10,BFV14}, where it is shown that the surface~$\partial E$ is smooth outside of a certain singular set~$S \subset \partial E$. By~\cite[Corollary~2]{SV13}, the set~$S$ is empty when~$n = 1$ and has Hausdorff dimension at most~$(n + 1) - 3 = n - 2$ when~$n \ge 2$. 

\subsection{Integral estimates on hypersurfaces} \label{intestsub}

In this subsection, we provide estimates for some integral quantities defined on general hypersurfaces with controlled volume growth.

Take a set~$\Sigma \subset \R^{n + 1}$, a domain~$\Omega \subseteq \R^{n + 1}$, and~$R_0 > 0$. Suppose that~$\Sigma$ satisfies
\begin{equation} \label{Sigmaperbound}
\Haus^n(\Sigma \cap B_\rho(x)) \le C_\star \rho^n \quad \mbox{for every } x \in \Sigma \cap \Omega \mbox{ and } \rho \in (0, R_0],
\end{equation}
for some constant~$C_\star$. Note that this estimate holds in particular when~$\Sigma$ is the (reduced) boundary of a minimizer of the~$\alpha$-perimeter, by inequality~\eqref{perimeterboundabove}. As a consequence, all results in this subsection apply to nonlocal minimal surfaces.

We begin with a lemma that deals with ``contributions coming from far''.

\begin{lemma} \label{kerL1overredElem}
Let~$n \ge 1$,~$\Sigma \subset \R^{n + 1}$,~$\Omega \subseteq \R^{n + 1}$ be a domain, and assume that~\eqref{Sigmaperbound} holds for some positive constants~$R_0$ and~$C_\star$. Let~$\beta > 1$,~$x_0 \in \R^{n + 1}$, and~$r \in (0, R_0 / 2]$.

Then,
$$
\int_{\left( \Sigma \cap \Omega \right) \setminus B_r(x_0)} \frac{d\sigma(y)}{|y - x_0|^{n + \beta}} \le \frac{C}{r^{\beta}}
$$
for some constant~$C$ depending only on~$n$,~$\beta$, and~$C_\star$.
\end{lemma}
\begin{proof}
We write~$\R^{n + 1} \setminus B_r(x_0) = \cup_{j = 1}^{+\infty} A_j$, where~$A_j := B_{(j + 1) r}(x_0) \setminus B_{j r}(x_0)$.

Note that each annulus~$A_j$ can be covered by a collection~$\B_j$ of at most~$c_n j^{(n + 1) - 1} = c_n j^n$ balls of radius~$r$, for some dimensional constant~$c_n > 0$ (recall that we are in~$\R^{n + 1}$). Considering only those balls that intersect~$\Sigma \cap \Omega$ and doubling their radius to be~$2 r$, we may assume that each of them is centered at some point in~$\Sigma \cap \Omega$. Thanks to this construction and~\eqref{Sigmaperbound}, we conclude that
\begin{align*}
\int_{\left( \Sigma \cap \Omega \right) \setminus B_r(x_0)} \frac{d\sigma(y)}{|y - x_0|^{n + \beta}} & \le \sum_{j = 1}^{+\infty} \int_{\Sigma \cap \Omega \cap A_j} \frac{d\sigma(y)}{(j r)^{n + \beta}} \le \sum_{j = 1}^{+\infty} \Bigg\{ \frac{1}{(j r)^{n + \beta}} \sum_{B \in \B_j} \Haus^n(\Sigma \cap B) \Bigg\} \\
& \le \frac{2^n C_\star}{r^{\beta}} \sum_{j = 1}^{+\infty} \frac{\# \B_j}{j^{n + \beta}} \le \frac{2^n c_n C_\star}{r^{\beta}} \sum_{j = 1}^{+\infty} \frac{1}{j^{\beta}} \le \frac{C}{r^\beta},
\end{align*}
since~$\beta > 1$, for some constant~$C$ depending only on~$n$,~$\beta$, and~$C_\star$.
\end{proof}

Next, we provide an estimate for interactions at small scales.

\begin{lemma} \label{kerdesingboundlem}
Let~$n \ge 1$,~$\Sigma \subset \R^{n + 1}$,~$\Omega \subseteq \R^{n + 1}$ be a domain, and assume that~\eqref{Sigmaperbound} holds for some positive constants~$R_0$ and~$C_\star$. Let~$\gamma > 0$,~$x_0 \in \Sigma \cap \Omega$, and~$r \in (0, R_0]$.

Then,
\begin{equation} \label{kerbound2}
\int_{\Sigma \cap \Omega \cap B_r(x_0)} \frac{d\sigma(y)}{|y - x_0|^{n - \gamma}} \le C \, r^{\gamma}
\end{equation}
for some constant~$C$ depending only on~$n$,~$\gamma$, and~$C_\star$.
\end{lemma}
\begin{proof}
When~$\gamma \ge n$, estimate~\eqref{kerbound2} is an immediate consequence of assumption~\eqref{Sigmaperbound}. We thus assume that~$\gamma \in (0, n)$ and consider, for every positive integer~$j$, the annulus~$A_j := B_{2^{1 - j} r}(x_0) \setminus B_{2^{- j} r}(x_0)$. By virtue of~\eqref{Sigmaperbound}, we have
\begin{align*}
\int_{\Sigma \cap \Omega \cap B_r(x_0)} \frac{d\sigma(y)}{|y - x_0|^{n - \gamma}} & \le \sum_{j = 1}^{+\infty} \int_{\Sigma \cap \Omega \cap A_j} \frac{d\sigma(y)}{(2^{- j} r)^{n - \gamma}} \\
& \le \sum_{j = 1}^{+\infty} \left( \frac{2^{j}}{r} \right)^{n - \gamma} \Haus^n(\Sigma \cap B_{2^{1 - j} r}(x_0)) \le C_\star 2^n \, r^{\gamma} \sum_{j = 1}^{+\infty} 2^{- \gamma j},
\end{align*}
which clearly yields~\eqref{kerbound2}, as~$\gamma > 0$.
\end{proof}

The last two lemmas immediately yield the following result.

\begin{corollary} \label{etasemilem}
Let~$n \ge 1$,~$\Sigma \subset \R^{n + 1}$,~$\Omega \subseteq \R^{n + 1}$ be a domain, and assume that~\eqref{Sigmaperbound} holds for some positive constants~$R_0$ and~$C_\star$. Let~$x_0 \in \Sigma \cap \Omega$,~$r \in (0, R_0/2]$,~$p \ge 1$, and~$s \in (0, 1)$ be such that~$sp > 1$.

Then, for every~$\eta \in W^{1, \infty}(\R^{n + 1})$ satisfying
$$
|\eta| \le C_\bullet \quad \mbox{ and } \quad |\nabla \eta| \le \frac{C_\bullet}{r} \quad \mbox{in } \R^{n + 1}
$$
for some constant~$C_\bullet$, it holds
$$
\int_{\Sigma \cap \Omega} \frac{|\eta(y) - \eta(x_0)|^p}{|y - x_0|^{n + s p}} \, d\sigma(y) \le \frac{C}{r^{s p}}
$$
for some constant~$C$ depending only on~$n$,~$s$,~$p$,~$C_\star$, and~$C_\bullet$.
\end{corollary}
\begin{proof}
We have
\begin{align*}
\int_{\Sigma \cap \Omega} \frac{|\eta(y) - \eta(x_0)|^p}{|y - x_0|^{n + s p}} \, d\sigma(y) & \le \frac{C_\bullet^p}{r^p} \int_{\Sigma \cap \Omega \cap B_r(x_0)} \frac{d\sigma(y)}{|y - x_0|^{n - (1 - s) p}} \\
& \quad + 2^p C_\bullet^p \int_{\left( \Sigma \cap \Omega \right) \setminus B_r(x_0)} \frac{ d\sigma(y)}{|y - x_0|^{n + s p}}.
\end{align*}
Applying Lemmas~\ref{kerL1overredElem} and~\ref{kerdesingboundlem} (with~$\beta = s p > 1$ and~$\gamma = (1 - s) p > 0$), we conclude the result.
\end{proof}

\subsection{Fractional capacities on hypersurfaces} \label{capsubsec}

We collect here a few facts on the fractional Sobolev capacity over rather general~$\Haus^n$-measurable subsets~$\Sigma$ of~$\R^{n + 1}$ and in particular over nonlocal minimal surfaces. Although most of the results presented here can probably be deduced from more general theories (see,~e.g.,~the recent~\cite{N16}), for the convenience of the reader we include all the proofs.

The next definition introduces the concept of~$s$-capacity on~$\Sigma$. In our future applications to~$\alpha$-minimal surfaces, we will only need~$s = (1 + \alpha) / 2$, which is always greater than~$1/2$. For this reason, and also since we use it in some proofs, we restrict ourselves to~$s \in (1/2, 1)$.

\begin{definition} \label{capdef1}
Given~$n \ge 1$,~$s \in (1/2, 1)$, a set~$\Sigma \subset \R^{n + 1}$ with locally finite~$\Haus^n$-measure, and~$A \subseteq \Sigma$, we define
\begin{align*}
\Cap_{\, \Sigma, s}(A) := \inf \Big\{ \| v \|_{H^s(\Sigma)}^2 & : v \in H^s(\Sigma) \mbox{ and } v \ge 1 \mbox{ $\Haus^n$-a.e.} \\[-5pt]
& \hspace{11pt} \mbox{in an open neighborhood of } A \Big\} \\[-2pt]
= \inf \Big\{ \| v \|_{H^s(\Sigma)}^2 & : v \in H^s(\Sigma), \, 0 \le v \le 1 \mbox{ $\Haus^n$-a.e.~in } \Sigma, \mbox{ and }\\[-5pt]
& \hspace{11pt} v = 1 \mbox{~$\Haus^n$-a.e.~in an open neighborhood of } A \Big\}.
\end{align*}
We call this quantity the~\emph{$s$-fractional capacity} of~$A$ on the set~$\Sigma$.
\end{definition}

Recall that~$\| v \|_{H^s(\Sigma)}^2 = \| v \|_{L^2(\Sigma)}^2 + [v]_{H^s(\Sigma)}^2$ and see Section~\ref{notsec} for the definition of the~$H^s = W^{s, 2}$ seminorm. Also, the second characterization for~$\Cap_{\, \Sigma, s}$ in Definition~\ref{capdef1} holds since, for every~$v \in H^s(\Sigma)$, we have~$\left\| \min \left\{ \max \{ v, 0 \}, 1 \right\} \right\|_{L^2(\Sigma)} \le \| v \|_{L^2(\Sigma)}$ and~$\left[ \min \left\{ \max \{ v, 0 \}, 1 \right\} \right]_{H^s(\Sigma)} \le [ v ]_{H^s(\Sigma)}$.

By Theorem~\ref{regsingthm}, we know that the singular set of a nonlocal minimal surface in~$\R^{n + 1}$ has Hausdorff dimension at most~$n - 2$. The following proposition---which is the main result of the subsection---provides a capacitary description of this fact.

\begin{proposition} \label{capsingthm}
Let~$n \ge 1$,~$\alpha \in (0, 1)$,~$\Omega \subseteq \R^{n + 1}$ be an open set, and~$\partial E \subset \R^{n + 1}$ be an~$\alpha$-minimal surface in~$\Omega$.

Then, the singular set~$S$ of~$\partial E$ in~$\Omega$ satisfies
$$
\Cap_{\, \Sigma \cap \Omega, s}(S) = 0 \quad \mbox{for every } s \in (1/2, 1).
$$
\end{proposition}

Proposition~\ref{capsingthm} follows immediately from Theorem~\ref{regsingthm} and the following result, which explores the relationship between the~$s$-fractional capacity and the Hausdorff measure on general~$\Haus^n$-measurable subsets~$\Sigma$ of~$\R^{n + 1}$ satisfying
\begin{equation} \label{densestforcap}
\Haus^n ( \Sigma \cap B_\rho(x) ) \le C_\star \rho^n,
\end{equation}
for some constant~$C_\star$. Notice that for nonlocal minimal surfaces this condition is warranted by Theorem~\ref{perimeterboundsthm}. For the next result, we restrict ourselves to~$n \ge 2$---observe that, when~$n = 1$, Proposition~\ref{capsingthm} is an immediate consequence of Theorem~\ref{regsingthm}, as in this case~$S = \varnothing$.

\begin{proposition} \label{caphausprop}
Let~$n \ge 2$,~$s \in (1/2, 1)$,~$\Sigma \subset \R^{n + 1}$ be a set with locally finite~$\Haus^n$-measure, and~$\Omega \subseteq \R^{n + 1}$ be an open set. Suppose that~\eqref{densestforcap} holds for every~$x \in \Sigma \cap \Omega$ and~$\rho \in (0, R]$, for some positive constants~$R$ and~$C_\star$.

Then, there exists a constant~$C$ depending only on~$n$,~$s$, and~$C_\star$, such that
$$
\Cap_{\, \Sigma \cap \Omega, s}(A) \le C \, \Haus^{n - 2 s}(A)
$$
for every set~$A \subseteq \Sigma \cap \Omega$.
\end{proposition}

To prove Proposition~\ref{caphausprop} we need some preliminary results.

Let~$\Sigma \subset \R^{n + 1}$ be a set with locally finite~$\Haus^n$-measure. It is obvious that~$\Cap_{\, \Sigma, s}(A) \le \Cap_{\, \Sigma, s}(B)$ if~$A \subseteq B$ are subsets of~$\Sigma$. The~$s$-fractional capacity is also countably subadditive, i.e., if~$\{ A_i \}_{i \in \N}$ are subsets of~$\Sigma$, then
\begin{equation} \label{subaddcap}
\Cap_{\, \Sigma, s} \! \left( \bigcup_{i = 1}^{+\infty} A_i \right) \le \sum_{i = 1}^{+\infty} \Cap_{\, \Sigma, s}(A_i).
\end{equation}
To verify this, we may assume that the right-hand side of~\eqref{subaddcap} is finite. Let~$\varepsilon > 0$ and~$\{ v_i \} \subset H^s(\Sigma)$ be a sequence of functions satisfying~$v_i \ge 1$ in an open neighborhood of~$A_i$ and~$\| v_i \|_{H^s(\Sigma)}^2 \le \Cap_{\, \Sigma, s}(A_i) + 2^{-i} \varepsilon$ for every~$i \in \N$. Since it can be easily seen that
$$
\| \max \{ u, w \} \|_{H^s(\Sigma)}^2 \le \| u \|_{H^s(\Sigma)}^2 + \| w \|_{H^s(\Sigma)}^2
$$
for any two~$u, w \in H^s(\Sigma)$, considering~$v := \sup_{i \in \N} v_i$, iterating the previous inequality, and applying Fatou's lemma, we obtain
$$
\| v \|_{H^s(\Sigma)}^2 \le \sum_{i = 1}^{+\infty} \| v_i \|_{H^s(\Sigma)}^2 \le \varepsilon + \sum_{i = 1}^{+\infty} \Cap_{\, \Sigma, s}(A_i).
$$
The proof of~\eqref{subaddcap} is then finished, as~$v \ge 1$ in an open neighborhood of~$\cup_{i = 1}^{+\infty} A_i$ and~$\varepsilon > 0$ can be chosen arbitrarily small.

The following is another ingredient towards the proof of Proposition~\ref{caphausprop}. It provides a sharp upper bound for the~$s$-fractional capacity of ambient balls.

\begin{lemma} \label{cap2lem}
Let~$n \ge 1$,~$s \in (1/2, 1)$,~$\Sigma$ be a subset of~$\R^{n + 1}$ having locally finite~$\Haus^n$-measure,~$\Omega \subseteq \R^{n + 1}$ be an open set,~$x_0 \in \Sigma \cap \Omega$, and~$R > 0$. Assume that~\eqref{densestforcap} holds for every point~$x \in \Sigma \cap \Omega$ and radius~$\rho \in (0, 2 R]$, for some constant~$C_\star$.

Then, there exists a constant~$C$ depending only on~$n$,~$s$, and~$C_\star$, such that
\begin{equation} \label{Capest}
\Cap_{\, \Sigma \cap \Omega, s} \left( \Sigma \cap \Omega \cap B_r(x_0) \right) \le C r^{n - 2 s} \left( 1 + r^{2 s} \right)
\end{equation}
for every~$r \in (0, R)$.
\end{lemma}
\begin{proof}
Take a smooth function~$v \in C^\infty(\R^{n + 1})$ satisfying~$\supp(v) \subset B_{2 r}(x_0)$,~$0 \le v \le 1$ in~$\R^{n + 1}$,~$v = 1$ in~$B_{3r/2}(x_0)$, and~$|\nabla v| \le 4 / r$ in~$\R^{n + 1}$. By~\eqref{densestforcap} and Corollary~\ref{etasemilem} we have
\begin{align*}
[v]_{H^s(\Sigma \cap \Omega)}^2 & \le 2 \int_{\Sigma \cap B_{2r}(x_0)} \left( \int_{\Sigma \cap \Omega} \frac{|v(y) - v(x)|^2}{|y - x|^{n + 2 s}} \, d\sigma(y) \right) d\sigma(x) \\
& \le \frac{C}{r^{2 s}} \, \Haus^n(\Sigma \cap B_{2r}(x_0)) \le C r^{n - 2 s}
\end{align*}
for some~$C$ depending only on~$n$,~$s$, and~$C_\star$. On the other hand,~\eqref{densestforcap} also yields
$$
\| v \|_{L^2(\Sigma \cap \Omega)}^2 = \int_{\Sigma \cap B_{2 r}(x_0)} |v(x)|^2 \, d\sigma(x) \le \Haus^n(\Sigma \cap B_{2 r}(x_0)) \le 2^n C_\star r^n.
$$
From these bounds, estimate~\eqref{Capest} follows at once.
\end{proof}

Thanks to these results, we are now ready to address the proof of Proposition~\ref{caphausprop}.

\begin{proof}[Proof of Proposition~\ref{caphausprop}]
Note that~$n \ge 2 > 2 s$ and that we may assume~$\Haus^{n - 2 s}(A)$ to be finite. Take~$\delta \in \left( 0, \min \{ 1, R/4 \} \right]$ and consider a countable covering~$\{ D_j \}_{j \in \N}$ of~$A$ by sets~$D_j \subset \R^{n + 1}$ having diameter~$d_j := \diam(D_j) \le \delta$. Assuming without loss of generality that all~$D_j$'s intersect~$A$, we deduce the existence of a new countable cover~$\{ B_{d_j}(x_j) \}_{j \in \N}$ of~$A$ made up of balls centered at points~$x_j \in A$. By taking advantage of~\eqref{subaddcap} and Lemma~\ref{cap2lem}, we have
$$
\Cap_{\, \Sigma \cap \Omega, s}(A) \le \sum_{j = 1}^{+\infty} \Cap_{\, \Sigma \cap \Omega, s} \left( \Sigma \cap \Omega \cap B_{d_j}(x_j) \right) \le C \sum_{j = 1}^{+\infty} d_j^{\, n - 2 s},
$$
for some constant~$C$ depending only on~$n$,~$s$, and~$C_\star$. The conclusion follows by applying the definition of Hausdorff measure, as~$\delta$ can be taken arbitrarily small.
\end{proof}

\section{The normal vector field and the fractional Jacobi operator} \label{jacsec}

\noindent
In this section we establish Theorem~\ref{Jacintrothm}. Part~$(i)$ has been stated within the proof of Corollary~B.1 of~\cite{DDPW14}. Here, we will give a proof of it, addressing in particular integrability issues in all details. On the other hand, point~$(ii)$ of our theorem is new and will play a fundamental role in the sequel.

We begin with a preliminary result concerning the derivative of the~$\alpha$-mean curvature of a set~$E$ along tangential directions to~$\partial E$. This is identity~\eqref{derH} below, which was already obtained in~\cite[Proposition~2.1]{CFSW}. For completeness, we include a slightly shorter proof of it below. From it, we deduce~\eqref{derH2}, which will be later used with~$\Omega$ being either the whole space~$\R^{n + 1}$ or a vertical cylinder. The sets~$E^{(t)}$, for~$t = 0, 1$, appearing in the statement denote the points of density~$t$ of~$E$, defined as
\begin{equation} \label{densitypoints}
E^{(t)} := \left\{ x \in \R^{n + 1} : \lim_{r \rightarrow 0^+} \frac{|E \cap B_r(x)|}{|B_r|} = t \right\};
\end{equation}
see~\cite{M12} for more details.

\begin{proposition} \label{derHprop}
Let~$n \ge 1$ and~$\alpha \in (0, 1)$. Let~$E$ be a measurable subset of~$\R^{n + 1}$ and assume that~$\partial E$ is of class~$C^{2, \beta}$ in a neighborhood of a point~$\bar{x} \in \partial E$, for some~$\beta > \alpha$.

Then,
\begin{equation} \label{derH}
\partial_v H_\alpha[E](\bar{x}) = - (n + 1 + \alpha) \, \PV \int_{\R^{n + 1}} \frac{\chi_{\R^{n + 1} \setminus E}(y) - \chi_{E}(y)}{|\bar{x} - y|^{n + 3 + \alpha}} \left( (\bar{x} - y) \cdot v \right) dy
\end{equation}
for every direction~$v \in \R^{n + 1}$ orthogonal to~$\nu_E(\bar{x})$.

Furthermore, let~$\Omega \subseteq \R^{n + 1}$ be an open set with locally Lipschitz boundary and such that~$\bar{x} \in \Omega$. Assume that~$E$ has locally finite perimeter in~$\overline{\Omega}$, that~$\Haus^n(\red E \cap \partial \Omega) = 0$, and that both~$\int_{\red E \cap \Omega} (1 + |y|)^{- n - 1 - \alpha} \, d\sigma(y)$ and~$\int_{\partial \Omega} (1 + |y|)^{- n - 1 - \alpha} \, d\sigma(y)$ are finite.

Then,
\begin{equation} \label{derH2}
\begin{aligned}
\partial_v H_\alpha[E](\bar{x}) & = 2 \, \PV \int_{\red E \cap \Omega} \frac{\left( \nu_E(y) - \nu_E(\bar{x}) \right) \cdot v}{|\bar{x} - y|^{n + 1 + \alpha}} \, d\sigma(y) \\
& \quad + \int_{\partial \Omega} \frac{\chi_{E^{(1)}}(y) - \chi_{E^{(0)}}(y)}{|\bar{x} - y|^{n + 1 + \alpha}} \left( \nu_\Omega(y) \cdot v \right) d\sigma(y) \\
& \quad - (n + 1 + \alpha) \int_{\R^{n + 1} \setminus \Omega} \frac{\chi_{\R^{n + 1} \setminus E}(y) - \chi_{E}(y)}{|\bar{x} - y|^{n + 3 + \alpha}} \left( (\bar{x} - y) \cdot v \right) dy
\end{aligned}
\end{equation}
for every direction~$v \in \R^{n + 1}$ orthogonal to~$\nu_E(\bar{x})$, where~$E^{(1)}$ and~$E^{(0)}$ are given by~\eqref{densitypoints}.
\end{proposition}

\begin{proof}
Up to a translation and a rotation, we may assume that~$\bar{x} = 0$ and~$\nu_E(0) = e_{n + 1}$. In this setting,~\eqref{derH} and~\eqref{derH2} are respectively equivalent to
\begin{equation} \label{derHbis}
\partial_{x_i} H_\alpha[E](0) = (n + 1 + \alpha) \, \PV \int_{\R^{n + 1}} \frac{\chi_{\R^{n + 1} \setminus E}(y) - \chi_{E}(y)}{|y|^{n + 3 + \alpha}} y_i \, dy
\end{equation}
and
\begin{equation} \label{derH2bis}
\begin{aligned}
\partial_{x_i} H_\alpha[E](0) & = 2 \, \PV \int_{\red E \cap \Omega} \frac{\nu_E^i(y)}{|y|^{n + 1 + \alpha}} \, d\sigma(y) \\
& \quad + \int_{\partial \Omega} \frac{\chi_{E^{(1)}}(y) - \chi_{E^{(0)}}(y)}{|y|^{n + 1 + \alpha}} \, \nu^i_\Omega(y) \, d\sigma(y) \\
& \quad + (n + 1 + \alpha) \int_{\R^{n + 1} \setminus \Omega} \frac{\chi_{\R^{n + 1} \setminus E}(y) - \chi_{E}(y)}{|y|^{n + 3 + \alpha}} \, y_i \, dy,
\end{aligned}
\end{equation}
for every~$i = 1, \ldots, n$.

We begin by establishing~\eqref{derHbis}. Take~$\varepsilon > 0$ small enough to have that
$$
E \cap B_\varepsilon = \Big\{ (x', x_{n + 1}) \in B_\varepsilon : x_{n + 1} < v(x') \Big\}
$$
for some function~$v \in C^{2, \beta}(\R^n)$ satisfying~$v(0) = 0$ and~$\nabla_{\! x'} v(0) = 0$. Let~$\eta \in C^\infty(\R^{n + 1})$ be a radially symmetric non-decreasing function, such that~$\eta = 0$ in~$B_1$,~$\eta = 1$ outside of~$B_2$, and~$| \nabla \eta | \le 2$ in~$\R^{n + 1}$. For~$\delta \in (0, \varepsilon/4)$, we consider~$\eta^{(\delta)}(x) := \eta(x / \delta)$ and write
$$
H_\alpha[E](x', v(x')) = \Psi^{(\delta)}_1(x') + \Psi^{(\delta)}_2(x'),
$$
with
\begin{align*}
\Psi^{(\delta)}_1(x') & := \int_{\R^{n + 1}} \frac{\chi_{\R^{n + 1} \setminus E}(y) - \chi_E(y)}{|(x', v(x')) - y|^{n + 1 + \alpha}} \, \eta^{(\delta)}((x', v(x')) - y) \, dy,\\
\Psi^{(\delta)}_2(x') & := \PV \int_{\R^{n + 1}} \frac{\chi_{\R^{n + 1} \setminus E}(y) - \chi_E(y)}{|(x', v(x')) - y|^{n + 1 + \alpha}} \left( 1 - \eta^{(\delta)}((x', v(x')) - y) \right) dy.
\end{align*}

First, we claim that
\begin{equation} \label{limPsi2=0}
\lim_{\delta \rightarrow 0^+} \partial_{x_i} \! \Psi_2^{(\delta)}(0) = 0.
\end{equation}
To check~\eqref{limPsi2=0}, let~$\Pi^-_{x'}$ be the lower half-space bounded by the tangent hyperplane to~$\partial E$ at~$(x', v(x'))$, i.e.,~$\Pi^-_{x'} := \left\{ (y', y_{n + 1}) \in \R^n \times \R : y_{n + 1} < v(x') + \langle \nabla_{\! x'} v(x'), y' - x' \rangle \right\}$.
By symmetry,
$$
\PV \int_{\R^{n + 1}} \frac{\chi_{\R^{n + 1} \setminus \Pi_{x'}^-}(y) - \chi_{\Pi^-_{x'}}(y)}{|(x', v(x')) - y|^{n + 1 + \alpha}} \left( 1 - \eta^{(\delta)}((x', v(x')) - y) \right) dy = 0.
$$
Hence, we may subtract this term from~$\Psi_2^{(\delta)}(x')$ and write, for~$x' \in B_\delta'$,
$$
\Psi_2^{(\delta)}(x') = - 2 \int_{B'_{4 \delta}} \int_{v(x') + \langle \nabla_{\! x'} v(x'), \, y' - x' \rangle}^{v(y')} \frac{1 - \eta^{(\delta)}((x', v(x')) - y)}{|(x', v(x')) - y|^{n + 1 + \alpha}} \, dy_{n + 1} dy'.
$$
We differentiate the above expression with respect to~$x_i$ and evaluate at~$0$. We get
\begin{align*}
\partial_{x_i} \! \Psi_2^{(\delta)}(0) & = - 2 \int_{B'_{4 \delta}} \int_{0}^{v(y')} \frac{(n + 1 + \alpha) \left( 1 - \eta^{(\delta)}(y) \right) y_i + |y|^2 \eta^{(\delta)}_{x_i}(y)}{|y|^{n + 3 + \alpha}} \, dy_{n + 1} dy' \\
& \quad + 2 \int_{B'_{4 \delta}} \frac{\left( 1 - \eta^{(\delta)}(y', 0) \right) \langle D_{x'}^2 v(0) e_i, y' \rangle}{|y'|^{n + 1 + \alpha}} \, dy'.
\end{align*}
By the symmetry properties of~$y_i$,~$\eta^{(\delta)}$, and~$\eta^{(\delta)}_{x_i}$, the last integral vanishes and also
$$
\int_{B'_{4 \delta}} \int_{0}^{\langle D_{x'}^2 v(0) y', y' \rangle / 2} \frac{ (n + 1 + \alpha) \left( 1 - \eta^{(\delta)}(y) \right) y_i + |y|^2 \eta_{x_i}^{(\delta)}(y)}{|y|^{n + 3 + \alpha}} \, dy_{n + 1} dy' = 0.
$$
Accordingly,
\begin{align*}
\partial_{x_i} \! \Psi_2^{(\delta)}(0) & = - 2 \int_{B'_{4 \delta}} \int_{\langle D_{x'}^2 v(0) y', y' \rangle / 2}^{v(y')} \frac{(n + 1 + \alpha) \left( 1 - \eta^{(\delta)}(y) \right) y_i + |y|^2 \eta^{(\delta)}_{x_i}(y)}{|y|^{n + 3 + \alpha}} \, dy_{n + 1} dy',
\end{align*}
and, using that~$v$ is~$C^{2, \beta}$, we estimate
\begin{align*}
\left| \partial_{x_i} \! \Psi_2^{(\delta)}(0) \right| & \le 2 \int_{B'_{4 \delta}} \left| v(y') - \frac{\langle D_{x'}^2 v(0) y', y' \rangle}{2} \right| \left( \frac{n + 2}{|y'|^{n + 2 + \alpha}} + \frac{2}{\delta} \frac{1}{|y'|^{n + 1 + \alpha}} \right) dy' \\
& \le C \int_{B'_{4 \delta}} \frac{dy'}{|y'|^{n - \beta + \alpha}} \le C \delta^{\beta - \alpha},
\end{align*}
where, from now on,~$C$ denotes constants independent of~$\delta$. Claim~\eqref{limPsi2=0} follows.

In view of~\eqref{limPsi2=0}, to obtain~\eqref{derHbis} we only need to show that
\begin{equation} \label{derHtris}
\lim_{\delta \rightarrow 0^+} \partial_{x_i} \! \Psi_1^{(\delta)}(0) = (n + 1 + \alpha) \, \PV \int_{\R^{n + 1}} \frac{\chi_{\R^{n + 1} \setminus E}(y) - \chi_{E}(y)}{|y|^{n + 3 + \alpha}} y_i \, dy.
\end{equation}
A straightforward computation reveals that
$$
\partial_{x_i} \! \Psi_1^{(\delta)}(0) = I_1^{(\delta)} - I_2^{(\delta)},
$$
with
\begin{align*}
I_1^{(\delta)} & := (n + 1 + \alpha) \int_{\R^{n + 1}} \frac{\chi_{\R^{n + 1} \setminus E}(y) - \chi_E(y)}{|y|^{n + 3 + \alpha}} \, y_i \eta^{(\delta)}(y) \, dy,\\
I_2^{(\delta)} & := \int_{\R^{n + 1}} \frac{\chi_{\R^{n + 1} \setminus E}(y) - \chi_E(y)}{|y|^{n + 1 + \alpha}} \, \eta^{(\delta)}_{x_i}(y) \, dy.
\end{align*}
On the one hand, the integrand in~$I_2^{(\delta)}$ with~$E$ replaced by~$\left\{ y_{n + 1} < \langle D^2_{x'}v(0) y', y' \rangle / 2 \right\}$ has zero integral over~$\R^{n + 1}$, by oddness of~$\eta^{(\delta)}_{x_i}$. Subtracting this integral from~$I_2^{(\delta)}$, we see that
\begin{equation} \label{I2deltaequiv}
I_2^{(\delta)} = - 2 \int_{B_{2 \delta}'} \int_{\langle D_{x'}^2 v(0) y', y' \rangle / 2}^{v(y')} \frac{\eta_{x_i}^{(\delta)}(y)}{|y|^{n + 1 + \alpha}} \, dy_{n + 1} dy'.
\end{equation}
Thus, by the~$C^{2, \beta}$ regularity of~$v$,
$$
\left| I_2^{(\delta)} \right| \le \frac{4}{\delta} \int_{B_{2 \delta}'} \left| v(y') - \frac{\langle D_{x'}^2 v(0) y', y' \rangle}{2} \right| \frac{dy'}{|y'|^{n + 1 + \alpha}} \le \frac{C}{\delta} \int_{B_{2 \delta}'} \frac{dy'}{|y'|^{n - 1 - \beta + \alpha}} \le C \delta^{\beta - \alpha}.
$$
On the other hand, a similar argument yields that
$$
\left| I_1^{(\delta)} - (n + 1 + \alpha) \int_{\R^{n + 1} \setminus B_{2 \delta}} \frac{\chi_{\R^{n + 1} \setminus E}(y) - \chi_E(y)}{|y|^{n + 3 + \alpha}} y_i \, dy \, \right| \le C \delta^{\beta - \alpha}.
$$
The combination of the last two estimates immediately leads us to~\eqref{derHtris}.

We now move to the proof of~\eqref{derH2bis}. Write
$$
J^{(\delta)} := (n + 1 + \alpha) \int_{\Omega \setminus B_\delta} \frac{\chi_{\R^{n + 1} \setminus E}(y) - \chi_{E}(y)}{|y|^{n + 3 + \alpha}} y_i \, dy.
$$
By the divergence theorem for sets of locally finite perimeter, we have
\begin{align*}
J^{(\delta)} & = \int_{(\Omega \cap E) \setminus B_\delta} \dive \left( \frac{e_i}{|y|^{n + 1 + \alpha}} \right) dy - \int_{\Omega \setminus (E \cup B_\delta)} \dive \left( \frac{e_i}{|y|^{n + 1 + \alpha}} \right) dy \\
& = 2 \int_{(\red E \cap \Omega) \setminus B_\delta} \frac{\nu_E^i(y)}{|y|^{n + 1 + \alpha}} \, d\sigma(y) + \int_{\partial \Omega} \frac{\chi_{E^{(1)}}(y) - \chi_{E^{(0)}}(y)}{|y|^{n + 1 + \alpha}} \, \nu_\Omega^i(y) \, d\sigma(y) \\
& \quad + \frac{1}{\delta^{n + 1 + \alpha}} \int_{\partial B_\delta} \left( \chi_{\R^{n + 1} \setminus E}(y) - \chi_E(y) \right) \nu^i_{B_\delta}(y) \, d\sigma(y),
\end{align*}
for a.e.~$\delta > 0$ small. Notice that we used the formulas of, say,~\cite[Theorem~16.3]{M12} in order to decompose the reduced boundary of intersections. We also took advantage of the hypothesis~$\Haus^n(\red E \cap \partial \Omega) = 0$. Thanks to the~$C^{2, \beta}$ regularity of~$\partial E$ near~$0$, using that~$\nu^i_{B_\delta}(y) = y_i/\delta$, and subtracting the quadratic part of~$v$ at~$0$---as in the argument leading to~\eqref{I2deltaequiv}---, it is not hard to check that
$$
\left| \int_{\partial B_\delta} \left( \chi_{\R^{n + 1} \setminus E}(y) - \chi_E(y) \right) \nu^i_{B_\delta}(y) \, d\sigma(y) \right| \le C \delta^{n + 1 + \beta}.
$$
Hence, letting~$\delta \rightarrow 0^+$ in the above identity, we see that~\eqref{derH2bis} follows from~\eqref{derHbis}.
\end{proof}

With this in hand, we may now proceed to prove Theorem~\ref{Jacintrothm}.

\begin{proof}[Proof of Theorem~\ref{Jacintrothm}]
First of all, we observe that, by Theorem~\ref{perimeterboundsthm} and Lemma~\ref{kerL1overredElem}, both fractional Jacobi operators~\eqref{Jacdef} and~\eqref{truncJac} are well-defined when acting on bounded functions that are smooth near the singularity, under the hypotheses of points~$(i)$ and~$(ii)$, respectively.

In both cases~$(i)$ and~$(ii)$, by the~$\alpha$-minimality of~$E$ there exists a small~$\varepsilon > 0$ for which~$H_\alpha[E](x) = 0$ for every~$x \in \partial E \cap B_\varepsilon(\bar{x})$.

Now, for point~$(i)$, we apply identity~\eqref{derH2} of Proposition~\ref{derHprop} with~$\Omega = \R^{n + 1}$ and deduce that
$$
\left\langle \PV \int_{\red E} \frac{\nu_E(\bar{x}) - \nu_E(y)}{|\bar{x} - y|^{n + 1 + \alpha}} \, d\sigma(y), v \right\rangle = 0
$$
for every~$v \in \R^{n + 1}$ orthogonal to~$\nu_E(\bar{x})$. This means that
$$
\PV \int_{\red E} \frac{\nu_E(\bar{x}) - \nu_E(y)}{|\bar{x} - y|^{n + 1 + \alpha}} \, d\sigma(y) \mbox{ is parallel to } \nu_E(\bar{x}),
$$
or equivalently that
$$
\PV \int_{\red E} \frac{\nu_E(\bar{x}) - \nu_E(y)}{|\bar{x} - y|^{n + 1 + \alpha}} \, d\sigma(y) = \left\langle \PV \int_{\red E} \frac{\nu_E(\bar{x}) - \nu_E(y)}{|\bar{x} - y|^{n + 1 + \alpha}} \, d\sigma(y), \nu_E(\bar{x}) \right\rangle \nu_E(\bar{x}).
$$
From this, the claim of point~$(i)$ readily follows.

To tackle point~$(ii)$, we first observe that, since~$E$ has locally finite perimeter in~$\C_{3 R / 2}$, we have that~$\Haus^n(\red E \cap \partial \C_{r_k}) = 0$ along a sequence of radii~$\{ r_k \}$ converging to~$R$. We now set~$r = r_k$ and prove the validity of inequality~\eqref{Jacle0} with~$r$ in place of~$R$. The conclusion will then follow by letting~$k \rightarrow +\infty$.

Up to a rotation in the hyperplane orthogonal to~$e_{n + 1}$, we may assume that~$\nu_E(\bar{x})$ lies in the~$2$-dimensional plane spanned by~$e_n$ and~$e_{n + 1}$, that is
\begin{equation} \label{nuE2span}
\nu_E(\bar{x}) = \nu_E^n(\bar{x}) e_n + \nu_E^{n + 1}(\bar{x}) e_{n + 1}.
\end{equation}
Applying~\eqref{derH2} with~$\Omega = \C_r$ and the vector~$v = \nu_E^{n + 1}(\bar{x}) e_n - \nu_E^{n}(\bar{x}) e_{n + 1}$ orthogonal to~$\nu_E(\bar{x})$, we get that
\begin{equation} \label{vchoice}
\begin{aligned}
0 & = 2 \, \PV \int_{\red E \cap \C_r} \frac{\nu_E^n(y) \nu_E^{n + 1}(\bar{x}) - \nu_E^{n + 1}(y) \nu_E^{n}(\bar{x})}{|\bar{x} - y|^{n + 1 + \alpha}} \, d\sigma(y) \\
& \quad + \frac{\nu_E^{n + 1}(\bar{x})}{r} \int_{\partial \C_r} \frac{\chi_{E^{(1)}}(y) - \chi_{E^{(0)}}(y)}{|\bar{x} - y|^{n + 1 + \alpha}} \, y_n \, d\sigma(y) \\
& \quad - (n + 1 + \alpha) \, \nu_E^{n + 1}(\bar{x}) \int_{\R^{n + 1} \setminus \C_r} \frac{\chi_{\R^{n + 1} \setminus E}(y) - \chi_{E}(y)}{|\bar{x} - y|^{n + 3 + \alpha}} \, (\bar{x}_n - y_n) \, dy \\
& \quad + (n + 1 + \alpha) \, \nu_E^{n}(\bar{x}) \int_{\R^{n + 1} \setminus \C_r} \frac{\chi_{\R^{n + 1} \setminus E}(y) - \chi_{E}(y)}{|\bar{x} - y|^{n + 3 + \alpha}} \, (\bar{x}_{n + 1} - y_{n + 1}) \, dy.
\end{aligned}
\end{equation}

We proceed to estimate the last three terms of~\eqref{vchoice}. On the one hand, since~$\bar{x} \in \C_{r/2}$ for~$r$ close enough to~$R$,
\begin{align*}
\left| \frac{1}{r} \int_{\partial \C_r} \frac{\chi_{E^{(1)}}(y) - \chi_{E^{(0)}}(y)}{|\bar{x} - y|^{n + 1 + \alpha}} \, y_n \, d\sigma(y) \right| & \le \frac{1}{r} \int_{\partial \C_r} \frac{|y_n|}{|\bar{x} - y|^{n + 1 + \alpha}} \, d\sigma(y) \\
& \le C r^{n - 1} \int_0^{+\infty} \frac{dt}{(r^2 + t^2)^{\frac{n + 1 + \alpha}{2}}} \le \frac{C}{r^{1 + \alpha}}
\end{align*}
for some constant~$C$ depending only on~$n$ and~$\alpha$. On the other hand, using again that~$\bar{x} \in \C_{r/2}$,
\begin{align*}
& \left| \int_{\R^{n + 1} \setminus \C_r} \frac{\chi_{\R^{n + 1} \setminus E}(y) - \chi_{E}(y)}{|\bar{x} - y|^{n + 3 + \alpha}} \, (\bar{x}_n - y_n) \, dy \right| \\
& \hspace{30pt} \le C \int_{\R^n \setminus B_r'} |y'| \Bigg( \int_{0}^{|y'|} \frac{dt}{\big( |y'|^2 + t^2 \big)^{\frac{n + 3 + \alpha}{2}}} + \int_{|y'|}^{+\infty} \frac{dt}{\big( |y'|^2 + t^2 \big)^{\frac{n + 3 + \alpha}{2}}} \Bigg) dy' \\
& \hspace{30pt} \le C \int_{\R^n \setminus B_r'} \frac{dy'}{|y'|^{n + 1 + \alpha}} \le \frac{C}{r^{1 + \alpha}},
\end{align*}
with~$C$ depending only on~$n$ and~$\alpha$. Lastly, using that~$E$ is the global subgraph of~$u$ and writing~$\tilde{u}(z') := u(z' + \bar{x}') - \bar{x}_{n + 1}$, we have
\begin{align*}
& \int_{\R^{n + 1} \setminus \C_r} \frac{\chi_{\R^{n + 1} \setminus E}(y) - \chi_{E}(y)}{|\bar{x} - y|^{n + 3 + \alpha}} \, (\bar{x}_{n + 1} - y_{n + 1}) \, dy \\
& \hspace{25pt} = \int_{\R^n \setminus B_r'(-\bar{x}')} \Bigg( \int_{-\infty}^{\tilde{u}(z')} \frac{t}{\big( |z'|^2 + t^2 \big)^{\frac{n + 3 + \alpha}{2}}} \, dt - \int_{\tilde{u}(z')}^{+\infty} \frac{t}{\big( |z'|^2 + t^2 \big)^{\frac{n + 3 + \alpha}{2}}} \, dt \Bigg) dz' \\
& \hspace{25pt} = - \frac{2}{n + 1 + \alpha} \int_{\R^n \setminus B_r'(-\bar{x}')} \frac{dz'}{\big( |z'|^2 + \tilde{u}(z')^2 \big)^{\frac{n + 1 + \alpha}{2}}} \le 0.
\end{align*}

By multiplying both sides of identity~\eqref{vchoice} by~$\nu_E^n(\bar{x}) / 2$ and taking advantage of the last three estimates, we find
$$
0 \le \PV \int_{\red E \cap \C_r} \frac{\big( \nu_E^n(y) \nu_E^{n + 1}(\bar{x}) - \nu_E^{n + 1}(y) \nu_E^{n}(\bar{x}) \big) \nu_E^n(\bar{x})}{|\bar{x} - y|^{n + 1 + \alpha}} \, d\sigma(y) + \frac{C}{r^{1 + \alpha}} \nu_E^{n + 1}(\bar{x}).
$$
Moreover, recalling~\eqref{nuE2span}, for every~$y \in \red E \cap \C_r$ we have
\begin{align*}
& \big( \nu_E^n(y) \nu_E^{n + 1}(\bar{x}) - \nu_E^{n + 1}(y) \nu_E^{n}(\bar{x}) \big) \nu_E^n(\bar{x}) \\
& \hspace{30pt} = \big( \langle \nu_E(y), \nu_E(\bar{x}) \rangle - \nu_E^{n + 1}(y) \nu_E^{n + 1}(\bar{x}) \big) \nu_E^{n + 1}(\bar{x}) - \big( 1 - \nu_E^{n + 1}(\bar{x})^2 \big) \nu_E^{n + 1}(y) \\
& \hspace{30pt} = \langle \nu_E(y), \nu_E(\bar{x}) \rangle \nu_E^{n + 1}(\bar{x}) - \nu_E^{n + 1}(y) \\
& \hspace{30pt} = \nu_E^{n + 1}(\bar{x}) - \nu_E^{n + 1}(y) - \langle \nu_E(\bar{x}) - \nu_E(y) , \nu_E(\bar{x}) \rangle \nu_E^{n + 1}(\bar{x}),
\end{align*}
and~\eqref{Jacle0} follows.
\end{proof}

As we pointed out in the introduction, for a global subgraph~$E$ to be a minimizer of~$\Per_\alpha$ is equivalent to solve~$H_\alpha[E] = 0$. A careful inspection of the proof just displayed actually reveals that inequality~\eqref{Jacle0} still holds true if we relax the requirement on the vanishing of~$H_\alpha[E]$ and assume only that
$$
H_\alpha[E] = h \quad \mbox{in a neighborhood of } \bar{x} \mbox{ on } \partial E,
$$
for some function~$h$ that admits a~$C^1(\R^{n + 1})$ extension---still called~$h$---that satisfies~$\partial_{x_{n + 1}} h \le 0$. In this case, the constant~$C$ in~\eqref{Jacle0} will also depend on the~$L^\infty$ norm of the first~$n$ components of the gradient of~$h$. We stress that such an assumption on~$\partial_{x_{n + 1}} h$ is consistent with others made in the literature to obtain interior gradient bounds for graphs with prescribed classical mean curvature---compare it with, e.g.,~hypothesis~(1) of~\cite{K86} or Remark~(3) on page~75 of~\cite{W98}, where the opposite sign convention for the mean curvature is chosen, in comparison to ours.

\section{Poincar\'e, Sobolev, and isoperimetric inequalities} \label{sobinesec}

\noindent
Here we show the validity of fractional Poincar\'e, Sobolev, and isoperimetric inequalities on rather general subsets of~$\R^{n + 1}$---even though most of the results could be deduced in some well-behaved metric measure spaces as well. In particular, we will apply these results to~$\alpha$-minimal surfaces.

Throughout the section we consider a set~$\Sigma \subset \R^{n + 1}$ having locally finite~$\Haus^n$-measure. Given~$x_0 \in \Sigma$,~$\rho > 0$, and~$c_\star > 0$, we introduce the condition
\begin{equation} \label{Sigmaperdens}
\Haus^n(\Sigma \cap B_\rho(x_0)) \ge c_\star \rho^n.
\end{equation}

We begin with the following Poincar\'e-type inequality, whose simple proof is an extension of the one given, e.g.,~in~\cite[Section~4]{M03} for the Euclidean setting~$\Sigma = \R^n \times \{ 0 \}$.

\begin{proposition} \label{poinineprop}
Let~$n \ge 1$,~$s \in (0, 1)$,~$p \ge 1$,~$\Sigma \subset \R^{n + 1}$ be a set with locally finite~$\Haus^n$-measure,~$x_0 \in \Sigma$, and~$R > 0$. Assume that~\eqref{Sigmaperdens} holds with~$\rho = R$, for some constant~$c_\star > 0$.

Then, there exists a constant~$C$ depending only on~$n$,~$p$, and~$c_\star$, such that
\begin{equation} \label{poinine}
\| v - (v)_{\Sigma \cap B_R(x_0)} \|_{L^p(\Sigma \cap B_R(x_0))} \le C R^{s} \, [v]_{W^{s, p}(\Sigma \cap B_R(x_0))}
\end{equation}
holds for every~$v \in W^{s, p}(\Sigma \cap B_R(x_0))$.
\end{proposition}
\begin{proof}
Given any~$x \in \Sigma \cap B_R(x_0)$, by Jensen's inequality we have
$$
\left| v(x) - (v)_{\Sigma \cap B_R} \right|^p = \left| \dashint_{\Sigma \cap B_R(x_0)} \left( v(x) - v(y) \right) d\sigma(y) \right|^p \le \dashint_{\Sigma \cap B_R(x_0)} \left| v(x) - v(y) \right|^p d\sigma(y).
$$
Since~$1 \le (2 R)^{n + s p} |x - y|^{- n - s p}$ for all~$x, y \in B_R(x_0)$, we deduce that
$$
\left| v(x) - (v)_{\Sigma \cap B_R} \right|^p \le (2 R)^{n + s p} \dashint_{\Sigma \cap B_R(x_0)} \frac{\left| v(x) - v(y) \right|^p}{|x - y|^{n + s p}} \, d\sigma(y).
$$
By integrating the above relation over~$\Sigma \cap B_R(x_0)$, we find
$$
\| v - (v)_{\Sigma \cap B_R} \|_{L^p(\Sigma \cap B_R(x_0))}^p \le \frac{ 2^{n + p} R^{n + s p}}{\Haus^{n}(\Sigma \cap B_R(x_0))} \, [v]_{W^{s, p}(\Sigma \cap B_R(x_0))}^p.
$$
From this, estimate~\eqref{poinine} follows by recalling that~\eqref{Sigmaperdens} holds with~$\rho = R$.
\end{proof}

We now address fractional Sobolev inequalities on~$\Sigma$ when~$n > s p$. We denote the the critical fractional Sobolev exponent by
$$
p^\star = p^\star_{n, s} := \frac{n p}{n - s p}.
$$
In this first result we assume the density hypothesis~\eqref{Sigmaperdens} to hold at all points and at all scales. Its verification is based on a proof of the fractional Sobolev inequality in the Euclidean space that we learned from~H.~Brezis~\cite{B01}.

\begin{proposition} \label{sobineprop1}
Let~$n \ge 1$,~$s \in (0, 1)$, and~$p \ge 1$ be such that~$n > s p$. Let~$\Sigma$ be a subset of~$\R^{n + 1}$ having locally finite~$\Haus^n$-measure,~$U$ be an open subset of~$\Sigma$, and suppose that condition~\eqref{Sigmaperdens} holds for all points~$x_0 \in U$ and radii~$\rho > 0$, for some constant~$c_\star > 0$.

Then, there exists a constant~$C$ depending only on~$n$,~$s$,~$p$, and~$c_\star$, such that
\begin{equation} \label{sobinepre}
\| v \|_{L^{p^\star}(\Sigma)} \le C \, [v]_{W^{s, p}(\Sigma)}
\end{equation}
holds for every~$v \in W^{s, p}(\Sigma)$ supported inside~$U$.
\end{proposition}

\begin{proof}
By a standard procedure using truncations, we may assume that~$v \in L^\infty(\Sigma)$. Now, for~$x, y \in \Sigma$, we have
$$
|v(x)| \le |v(x) - v(y)| + |v(y)|.
$$
By integrating this relation in~$y$ over~$\Sigma \cap B_r(x)$ for any given~$r > 0$, we obtain
\begin{equation} \label{sobtech1}
|v(x)| \le \dashint_{\Sigma \cap B_r(x)} |v(x) - v(y)| \, d\sigma(y) + \dashint_{\Sigma \cap B_r(x)} |v(y)| \, d\sigma(y).
\end{equation}
On the one hand, by the fact that~$1 \le r^{n + s p}|x - y|^{- n - s p}$ for all~$y \in B_r(x)$ and Jensen's inequality, we have
\begin{align*}
\dashint_{\Sigma \cap B_r(x)} |v(x) - v(y)| \, d\sigma(y) & \le \left( \dashint_{\Sigma \cap B_r(x)} |v(x) - v(y)|^p \, d\sigma(y) \right)^{\frac{1}{p}} \\
& \le \left( r^{n + s p} \dashint_{\Sigma \cap B_r(x)} \frac{|v(x) - v(y)|^p}{|x - y|^{n + s p}} \, d\sigma(y) \right)^{\frac{1}{p}}.
\end{align*}
On the other hand, we use Jensen's inequality once again to estimate
$$
\dashint_{\Sigma \cap B_r(x)} |v(y)| \, d\sigma(y) \le \left( \dashint_{\Sigma \cap B_r(x)} |v(y)|^{p^\star} \, d\sigma(y) \right)^{\frac{1}{p^\star}}.
$$
Note that the right-hand side of the above estimate is finite---this follows from the initial reduction to~$v \in L^\infty(\Sigma)$ and our starting hypothesis~$v \in L^p(\Sigma)$.

By plugging the last two estimates in~\eqref{sobtech1} and taking advantage of~\eqref{Sigmaperdens}, we get
\begin{equation} \label{keytechsobest}
|v(x)| \le C \, \Bigg\{ r^s \left( \int_{\Sigma} \frac{|v(x) - v(y)|^p}{|x - y|^{n + s p}} \, d\sigma(y) \right)^{\frac{1}{p}} + r^{- \frac{n}{p^\star}} \left( \int_{\Sigma} |v(y)|^{p^\star} d\sigma(y) \right)^{\frac{1}{p^\star}} \Bigg\}
\end{equation}
for a.e.~$x \in U$ and every~$r > 0$, for some constant~$C$ depending only on~$n$,~$s$,~$p$, and~$c_\star$. By optimizing in~$r > 0$ the above relation, say, picking as~$r$ the quantity
\begin{equation} \label{rxdef}
r(x) := \left( \int_{\Sigma} |v(y)|^{p^\star} d\sigma(y) \right)^{\frac{p}{n p^\star}} \left( \int_{\Sigma} \frac{|v(x) - v(y)|^p}{|x - y|^{n + s p}} \, d\sigma(y) \right)^{- \frac{1}{n}},
\end{equation}
we conclude that
\begin{equation} \label{keyestoptimized}
|v(x)| \le C \left( \int_{\Sigma} \frac{|v(x) - v(y)|^p}{|x - y|^{n + s p}} \, d\sigma(y) \right)^{\frac{1}{p^\star}} \left( \int_{\Sigma} |v(y)|^{p^\star} \, d\sigma(y) \right)^{\frac{s p}{n p^\star}},
\end{equation}
for a possibly larger~$C$.

By raising both sides of the previous inequality to the power~$p^\star$ and integrating it as~$x$ ranges over~$U$, we conclude that~\eqref{sobinepre} holds true.
\end{proof}

As a particular case of this result, we obtain the fractional Sobolev inequality on nonlocal minimal surfaces, as stated in Theorem~\ref{sobineintrothm}.

\begin{proof}[Proof of Theorem~\ref{sobineintrothm}]
It suffices to apply Proposition~\ref{sobineprop1} with~$\Omega = \R^{n + 1}$. Note that~\eqref{Sigmaperdens} is satisfied thanks to estimate~\eqref{perimeterboundbelow} of Theorem~\ref{perimeterboundsthm}, with a constant~$c_\star$ depending only on~$n$ and~$\alpha$.
\end{proof}

Proposition~\ref{sobineprop1} holds under global lower bounds on the density of~$\Sigma$. For later applications, it is important to have a related inequality involving only local quantities---such as in Proposition~\ref{poinineprop}. This is done in the next corollary, at the cost of complementing the density lower bound with the upper bound
\begin{equation} \label{Sigmaperbound2}
\Haus^n ( \Sigma \cap B_\rho(x_0) ) \le C_\star \rho^n,
\end{equation}
in order to estimate the tails of the~$W^{s, p}$ seminorm.

\begin{corollary} \label{restrictedsobinecor}
Let~$n \ge 2$,~$s \in (0, 1)$, and~$p \ge 1$ be such that~$1 < s p < n$. Given a set~$\Sigma \subset \R^{n + 1}$,~$\bar{x} \in \Sigma$, and~$R > 0$, assume that~\eqref{Sigmaperdens} and~\eqref{Sigmaperbound2} hold for every point~$x_0 \in \Sigma \cap B_{2 R}(\bar{x})$ and every radius~$\rho \in (0, R]$.

Then, there is a constant~$C$ depending only on~$n$,~$s$,~$p$,~$c_\star$, and~$C_\star$, such that
\begin{equation} \label{restrictedsobine}
\| v \|_{L^{p^\star}(\Sigma \cap B_r(\bar{x}))} \le C \left( [v]_{W^{s, p}(\Sigma \cap B_R(\bar{x}))} + \frac{1}{(R - r)^{s}} \| v \|_{L^p(\Sigma \cap B_R(\bar{x}))} \right)
\end{equation}
holds for all~$v \in W^{s, p}(\Sigma \cap B_R(\bar{x}))$ and~$r \in (0, R)$.
\end{corollary}

In the following proof, we will change~$\Sigma$ outside the ball~$B_{3 R/2}$ by replacing it with a hyperplane. In this way, we obtain a new hypersurface~$\Sigma_1$ that satisfies the density assumptions~\eqref{Sigmaperdens} and~\eqref{Sigmaperbound2} at all scales~$\rho > 0$. By doing this, we are able to apply the global fractional Sobolev estimate of Proposition~\ref{sobineprop1} on~$\Sigma_1$ and ultimately establish~\eqref{restrictedsobine}.

Although simple, this replacement trick is a bit non-standard. For this reason, in the forthcoming Remark~\ref{alttrickrmk} we sketch a different proof of Sobolev inequality~\eqref{restrictedsobine}. 

\begin{proof}[Proof of Corollary~\ref{restrictedsobinecor}]
We may suppose that~$\bar{x} = 0$. Consider the auxiliary set
$$
\Sigma_1 := \left( \Sigma \cap B_{\frac{3 R}{2}} \right) \cup \left( \Pi \setminus B_{\frac{3 R}{2}} \right),
$$
where~$\Pi$ denotes the hyperplane orthogonal to~$e_{n + 1}$ passing through the origin, that is,~$\Pi := \left\{ y \in \R^{n + 1} : y_{n + 1} = 0 \right\}$. It is not hard to see that~$\Sigma_1$ satisfies
\begin{equation} \label{Etildeperbelow}
\Haus^n(\Sigma_1 \cap B_\rho(x)) \ge c_\sharp \, \rho^n \quad \mbox{for every } x \in \Sigma_1 \cap B_R \mbox{ and } \rho > 0
\end{equation}
and
\begin{equation} \label{Etildeperabove}
\Haus^n(\Sigma_1 \cap B_\rho(x)) \le C_\sharp \, \rho^n \quad \mbox{for every } x \in \Sigma_1 \mbox{ and } \rho > 0,
\end{equation}
with~$0 < c_\sharp \le C_\sharp$ depending only on~$n$,~$c_\star$, and~$C_\star$.

To prove~\eqref{Etildeperbelow}, we consider separately the cases of small radii~$\rho \in (0, 4 R)$ and large radii~$\rho \ge 4 R$. In the first situation, we apply estimate~\eqref{Sigmaperdens} to the smaller set~$\Sigma_1 \cap B_{\rho / 8}(x) = \Sigma \cap B_{\rho / 8}(x)$, while in the second we just notice that~$\Sigma_1 \cap B_\rho(x)$ contains the flat set~$\Pi \cap \left( B_{3 \rho/4} \setminus B_{2 R} \right)$, which has~$\Haus^n$-measure of order~$\rho^n$.

Claim~\eqref{Etildeperabove} follows instead from the trivial estimate~$\Haus^n(\Pi \cap B_\rho(x)) \le C \rho^n$ and
\begin{equation} \label{Etildeperabovebis}
\Haus^n(\Sigma \cap B_{3 R / 2} \cap B_\rho(x)) \le C \rho^n \quad \mbox{for every } x \in \Sigma_1 \mbox{ and } \rho > 0,
\end{equation}
for some constant~$C$ depending only on~$n$ and~$C_\star$. To get~\eqref{Etildeperabovebis}, we distinguish between~$\rho \in (0, R/8)$ and~$\rho \ge R/8$. In the first case, we may assume that~$x \in B_{7 R / 4}$ (since otherwise~$B_\rho(x) \cap B_{3 R / 2} = \varnothing$ and~\eqref{Etildeperabovebis} is automatically satisfied) and directly apply~\eqref{Sigmaperbound2}. On the other hand, when~$\rho \ge R / 8$, hypothesis~\eqref{Sigmaperbound2} and a covering argument immediately lead to the desired estimate for the larger set~$\Sigma \cap B_{3 R / 2}$.

Let now~$v \in W^{s, p}(\Sigma \cap B_R)$,~$r \in (0, R)$, and consider a cutoff function~$\eta \in C^\infty(\R^{n + 1})$ satisfying~$0 \le \eta \le 1$ in~$\R^{n + 1}$,~$\supp(\eta) \subset B_{(R + r) / 2}$,~$\eta = 1$ in~$B_r$, and~$|\nabla \eta| \le 4 / (R - r)$ in~$\R^{n + 1}$. Observe that~$\eta v$ is supported inside~$U := \Sigma \cap B_R = \Sigma_1 \cap B_R$. Thanks to this and~\eqref{Etildeperbelow}, we may apply Proposition~\ref{sobineprop1} to the function~$\eta v$, on the set~$\Sigma_1$. We get
\begin{equation} \label{restrictedsobinetech0}
\| v \|_{L^{p^\star}(\Sigma \cap B_r)} \le \| \eta v \|_{L^{p^\star}(\Sigma \cap B_R)} \le \| \eta v \|_{L^{p^\star}(\Sigma_1)} \le C \, [\eta v]_{W^{s, p}(\Sigma_1)},
\end{equation}
for some constant~$C$ depending only on~$n$,~$s$,~$p$,~$c_\star$, and~$C_\star$. Thanks to the properties of~$\eta$,
\begin{equation} \label{restrictedsobinetech1}
\begin{aligned}
[\eta v]_{W^{s, p}(\Sigma_1)}^p & = \int_{\Sigma \cap B_R} \int_{\Sigma \cap B_R} \frac{|\eta(x) v(x) - \eta(y) v(y)|^p}{|x - y|^{n + s p}} \, d\sigma(x) \, d\sigma(y) \\
& \quad + 2 \int_{\Sigma \cap B_{\frac{R + r}{2}}} |v(x)|^p \left( \int_{\Sigma_1 \setminus B_R} \frac{d\sigma(y)}{|x - y|^{n + s p}} \right) d\sigma(x). 
\end{aligned}
\end{equation}

We have
\begin{align*}
|\eta(x) v(x) - \eta(y) v(y)|^p & = \left| \eta(y) \left( v(x) - v(y) \right) + v(x) \left( \eta(x) - \eta(y) \right) \right|^p \\
& \le 2^{p - 1} \Big( |v(x) - v(y)|^p + |v(x)|^p |\eta(x) - \eta(y)|^p \Big),
\end{align*}
and thus
\begin{align*}
& \int_{\Sigma \cap B_R} \int_{\Sigma \cap B_R} \frac{|\eta(x) v(x) - \eta(y) v(y)|^p}{|x - y|^{n + s p}} \, d\sigma(x) \, d\sigma(y) \\
& \hspace{20pt} \le 2^{p - 1} \left\{ [v]_{W^{s, p}(\Sigma \cap B_R)}^p + \int_{\Sigma \cap B_{\frac{R + r}{2}}} |v(x)|^p \left( \int_{\Sigma_1} \frac{|\eta(x) - \eta(y)|^p}{|x - y|^{n + s p}} \, d\sigma(y) \right) d\sigma(x) \right\}.
\end{align*}
By combining this with~\eqref{restrictedsobinetech1} and noting that~$B_{\frac{R - r}{2}}(x) \subseteq B_R$ for every~$x \in B_{\frac{R + r}{2}}$, we obtain
$$
[\eta v]_{W^{s, p}(\Sigma_1)}^p \le C \left( [v]_{W^{s, p}(\Sigma \cap B_R)}^p + \int_{\Sigma \cap B_{\frac{R + r}{2}}} |v(x)|^p \, \Xi(x) \, d\sigma(x) \right),
$$
with
$$
\Xi(x) := \int_{\Sigma_1} \frac{|\eta(x) - \eta(y)|^p}{|x - y|^{n + s p}} \, d\sigma(y) + \int_{\Sigma_1 \setminus B_{\frac{R - r}{2}}(x)} \frac{d\sigma(y)}{|x - y|^{n + s p}}.
$$
In view of Lemma~\ref{kerL1overredElem} (with~$\beta = sp > 1$) and Corollary~\ref{etasemilem}---that both can be applied, thanks to~\eqref{Etildeperabove}---, we see that~$\Xi(x) \le C (R - r)^{- s p}$ for every~$x \in \Sigma \cap B_{\frac{R + r}{2}}$. Accordingly,
$$
[\eta v]_{W^{s, p}(\Sigma_1)}^p \le C \left( [v]_{W^{s, p}(\Sigma \cap B_R)}^p + \frac{1}{(R - r)^{s p}} \| v \|_{L^p(\Sigma \cap B_R)}^p \right),
$$
and~\eqref{restrictedsobine} follows from~\eqref{restrictedsobinetech0}.
\end{proof}

\begin{remark} \label{alttrickrmk}
The following is an alternative, perhaps more natural, proof of Corollary~\ref{restrictedsobinecor}. Let~$v \in W^{s, p}(\Sigma \cap B_R)$ be a bounded function, extended by zero in~$\Sigma \setminus B_R$. We may suppose that~$\| v \|_{L^{p^\star}(\Sigma \cap B_R)} = 1$. For~$x \in \Sigma \cap B_R$, let~$r(x)$ be as in~\eqref{rxdef}, but with both domains of integration replaced by~$\Sigma \cap B_{2 R}$---as the proof of~\eqref{keytechsobest} really gives. Set~$G := \{ x \in \Sigma \cap B_R : r(x) < R \}$ and~$F := (\Sigma \cap B_R) \setminus G = \{ x \in \Sigma \cap B_R : r(x) \ge R \}$.

For~$x \in G$, we argue as in the proof of Proposition~\ref{sobineprop1} to establish~\eqref{keyestoptimized} with~$\Sigma$ replaced by~$\Sigma \cap B_{2 R}$---this can be done since we may apply~\eqref{Sigmaperdens} with~$\rho = r(x) \in (0, R)$. On the other hand, for~$x \in F$ we can still obtain the weaker inequality~\eqref{keytechsobest} with~$r = R$. This, combined with~$\| v \|_{L^{p^\star}(\Sigma \cap B_R)} = 1$ and~$r(x) \ge R$, leads to the uniform bound~$|v(x)| \le C R^{- n / p^\star}$. This, in turn, gives that~$|v|^{p^\star} \le R^{- s p} |v|^p$ in~$F$. By putting together this estimate with the one in~$G$, we conclude that
$$
\| v \|_{L^{p^\star}(\Sigma \cap B_R)} \le C \, \Big( [v]_{W^{s, p}(\Sigma \cap B_{2 R})} + R^{- s} \| v \|_{L^p(\Sigma \cap B_R)} \Big).
$$

This is essentially estimate~\eqref{restrictedsobinetech0}. The more refined inequality~\eqref{restrictedsobine} then follows via a cutoff argument as in the second part of the proof of Corollary~\ref{restrictedsobinecor}.
\end{remark}

Next is a Morrey-type~$L^\infty$ estimate for functions in the Sobolev space~$W^{s, p}$ whenever~$n < s p$. Our argument is an appropriate modification of those presented in~\cite[Lemma~2.2]{G03} and~\cite[Theorem~8.2]{DPV12}.

\begin{proposition} \label{morreyprop}
Let~$n \ge 1$,~$s \in (0, 1)$, and~$p \ge 1$ be such that~$n < s p$. Given a set~$\Sigma \subset \R^{n + 1}$, a point~$\bar{x}, \in \Sigma$ and~$R > 0$, assume that~\eqref{Sigmaperdens} and~\eqref{Sigmaperbound2} hold for every~$x_0 \in \Sigma \cap B_{2 R}(\bar{x})$ and~$\rho \in (0, R]$.

Then, there exists a constant~$C$ depending only on~$n$,~$s$,~$p$,~$c_\star$, and~$C_\star$, such that
\begin{equation} \label{morreyine}
\| v \|_{L^\infty(\Sigma \cap B_r(\bar{x}))} \le C R^{\frac{sp - n}{p}} \left( [v]_{W^{s, p}(\Sigma \cap B_R(\bar{x}))} + \frac{1}{(R - r)^{s}} \| v \|_{L^p(\Sigma \cap B_R(\bar{x}))} \right)
\end{equation}
holds for all~$v \in W^{s, p}(\Sigma \cap B_R(\bar{x}))$ and~$r \in (0, R)$.
\end{proposition}
\begin{proof}
To prove~\eqref{morreyine}, it suffices to show that
\begin{equation} \label{Morreyclaim}
|u(x)| \le C \left( R^{\frac{sp - n}{p}} [u]_{W^{s, p}(\Sigma \cap B_{2 R}(\bar{x}))} + R^{- \frac{n}{p}} \| u \|_{L^p(\Sigma \cap B_{2 R}(\bar{x}))} \right)
\end{equation}
for~$\Haus^n$-a.e.~$x \in \Sigma \cap B_R(\bar{x})$ and every~$u \in W^{s, p}(\Sigma \cap B_{2 R}(\bar{x}))$. Indeed, one can then apply~\eqref{Morreyclaim} to the truncated function~$u = \eta v$, with~$\eta$ a convenient cutoff between the balls~$B_r(\bar{x})$ and~$B_{(R + r)/2}(\bar{x})$, and argue as in the final part of the proof of Corollary~\ref{restrictedsobinecor}.

Let~$u \in W^{s, p}(\Sigma \cap B_{2 R}(\bar{x}))$ and note that
\begin{equation} \label{TDL}
u(x) = \lim_{r \rightarrow 0^+} (u)_{\Sigma \cap B_r(x)}
\end{equation}
for~$\Haus^n$-a.e.~$x \in \Sigma \cap B_R(\bar{x})$. This is true by the Lebesgue differentiation theorem, which holds in the metric measure space~$\Sigma \cap B_{2 R}(\bar{x})$ (endowed with the metric inherited from the ambient space~$\R^{n + 1}$ and the measure~$\Haus^n \llcorner \Sigma$) thanks to assumptions~\eqref{Sigmaperdens} and~\eqref{Sigmaperbound2}, as shown in~\cite[Section~3.4]{HKST15}.

Take then any~$x \in \Sigma \cap B_R(\bar{x})$ for which~\eqref{TDL} holds true. Up to a translation, we may assume that~$x = 0$. For~$\lambda > n$, we define
\begin{equation} \label{campadef}
M_\lambda := \sup_{\rho \in (0, R/2]} \left( \rho^{-\lambda} \int_{\Sigma \cap B_\rho} |u(y) - (u)_{\Sigma \cap B_\rho}|^p \, d\sigma(y) \right)^{\frac{1}{p}}.
\end{equation}
We claim that there is a constant~$C_\lambda$ depending on~$n$,~$p$,~$c_\star$, and~$\lambda$, such that
\begin{equation} \label{giustiine}
\left| u(0) - (u)_{\Sigma \cap B_{R / 2}} \right| \le C_\lambda R^{\frac{\lambda - n}{p}} M_\lambda.
\end{equation}

Indeed, consider any two radii~$0 < \rho_1 < \rho_2 \le R/2$. By Jensen's inequality, we have
\begin{align*}
\left| (u)_{\Sigma \cap B_{\rho_1}} - (u)_{\Sigma \cap B_{\rho_2}} \right| & \le \dashint_{\Sigma \cap B_{\rho_1}} \left| u(y) - (u)_{\Sigma \cap B_{\rho_2}} \right| d\sigma(y) \\
& \le \left( \dashint_{\Sigma \cap B_{\rho_1}} \left| u(y) - (u)_{\Sigma \cap B_{\rho_2}} \right|^p d\sigma(y) \right)^{\frac{1}{p}}
\end{align*}
and thus, recalling~\eqref{campadef} and~\eqref{Sigmaperdens},
\begin{equation} \label{giustitech}
\left| (u)_{\Sigma \cap B_{\rho_1}} - (u)_{\Sigma \cap B_{\rho_2}} \right| \le c_\star^{- \frac{1}{p}} \rho_1^{- \frac{n}{p}} \rho_2^{\,\,\, \frac{\lambda}{p}} \, M_\lambda.
\end{equation}
Let~$\{ r_k \}$ be defined by~$r_k := 2^{- k} R$ for all~$k \in \N$. By applying inequality~\eqref{giustitech} with~$\rho_1 = r_{k + 1}$,~$\rho_2 = r_{k}$ and summing over~$k$, we easily obtain that
\begin{align*}
\left| (u)_{\Sigma \cap B_{R/2}} - \lim_{\ell \rightarrow +\infty} (u)_{\Sigma \cap B_{r_\ell}} \right| & \le \left( \frac{2^n}{c_\star} \right)^{\frac{1}{p}} R^{\frac{\lambda - n}{p}} M_\lambda \sum_{k = 1}^{+\infty} 2^{ - \frac{\lambda - n}{p} \, k} \le C_\lambda \, R^{\frac{\lambda - n}{p}} M_\lambda.
\end{align*}
Claim~\eqref{giustiine} then follows, since~\eqref{TDL} is satisfied for~$x = 0$.

Let~$\rho \in (0, R/2]$. Using once again Jensen's inequality, assumption~\eqref{Sigmaperdens}, and the fact that~$1 \le (2 \rho)^{n + s p} |y - z|^{ - n - s p}$ for all~$y, z \in B_\rho$, we have
\begin{align*}
\int_{\Sigma \cap B_\rho} |u(y) - (u)_{\Sigma \cap B_\rho}|^p \, d\sigma(y) & \le \int_{\Sigma \cap B_\rho} \dashint_{\Sigma \cap B_\rho} |u(y) - u(z)|^p \, d\sigma(z) d\sigma(y) \\
& \le C \rho^{s p} \int_{\Sigma \cap B_\rho} \int_{\Sigma \cap B_\rho} \frac{|u(y) - u(z)|^p}{|y - z|^{n + s p}} \, d\sigma(z) d\sigma(y) \\
& \le C \rho^{s p} \, [u]_{W^{s, p}(\Sigma \cap B_{R})}^p,
\end{align*}
where from now on~$C$ denotes constants depending only on~$n$,~$s$,~$p$,~$c_\star$, and~$C_\star$. Recalling definition~\eqref{campadef}, the above inequality yields~$M_{s p} \le C [u]_{W^{s, p}(\Sigma \cap B_{R})}$. By this,~\eqref{giustiine}, and the fact that~$n < s p$, we get that
$$
\left| u(0) - (u)_{\Sigma \cap B_{R / 2}} \right| \le C R^{\frac{sp - n}{p}} [u]_{W^{s, p}(\Sigma \cap B_{R})}.
$$
Thanks to this bound, Jensen's inequality, and~\eqref{Sigmaperdens}, we are led to~\eqref{Morreyclaim} for~$x = 0$.
\end{proof}

To conclude the section, we deduce an isoperimetric-type inequality involving a nonlocal notion of perimeter on~$\Sigma$. For any~$\Haus^n$-measurable set~$A \subseteq \Sigma$, we define the~$s$-perimeter of~$A$ on~$\Sigma$ as the quantity
$$
\Per_{\Sigma, s}(A) := \int_{A} \int_{\Sigma \setminus A} \frac{d\sigma(x) d\sigma(y)}{|x - y|^{n + s}}.
$$
We then have the following result.

\begin{corollary} \label{isopinecor}
Let~$n \ge 1$,~$s \in (0, 1)$, and~$\Sigma \subset \R^{n + 1}$ be a set with locally finite~$\Haus^n$-measure. Assume that~\eqref{Sigmaperdens} holds for all~$x_0 \in \Sigma$ and~$\rho > 0$, for some constant~$c_\star > 0$.

Then, there exists a constant~$C > 0$ depending only on~$n$,~$s$, and~$c_\star$, such that
\begin{equation} \label{isopine}
\Haus^n(A)^{\frac{n - s}{n}} \le C \Per_{\Sigma, s}(A)
\end{equation}
for every set~$A \subseteq \Sigma$ with finite~$\Haus^n$-measure.

In particular,~\eqref{isopine} holds when~$\Sigma$ is the reduced boundary~$\red E$ of a minimizer of the~$\alpha$-perimeter in all of~$\R^{n + 1}$---in this case, with a constant~$C$ depending only on~$n$,~$s$, and~$\alpha$.
\end{corollary}
\begin{proof}
It suffices to apply Proposition~\ref{sobineprop1} with~$p = 1$,~$v = \chi_A$, and~$U = \Sigma$---observe that~$U$ need not be bounded in the statement of Proposition~\ref{sobineprop1}. When~$\Sigma = \red E$, with~$\red E$ an~$\alpha$-minimal surface in~$\R^{n + 1}$, the result follows since estimate~\eqref{perimeterboundbelow} of Theorem~\ref{perimeterboundsthm} ensures the validity of~\eqref{Sigmaperdens} for any point~$x_0 \in \red E$ and scale~$\rho > 0$.
\end{proof}

\section{The weak Harnack inequality} \label{weakharsec}

\noindent
In this section, we establish a weak Harnack inequality for non-negative supersolutions of rather general integral equations posed on hypersurfaces of Euclidean space satisfying appropriate density assumptions. In particular, the inequality will hold for equations posed on nonlocal minimal surfaces---including those which are not graphs.

\subsection{The weak Harnack inequality on general hypersurfaces}

Let~$n \ge 1$ and~$\Sigma$ be a~$\Haus^n$-measurable subset of~$\R^{n + 1}$. In this subsection, we prove a weak Harnack inequality for non-negative supersolutions of equations driven by the integro-differential operator
$$
\L_K v(x) := \PV \int_{\Sigma} \left( v(y) - v(x) \right) K(x, y) \, d\sigma(y),
$$
where~$K: \Sigma \times \Sigma \to [0, +\infty]$ is a~$\Haus^n$-measurable function. See Section~\ref{notsec} for the definition of principal value surface integrals.

Given an open set~$\Omega \subseteq \R^{n + 1}$ and three real numbers~$R_0 > 0$,~$\Lambda \ge 1$,~$s \in (1/2, 1)$, we require~$K$ to satisfy the conditions
\begin{alignat}{2}
\label{Ksymm} K(x, y) = K(y, x) \hspace{50pt} & \quad \mbox{for~$\Haus^n$-a.e.~} x, y \in \Sigma, \\
\label{Kell} \frac{\Lambda^{-1} \chi_{B_{R_0}}(x - y)}{|x - y|^{n + 2 s}} \le K(x, y) \le \frac{\Lambda}{|x - y|^{n + 2 s}} & \quad \mbox{for~$\Haus^n$-a.e.~} x, y \in \Sigma \cap \Omega,
\end{alignat}
and
\begin{equation} \label{Ksupp}
K(x, y) = 0 \quad \mbox{for~$\Haus^n$-a.e.~} x \in \Sigma \cap \Omega \mbox{ and } y \in \Sigma \setminus \Omega.
\end{equation}
The integro-differential inequality defining our class of supersolutions will be assumed to hold only inside~$\Sigma \cap \Omega$. In this set, we suppose~$\Sigma$ to have the density property
\begin{equation} \label{EperCflat}
c_\star \rho^n \le \Haus^n(\Sigma \cap B_\rho(x)) \le C_\star \rho^n \quad \mbox{for every } x \in \Sigma \cap \Omega \mbox{ and~} \rho \in (0, R_0]
\end{equation}
for some constants~$C_\star \ge c_\star > 0$. We also assume that
\begin{equation} \label{0inOmega}
0 \in \Sigma \quad \mbox{and} \quad \Sigma \cap B_{R_0} \subseteq \Omega.
\end{equation}

Observe that, unless~$\Sigma \cap \Omega = \Sigma$, for~$x \in \Sigma \cap \Omega$ condition~\eqref{Ksupp} forces~$K(x, \cdot)$ to be supported inside the proper set~$\Sigma \cap \Omega$, where~$\Sigma$ behaves nicely, according to~\eqref{EperCflat}. This is why we refer to~$\L_K$ as a truncated operator in~$\Omega$. In light of this, the operator~$\L_K$ can be equivalently written as
$$
\L_K v(x) := \PV \int_{\Sigma \cap \Omega} \left( v(y) - v(x) \right) K(x, y) \, d\sigma(y),
$$
whenever it is evaluated at a point~$x \in \Sigma \cap \Omega$.

Concrete examples include the cases where~$\Omega = \C_R$ is a cylinder,~$R_0 = R/2$, and
$$
K(x, y) = \frac{\chi_{\C_R}(x) \chi_{\C_R}(y)}{|x - y|^{n + 2 s}},
$$
as well as~$\Omega = \R^{n + 1}$ and~$K(x, y) = |x - y|^{- n - 2 s}$. For~$s = (1 + \alpha) / 2$, these choices lead to the operators introduced in Subsection~\ref{jacobi}, which play a key role throughout the paper. Note that we also include the case when~$\Omega = \R^{n + 1}$ and~$K$ is a translation invariant truncated kernel such as~$K(x, y) = \chi_{B_{R_0}}(x - y) |x - y|^{- n - 2 s}$.

We stress that our results are limited to operators of fractional order~$2 s$ strictly greater than~$1$. See Remark~\ref{s>12rmk} for comments on this point.

Given an open set~$U \subseteq \Sigma \cap \Omega$, we consider the Sobolev space~$\H^K(U)$ made up by the~$\Haus^n$-measurable functions~$v: \Sigma \to \R$ for which~$v|_U \in L^2(U)$ and
$$
[v]_{\H^K(U)}^2 := \iint_{Q(U)} \left| v(x) - v(y) \right|^2 d\mu_K(x, y) < +\infty,
$$
where we set
$$
Q(U) := \left( \Sigma \times \Sigma \right) \setminus \left( \left( \Sigma \setminus U \right) \times \left( \Sigma \setminus U \right) \right)
$$
and we adopted the short-hand notation
$$
d\mu = d\mu_K(x, y) := K(x, y) \, d\sigma(x) d\sigma(y).
$$
Notice that~$\H^K(U)$ differs from~$H^s(U)$---as defined in Section~\ref{notsec}---primarily because the seminorm associated to the former is determined by an integral over~$Q(U)$ and not just~$U \times U$. We also define~$\H^K_\loc(U)$ as the set of functions that belong to~$\H^K(V)$ for every open set~$V \subset \subset U$.

Let~$b, f: U \to \R$ be two bounded~$\Haus^n$-measurable functions satisfying
\begin{equation} \label{bgeb*}
b \le R_0^{- 2 s} b_* \quad \mbox{in } \Sigma \cap B_{R_0}
\end{equation}
and
\begin{equation} \label{fged}
f \ge - d \quad \mbox{in } \Sigma \cap B_{R_0}
\end{equation}
for some non-negative constants~$b_*$ and~$d$. We deal with supersolutions of the equation
\begin{equation} \label{Lv=0}
- \L_K v + b v = f
\end{equation}
in~$U$. As we will see later, we need to take into account two types of them:~\emph{weak supersolutions} and~\emph{generalized pointwise supersolutions}, as defined next. All results in this section apply to both of them.

First, we consider the following standard variational notion of supersolution. The Moser iteration that leads to the weak Harnack inequality will be carried out for this type of supersolutions.

\begin{definition}
A function~$w \in \H_\loc^K(U)$ is a~\emph{weak supersolution} of~\eqref{Lv=0} in~$U$ if
\begin{equation} \label{weakdefine}
\begin{aligned}
& \frac{1}{2} \int_{\Sigma} \int_{\Sigma} \left( w(x) - w(y) \right) \left( \varphi(x) - \varphi(y) \right) d\mu + \int_{\Sigma} b(x) w(x) \varphi(x) \, d\sigma(x) \\
& \hspace{250pt} \ge \int_{\Sigma} f(x) \varphi(x) \, d\sigma(x)
\end{aligned}
\end{equation}
for every non-negative function~$\varphi \in \H^K(U)$ compactly supported inside~$U$.
\end{definition}

Since the equations involving the fractional Jacobi operators in Theorem~\ref{Jacintrothm} hold a priori only at points where the nonlocal minimal surface~$\Sigma = \partial E$ is~$C^3$, we also need to consider the following notion of supersolution.

\begin{definition} \label{esspointsoldef}
An~$\Haus^n$-measurable function~$w: \Sigma \cap \Omega \to \R$ is called a~\emph{generalized pointwise supersolution} of~\eqref{Lv=0} in~$U$ if there exists a closed set~$S \subset U$ satisfying~$\Cap_{\, \Sigma \cap \Omega, s}(S) = 0$ for~$n \ge 2$ and~$S = \varnothing$ for~$n = 1$, such that:
\begin{itemize}[leftmargin=*]
\item $w$ is bounded in~$\Sigma \cap \Omega$;
\item the quantity~$\L_K w(x)$ is well-defined at every~$x \in U \setminus S$, and the principal value that defines it converges uniformly in every compact set~$V \subset U \setminus S$;
\item it holds
\begin{equation} \label{Lwge0}
- \L_K w(x) + b(x) w(x) \ge f(x)
\end{equation}
for every~$x \in U \setminus S$.
\end{itemize}
\end{definition}

Note that the quantity~$\Cap_{\, \Sigma \cap \Omega, s}$ appearing in Definition~\ref{esspointsoldef} is the~$s$-fractional capacity on~$\Sigma \cap \Omega$ introduced in Subsection~\ref{capsubsec}. In our main applications, the set~$S$ will be the singular set of a nonlocal minimal surface~$\partial E$.

The hypothesis that~$w$ is bounded rules out undesirable pointwise supersolutions such as fundamental solutions. For instance, when~$n \ge 2$, the function~$- |x|^{- n + 2 s}$ is a pointwise solution for the fractional Laplacian of order~$2 s$ in~$\R^n$ outside of a set of zero~$s$-fractional capacity, but not a weak supersolution.

Before heading to the statement of the weak Harnack inequality, we establish that every generalized pointwise supersolution~$w$ is indeed a weak supersolution. Note that our proof will use crucially that~$w$ is bounded.

\begin{lemma} \label{genisweaklem}
Under assumptions~\eqref{Ksymm}-\eqref{0inOmega}, let~$w$ be a generalized pointwise supersolution of~\eqref{Lv=0} in~$\Sigma \cap B_R$, for some~$R \in (0, R_0)$ and with~$b$ and~$f$ bounded in~$\Sigma \cap B_R$.

Then,~$w$ is a weak supersolution of~\eqref{Lv=0} in~$\Sigma \cap B_R$.
\end{lemma}

\begin{proof}
Up to scaling, we may assume that~$R = 1$. Let~$S$ be a closed subset of~$\Sigma \cap B_1$ with~$\Cap_{\, \Sigma \cap \Omega, s}(S) = 0$, outside of which~\eqref{Lwge0} is satisfied. By the Definition~\ref{capdef1} of~$s$-fractional capacity, there exists a sequence of functions~$\{ v_k \}_{k \in \N} \subset H^s(\Sigma \cap \Omega)$ such that~$0 \le v_k \le 1$ in~$\Sigma \cap \Omega$,~$v_k = 1$ in an open neighborhood~$A_k$ of~$S$, and
\begin{equation} \label{vkcap}
\| v_k \|_{H^s(\Sigma \cap \Omega)} \le \frac{1}{k}.
\end{equation}
By taking a subsequence, we also suppose that~$v_k$ converges to zero~$\Haus^n$-a.e.~in~$\Sigma \cap \Omega$.

Given~$r \in (1/2, 1)$, let~$\eta \in C^\infty(\R^{n + 1})$ be such that~$0 \le \eta \le 1$ in~$\R^{n + 1}$,~$\supp(\eta) \subset B_{(2 + r)/3}$,~$\eta = 1$ in~$B_{(1 + 2 r)/3}$, and~$|\nabla \eta| \le 6 / (1 - r)$ in~$\R^{n + 1}$. Write~$\psi_k := \eta (1 - v_k)$.

First, we show that~$w \in \H^{K}(\Sigma \cap B_r)$. As~$w$ is bounded, we only need to check that
\begin{equation} \label{wHkfin}
[w]_{\H^K(\Sigma \cap B_r)} < +\infty.
\end{equation}

Set~$\widetilde{w} := w - \inf_{\Sigma \cap \Omega} w + 1$. We multiply inequality~\eqref{Lwge0} by the non-negative function~$\psi_k^2 \widetilde{w}^{-1}$ and integrate over the set~$(\Sigma \setminus A_k) \cap B_1$, which contains its support. We obtain
\begin{equation} \label{easylemeq1}
\int_{(\Sigma \setminus A_k) \cap B_1} \Big( - \L_K w(x) + b(x) w(x) - f(x) \Big) \frac{\psi_k^2(x)}{\widetilde{w}(x)} \, d\sigma(x) \ge 0.
\end{equation}
Taking advantage of the uniform convergence in~$(\Sigma \setminus A_k) \cap \overline{B}_{(2 + r) / 3}$ of the principal value defining~$\L_K$ (as required in Definition~\ref{esspointsoldef}), of the fact that~$\psi_k$ is supported inside the set~$(\Sigma \setminus A_k) \cap B_{(2 + r) / 3}$, and symmetrizing in~$x$ and~$y$ thanks to~\eqref{Ksymm}, we get
\begin{equation} \label{psikovertildewtest}
\begin{aligned}
& - \int_{(\Sigma \setminus A_k) \cap B_1} \L_K w(x) \frac{\psi_k^2(x)}{\widetilde{w}(x)} \, d\sigma(x) \\
& \hspace{35pt} = \lim_{\delta \rightarrow 0^+} \int_{\Sigma} \frac{\psi_k^2(x)}{\widetilde{w}(x)} \left( \int_{\Sigma \setminus B_\delta(x)} \left( w(x) - w(y) \right) K(x, y) \, d\sigma(y) \right) d\sigma(x) \\
& \hspace{35pt} = \frac{1}{2} \lim_{\delta \rightarrow 0^+} \iint_{ \{ (x, y) \in (\Sigma \times \Sigma) : |x - y| > \delta \}} \left( \widetilde{w}(x) - \widetilde{w}(y) \right) \left( \frac{\psi_k^2(x)}{\widetilde{w}(x)} - \frac{\psi_k^2(y)}{\widetilde{w}(y)} \right) d\mu.
\end{aligned}
\end{equation}
Since
$$
\left( \widetilde{w}(x) - \widetilde{w}(y) \right) \left( \frac{\psi_k^2(x)}{\widetilde{w}(x)} - \frac{\psi_k^2(y)}{\widetilde{w}(y)} \right) = |\psi_k(x) - \psi_k(y)|^2 - \frac{|\psi_k(y) \widetilde{w}(x) - \psi_k(x) \widetilde{w}(y)|^2}{\widetilde{w}(x) \widetilde{w}(y)},
$$
we deduce from~\eqref{easylemeq1} and~\eqref{psikovertildewtest} that
\begin{align*}
\int_{\Sigma} \int_{\Sigma} \frac{|\psi_k(y) \widetilde{w}(x) - \psi_k(x) \widetilde{w}(y)|^2}{\widetilde{w}(x) \widetilde{w}(y)} \, d\mu & \le \int_{\Sigma} \int_{\Sigma} |\psi_k(x) - \psi_k(y)|^2 \, d\mu \\
& \quad + 2 \int_{\Sigma \cap B_1} \big( b(x) w(x) + f(x) \big) \frac{\psi_k^2(x)}{\widetilde{w}(x)} \, d\sigma(x).
\end{align*}

Next, on the one hand, by the definition of~$\psi_k$ and the boundedness of~$w$, we have
\begin{align*}
& \int_{\Sigma} \int_{\Sigma} \frac{|\psi_k(y) \widetilde{w}(x) - \psi_k(x) \widetilde{w}(y)|^2}{\widetilde{w}(x) \widetilde{w}(y)} \, d\mu \\
& \hspace{60pt} \ge \frac{1}{C} \iint_{\big( \Sigma \cap B_{\frac{1 + 2 r}{3}} \big)^2} |(1 - v_k(x)) \widetilde{w}(x) - (1 - v_k(y)) \widetilde{w}(y)|^2 \, d\mu,
\end{align*}
where, from now on,~$C \ge 1$ denotes constants independent of~$k$. On the other hand,
$$
\left| \psi_k(x) - \psi_k(y) \right|^2 \le 2 \left| \eta(x) - \eta(y) \right|^2 + 2 \left| v_k(x) - v_k(y) \right|^2,
$$
and thus, using this, conditions~\eqref{Ksymm}-\eqref{0inOmega}, Corollary~\ref{etasemilem} (used with~$r$ replaced by~$1 - r$), and~\eqref{vkcap}, we find that
\begin{align*}
& \int_{\Sigma} \int_{\Sigma} |\psi_k(x) - \psi_k(y)|^2 \, d\mu \\
& \hspace{50pt} \le C \left\{ \int_{\Sigma \cap B_1} \left( \int_{\Sigma \cap \Omega} \frac{|\eta(x) - \eta(y)|^2}{|x - y|^{n + 2 s}} \, d\sigma(y) \right) d\sigma(x) + [v_k]_{H^s(\Sigma \cap \Omega)}^2 \right\} \\
& \hspace{50pt} \le C \left( \frac{\Haus^{n}(\Sigma \cap B_1)}{(1 - r)^{2 s}} + \frac{1}{k^2} \right) \le C.
\end{align*}
Moreover, as~$\widetilde{w} \ge 1$ in~$\Sigma \cap B_1$,
\begin{equation} \label{RHSest}
\begin{aligned}
& \int_{\Sigma \cap B_1} \big( b(x) w(x) + f(x) \big) \frac{\psi_k^2(x)}{\widetilde{w}(x)} \, d\sigma(x) \\
& \hspace{40pt} \le \left( \| b \|_{L^\infty(\Sigma \cap B_1)} \| w \|_{L^\infty(\Sigma \cap B_1)} + \| f \|_{L^\infty(\Sigma \cap B_1)} \right) \Haus^n(\Sigma \cap B_1) \le C.
\end{aligned}
\end{equation}
Putting the last estimates together and letting~$k \rightarrow +\infty$, by Fatou's lemma and the pointwise a.e.~convergence of~$v_k$ to~$0$, we conclude that
\begin{equation} \label{locsemifinite}
\iint_{\big( \Sigma \cap B_{\frac{1 + 2 r}{3}} \big)^2} |w(x) - w(y)|^2 \, d\mu < +\infty.
\end{equation}

By Lemma~\ref{kerL1overredElem} (used with~$(1 - r) / 3$ in place of~$r$), we can also estimate
\begin{align*}
& \int_{\Sigma \cap B_r} \int_{ \Sigma \setminus B_{\frac{1 + 2 r}{3}} } |w(x) - w(y)|^2 \, d\mu \\
& \hspace{30pt} \le 4 \Lambda \| w \|_{L^\infty(\Sigma \cap \Omega)}^2 \int_{\Sigma \cap B_r} \bigg( \int_{(\Sigma \cap \Omega) \setminus B_{\frac{1 - r}{3}}(x)} \frac{d\sigma(y)}{|y - x|^{n + 2 s}} \bigg) d\sigma(x) < +\infty.
\end{align*}
From this and~\eqref{locsemifinite}, claim~\eqref{wHkfin} follows.

To finish the proof, we are left to show that inequality~\eqref{weakdefine} is satisfied for every non-negative function~$\varphi \in \H^K(\Sigma \cap B_r)$ supported inside~$\Sigma \cap B_r$. Pick any such function and, for~$M > 0$, let~$\varphi_M := \min \{ \varphi, M \}$. Multiplying~\eqref{Lwge0} by~$(1 - v_k) \varphi_M$, integrating over~$(\Sigma \setminus A_k) \cap B_1$, and arguing as for~\eqref{easylemeq1}-\eqref{psikovertildewtest}, we obtain
\begin{equation} \label{genisweaktech}
\begin{aligned}
& \frac{1}{2} \iint_{Q(\Sigma \cap B_r)} \left( w(x) - w(y) \right) \left( \varphi_M(x) - \varphi_M(y) \right) (1 - v_k(x)) \, d\mu \\
& \hspace{70pt} \ge \frac{1}{2} \iint_{Q(\Sigma \cap B_r)} \left( w(x) - w(y) \right) \left( v_k(x) - v_k(y) \right) \varphi_M(y) \, d\mu \\
& \hspace{70pt} \quad + \int_{\Sigma \cap B_r} \big( f(x) - b(x) w(x) \big) (1 - v_k(x)) \varphi_M(x) \, d\sigma(x).
\end{aligned}
\end{equation}
Since both~$w$ and~$\varphi_M$ belong to~$\H^K(\Sigma \cap B_r)$, Lebesgue's dominated convergence theorem gives that the left-hand side of~\eqref{genisweaktech} converges to
$$
\frac{1}{2} \iint_{Q(\Sigma \cap B_r)} \left( w(x) - w(y) \right) \left( \varphi_M(x) - \varphi_M(y) \right) d\mu
$$
as~$k \rightarrow +\infty$. On the other hand, as~$\supp(\varphi_M) \subset \Sigma \cap \Omega$ and~$K$ satisfies~\eqref{Ksupp}, the first term on the right-hand side of~\eqref{genisweaktech} is integrated over~$Q(\Sigma \cap B_r) \cap (\Omega \times \Omega)$. Hence, using that~$0 \le \varphi_M \le M$,~\eqref{vkcap},~\eqref{wHkfin}, and the Cauchy-Schwarz inequality, we infer its convergence to~$0$. As the second term can be easily addressed by dominated convergence (recall that~$b$ and~$f$ are assumed to be bounded), we conclude that
\begin{align*}
& \frac{1}{2} \iint_{Q(\Sigma \cap B_r)} \left( w(x) - w(y) \right) \left( \varphi_M(x) - \varphi_M(y) \right) d\mu \ge \int_\Sigma \big( f(x) - b(x) w(x) \big) \varphi_M(x) \, d\sigma(x).
\end{align*}

We now let~$M \rightarrow +\infty$ in the above inequality. The limit can be carried out on both sides using dominated convergence. On the left, we also take advantage of the fact that~$|\varphi_M(x) - \varphi_M(y)| \le |\varphi(x) - \varphi(y)|$ and of the Cauchy-Schwarz inequality. From this,~\eqref{weakdefine} follows.
\end{proof}

We can now proceed to the statement of our weak Harnack inequality.

\begin{theorem} \label{weakharthm}
Let~$n \ge 1$,~$s \in (1/2, 1)$, and assume conditions~\eqref{Ksymm}-\eqref{fged} to hold. Let~$w$ be a supersolution of~\eqref{Lv=0} in~$\Sigma \cap B_{4 R}$, for some~$R \in (0, R_0/4)$, such that
\begin{equation} \label{wge0thm}
w \ge 0 \quad \mbox{in } \Sigma \cap \Omega.
\end{equation}

Then, there exists an exponent~$\bar{p} > 1$, given by the value~$\bar{p} = \theta$ in~\eqref{thetadef} and in particular depending only on~$n$ and~$s$, such that for every~$p \in (0, \bar{p})$ it holds
\begin{align*}
& \inf_{\Sigma \cap B_{R}} w + R^{2 s} d \\
& \hspace{30pt} \ge c \left\{ \Bigg( \dashint_{\Sigma \cap B_R} w^{p}(x) \, d\sigma(x) \Bigg)^{\frac{1}{p}} + R^{2 s} \inf_{x \in \Sigma \cap B_{R / 2}} \int_{\Sigma \setminus B_R} w(y) K(x, y) \, d\sigma(y) \right\},
\end{align*} 
for some constant~$c \in (0, 1]$ depending only on~$n$,~$s$,~$c_\star$,~$C_\star$,~$\Lambda$,~$b_*$, and~$p$.
\end{theorem}

Note that, according to~\eqref{wge0thm}, the supersolution~$w$ is required to be non-negative in the whole set~$\Sigma \cap \Omega$ (that is, as far as the operator is extended), and not just in~$\Sigma \cap B_{4 R}$ where the equation is posed. When~$\Sigma = \R^n$, this condition has been proved in~\cite{Kas07,Kas11} to be necessary for the Harnack inequality to hold in this form.

We prove Theorem~\ref{weakharthm} by following the classical Moser iteration method, adapted to our fractional and non-Euclidean framework. To carry it out, we need the fractional Sobolev inequalities developed in Section~\ref{sobinesec}. For technical reasons, it will be convenient to have an exponent~$2 \theta \le 4$. For some~$n$ and~$s$, this will force~$2 \theta$ to be smaller than the fractional Sobolev exponent. This is why we define
\begin{equation} \label{thetadef}
\theta :=
\begin{dcases}
2 & \quad \mbox{if } n \le 2, \mbox{ or } n = 3 \mbox{ and } s \in \left[ \frac{3}{4}, 1 \right), \\
\frac{n}{n - 2 s} & \quad \mbox{if } n = 3 \mbox{ and } s \in \left( \frac{1}{2}, \frac{3}{4} \right), \mbox{ or } n \ge 4,
\end{dcases}
\end{equation}
for which we have~$2 < 2 \theta \le 4$. Note that there exists a constant~$C_\bullet$ depending only on~$n$,~$s$,~$c_\star$, and~$C_\star$, such that
\begin{equation} \label{sobproharnack}
\| v \|_{L^{2\theta}(\Sigma \cap B_{\rho_1})} \le C_\bullet \, \rho_2^{s - \frac{\theta - 1}{2 \theta} n} \left( [v]_{H^s(\Sigma \cap B_{\rho_2})} + \frac{1}{(\rho_2 - \rho_1)^{s}} \| v \|_{L^2(\Sigma \cap B_{\rho_2})} \right)
\end{equation}
holds for every~$0 < \rho_1 < \rho_2 \le R$ and all~$\Haus^n$-measurable functions~$v: \Sigma \cap B_R \to \R$. Inequality~\eqref{sobproharnack} follows from the results of Section~\ref{sobinesec}---in particular, from estimate~\eqref{restrictedsobine} when~$n \ge 2 > 2 s$ and from~\eqref{morreyine} along with H\"older's inequality when~$n = 1 < 2 s$. Note that the case~$n = 1 \ge 2 s$ never occurs since we assume~$s \in (1/2, 1)$.

We can now head to the proof of Theorem~\ref{weakharthm}. Thanks to Lemma~\ref{genisweaklem}, it suffices to present it for the case of weak supersolutions.

We start with the following lower bound for the infimum of~$w$ in terms of its norm in Lebesgue spaces of small, negative summability. Here, and also in Lemma~\ref{q<1lem}, we need to consider two independent scales~$R > r$, instead of just~$R$ and~$R/2$ as sometimes customary with Moser iterations---see, e.g.,~inequality~(4.4) of~\cite{M61}. This will be important in order to apply an abstract result, Lemma~\ref{JNlem}, that replaces in our setting the classical John-Nirenberg inequality.

\begin{lemma} \label{q>1lem}
Let~$w$ be a supersolution of~\eqref{Lv=0} in~$\Sigma \cap B_R$, for some~$R \in (0, R_0 / 2)$, satisfying
\begin{equation} \label{wged}
w \ge R^{2 s} d + \varepsilon \quad \mbox{in } \Sigma \cap B_R
\end{equation}
for some~$\varepsilon > 0$.

Then, there exist two constants~$\gamma_1 > 0$ and~$c > 0$ such that, for all~$p_0 \in (0, 1]$ and~$r \in [R/2, R)$,
\begin{equation} \label{q>1ine}
\inf_{\Sigma \cap B_r} w \ge \left\{ c \left( \frac{R - r}{R} \right)^{\gamma_1} \right\}^{\frac{1}{p_0}} \left( \dashint_{\Sigma \cap B_{R}} w(x)^{- p_0} \, d\sigma(x) \right)^{- \frac{1}{p_0}}.
\end{equation}
The constant~$\gamma_1$ depends only on~$n$ and~$s$, while $c$ depends only on~$n$,~$s$,~$c_\star$,~$C_\star$,~$\Lambda$, and~$b_*$.
\end{lemma}

\begin{proof}
Up to scaling, we may suppose that~$R = 2$. Given~$1 \le r \le \tau < t \le 2$, we consider a cutoff function~$\eta \in C^\infty(\R^{n + 1})$ satisfying~$0 \le \eta \le 1$ in~$\R^{n + 1}$,~$\supp(\eta) \subset B_t$,~$\eta = 1$ in~$B_{(t + \tau) / 2}$, and~$|\nabla \eta| \le 4 / (t - \tau)$ in~$\R^{n + 1}$. Given~$q > 1$, we test the weak formulation~\eqref{weakdefine} with the non-negative function~$\varphi = \eta^{q + 1} w^{-q}$, obtaining
\begin{align*}
& \int_{\Sigma} \int_{\Sigma} \left( w(y) - w(x) \right) \left( \frac{\eta_k^{q + 1}(x)}{w^{q}(x)} - \frac{\eta^{q + 1}(y)}{w^{q}(y)} \right) d\mu \\
& \hspace{120pt} \le 2 \int_{\Sigma} \left( b(x) w(x) - f(x) \right) \frac{\eta^{q + 1}(x)}{w^q(x)} \, d\sigma(x).
\end{align*}
Using Lemma~\ref{numericlem2} and assumptions~\eqref{Kell}-\eqref{fged},
this yields
\begin{equation} \label{p>1tech1}
\begin{aligned}
& \int_{\Sigma} \int_{\Sigma} \eta(x) \eta(y) \left| \left( \frac{\eta(x)}{w(x)} \right)^{\frac{q - 1}{2}} - \left(\frac{\eta(y)}{w(y)} \right)^{\frac{q - 1}{2}} \right|^2 \frac{d\sigma(x) \, d\sigma(y)}{|x - y|^{n + 2 s}} \\
& \hspace{40pt} \le C q^2 \Bigg\{ \int_{\Sigma \cap B_t} \left( \frac{\eta(x)}{w(x)} \right)^{q - 1} \left( \int_{\Sigma \cap \Omega} \frac{\left| \eta(x) - \eta(y) \right|^2}{|x - y|^{n + 2 s}} \, d\sigma(y) \right) d\sigma(x) \\
& \hspace{40pt} \quad + \int_{\Sigma} \left( R_0^{- 2 s} b_* w(x) + d \right) \frac{\eta^{q + 1}(x)}{w^q(x)} \, d\sigma(x) \Bigg\}
\end{aligned}
\end{equation}
where, from here on,~$C \ge 1$ is a constant depending only on~$n$,~$s$,~$c_\star$,~$C_\star$,~$\Lambda$, and~$b_*$.

Now, on the one hand
\begin{equation} \label{p>1tech2}
\begin{aligned}
& \int_{\Sigma} \int_{\Sigma} \eta(x) \eta(y) \left| \left( \frac{\eta(x)}{w(x)} \right)^{\frac{q - 1}{2}} - \left( \frac{\eta(y)}{w(y)} \right)^{\frac{q - 1}{2}} \right|^2 \frac{d\sigma(x) \, d\sigma(y)}{|x - y|^{n + 2 s}} \\
& \hspace{70pt} \ge \iint_{\big( \Sigma \cap B_{\frac{t + \tau}{2}} \big)^2} \left| w^{- \frac{q - 1}{2}}(x) - w^{- \frac{q - 1}{2}}(y) \right|^2 \frac{d\sigma(x) \, d\sigma(y)}{|x - y|^{n + 2 s}}.
\end{aligned}
\end{equation}
On the other hand, by Corollary~\ref{etasemilem}---which can be applied thanks to the second inequality in~\eqref{EperCflat}---, we have
$$
\int_{\Sigma \cap B_t} \left( \frac{\eta(x)}{w(x)} \right)^{q - 1} \left( \int_{\Sigma \cap \Omega} \frac{\left| \eta(x) - \eta(y) \right|^2}{|x - y|^{n + 2 s}} \, d\sigma(y) \right) d\sigma(x) \le C \frac{\left\| w^{- q + 1} \right\|_{L^1(\Sigma \cap B_t)}}{(t - \tau)^{2 s}}.
$$
Also, by~\eqref{wged},
\begin{align*}
\int_{\Sigma} \left( R_0^{- 2 s} b_* w(x) + d \right) \frac{\eta^{q + 1}(x)}{w^q(x)} \, d\sigma(x) & \le \frac{b_* + 1}{R^{2 s}} \int_{\Sigma} \frac{\eta^{q + 1}(x)}{w^{q- 1}(x)} \, d\sigma(x) \\
& \le C \, \frac{ \left\| w^{- q + 1} \right\|_{L^1(\Sigma \cap B_t)}}{(t - \tau)^{2 s}}.
\end{align*}
By plugging~\eqref{p>1tech2} together with the last two inequalities into~\eqref{p>1tech1}, we obtain that
$$
\left[ w^{- \frac{q - 1}{2}} \right]_{H^{s} \left( \Sigma \cap B_{\frac{t + \tau}{2}} \right) }^2 \le \frac{C q^2}{(t - \tau)^{2 s}} \left\| w^{- q + 1} \right\|_{L^1(\Sigma \cap B_t)}.
$$
Taking advantage of inequality~\eqref{sobproharnack}, this leads to
\begin{equation} \label{p>1tech4}
\left\| w^{- q + 1} \right\|_{L^{\theta}(\Sigma \cap B_\tau)} \le C q^2 \, \frac{t^{2 s - \frac{\theta - 1}{\theta} n} }{(t - \tau)^{2 s}} \left\| w^{- q + 1} \right\|_{L^1(\Sigma \cap B_t)},
\end{equation}
where~$\theta$ is the constant defined in~\eqref{thetadef}.

Let now~$p_0 \in (0, 1]$ (as in the statement of the lemma) and consider the three sequences~$\{ r_j \}$,~$\{ p_j \}$, and~$\{ \Phi_j \}$ defined by
$$
r_j := r + 2^{- j}(2 - r), \quad p_j := \theta^j p_0, \quad \mbox{and} \quad \Phi_j := \left( \dashint_{\Sigma \cap B_{r_j}} w(x)^{- p_j} \, d\sigma(x) \right)^{\frac{1}{p_j}}
$$
for every non-negative integer~$j$. By applying estimate~\eqref{p>1tech4} with~$\tau = r_{j + 1}$,~$t = r_j$, and~$q = 1 + p_j$, it is not hard to see that, for every integer~$j \ge 0$, we have
\begin{equation} \label{Phijest}
\Phi_{j + 1} \le \lambda_j \Phi_j \quad \mbox{with} \quad \lambda_j := \Big\{ C \, 2^{ 2 s j} (1 + p_j)^2 (2 - r)^{- 2 s} \Big\}^{\frac{1}{p_j}} \ge 1.
\end{equation}
Since the terms~$\Phi_j$ involve averages, to get the above relation we used once again~\eqref{EperCflat} to bound from above and below the~$n$-dimensional measures of the sets~$\Sigma \cap B_{r_j}$.

Notice now that
\begin{align*}
\frac{\log C}{p_i} & = \frac{\log C}{p_0} \, \theta^{-i} \le \frac{C}{p_0} (1 + \delta)^{- i}, \\
\frac{\log 2^{2 s i}}{p_i} & = \frac{2 s \log 2}{p_0} \, \theta^{-i} i \le \frac{C}{p_0} (1 + \delta)^{- i}, \quad \mbox{and} \\
\frac{\log (1 + p_i)^2}{p_i} & = \frac{2}{p_0} \log \left( 1 + p_0 \theta^i \right) \theta^{-i} \le \frac{C}{p_0} (1 + \delta)^{- i}
\end{align*}
for a small~$\delta > 0$ depending only on~$n$,~$s$,~$c_\star$,~$C_\star$,~$\Lambda$, and~$b_*$. Moreover,
$$
\prod_{i = 0}^{+\infty} (2 - r)^{ - \frac{2 s}{p_i}} = (2 - r)^{ - \frac{2 s}{p_0} \sum\limits_{i = 0}^{+\infty} \theta^{-i} } = (2 - r)^{- \frac{2 s \theta}{(\theta - 1) p_0} },
$$
and hence
\begin{align*}
\prod_{i = 0}^{+\infty} \lambda_i & = \left\{ \prod_{i = 0}^{+\infty} (2 - r)^{- \frac{2 s}{p_i}} \right\} \exp \left\{ \sum_{i = 0}^{+\infty} \left( \frac{\log C}{p_i} + \frac{\log 2^{2 s i}}{p_i} + \frac{\log (1 + p_i)^2}{p_i} \right) \right\} \\
& \le (2 - r)^{- \frac{2 s \theta}{(\theta - 1) p_0}} \exp \left\{ \frac{C}{p_0} \sum_{i = 0}^{+\infty} (1 + \delta)^{- i} \right\} \le C^{\frac{1}{p_0}}  (2 - r)^{- \frac{2 s \theta}{(\theta - 1) p_0}}.
\end{align*}
This and~$\Phi_j \le \Phi_0 \prod_{i = 0}^{j - 1} \lambda_i \le \Phi_0 \prod_{i = 0}^{+\infty} \lambda_i$, which follows by iterating~\eqref{Phijest}, lead to
\begin{align*}
\sup_{\Sigma \cap B_r} w^{-1} & \le \limsup_{j \rightarrow +\infty} \Phi_j \le C^{\frac{1}{p_0}} (2 - r)^{- \frac{2 s \theta}{(\theta - 1) p_0}} \Phi_0 \\
& = C^{\frac{1}{p_0}} (2 - r)^{- \frac{\gamma_1}{p_0}} \left( \dashint_{\Sigma \cap B_2} w(x)^{- p_0} \, d\sigma(x) \right)^{\frac{1}{p_0}},
\end{align*}
with~$\gamma_1 := 2 s \theta / (\theta - 1)$. From this, inequality~\eqref{q>1ine} follows.
\end{proof}

We proceed with a second lemma, where we get a uniform bound for the~$\mbox{BMO}$ seminorm of the logarithm of the supersolution~$w$. This will essentially allow us to control the integrals of small positive powers of~$w$ with those of small negative powers. Its proof follows a similar strategy to the one of Lemma~\ref{genisweaklem}, showing that generalized pointwise supersolutions are weak supersolutions.

\begin{lemma} \label{loglem}
Let~$w$ be a supersolution of~\eqref{Lv=0} in~$\Sigma \cap B_{2 R}$, for some~$R \in (0, R_0/2)$. Suppose that
\begin{equation} \label{wgedRn}
w \ge R^{2 s} d + \varepsilon \quad \mbox{in } \Sigma \cap B_{2 R}
\end{equation}
for some~$\varepsilon > 0$.

Then, there exists a constant~$C$ depending only on~$n$,~$s$,~$c_\star$,~$C_\star$,~$\Lambda$, and~$b_*$, such that
\begin{equation} \label{logine}
\dashint_{\Sigma \cap B_R} \left| \log w(x) - \left( \log w \right)_{\Sigma \cap B_R} \right| d\sigma(x) \le C.
\end{equation}
\end{lemma}
\begin{proof}
Without loss of generality, we may restrict ourselves to the case~$R = 1$. Take a cutoff~$\eta \in C^\infty(\R^{n + 1})$ such that~$0 \le \eta \le 1$ in~$\R^{n + 1}$,~$\supp(\eta) \subseteq B_{3 / 2}$,~$\eta = 1$ in~$B_1$, and~$|\nabla \eta| \le 4$ in~$\R^{n + 1}$. Testing inequality~\eqref{weakdefine} with the function~$\varphi=\eta^2 w^{-1}$ and arguing as we did at the beginning of the proof of Lemma~\ref{genisweaklem}, one obtains that
\begin{equation} \label{logesteq1}
\int_{\Sigma \cap B_1} \int_{\Sigma \cap B_1} \frac{|w(x) - w(y)|^2}{w(x) w(y)} \frac{d\sigma(x) \, d\sigma(y)}{|x - y|^{n + 2 s}} \le C,
\end{equation}
for some constant~$C$ depending only on the quantities declared in the statement. Note that, instead of proceeding as in~\eqref{RHSest}, to estimate the right-hand side and zeroth order term of~\eqref{weakdefine} one needs to take advantage of~\eqref{wgedRn} and~\eqref{EperCflat}, obtaining
\begin{align*}
\int_{\Sigma} \left( R_0^{- 2 s} b_* w(x) + d \right) \frac{\eta^2(x)}{w(x)} \, d\sigma(x) & \le (b_* + 1) \int_{\Sigma \cap B_2} \eta^2(x) \, d\sigma(x) \le C.
\end{align*}

Due to Lemma~\ref{lognumericlem}, it follows from~\eqref{logesteq1} that
$$
[\log w]_{H^{s}(\Sigma \cap B_1)} \le C.
$$
Using Jensen's inequality, the left bound in~\eqref{EperCflat}, and Proposition~\ref{poinineprop}, we then deduce that
\begin{align*}
\dashint_{\Sigma \cap B_1} \left| \log w(x) - \left( \log w \right)_{\Sigma \cap B_1} \right| d\sigma(x) & \le \frac{\| \log w - (\log w)_{\Sigma \cap B_1} \|_{L^2(\Sigma \cap B_1)}}{\Haus^{n}(\Sigma \cap B_1)^{\frac{1}{2}}} \\
& \le C [\log w]_{H^{s}(\Sigma \cap B_1)} \le C.
\end{align*}
Thus,~\eqref{logine} is established.
\end{proof}

Finally, we connect small positive powers of~$w$ with the possibly larger power~$p$ appearing in the statement of Theorem~\ref{weakharthm}. Moreover, we obtain a bound from below for the~$L^1$ norm of~$w$ in terms of its tail, i.e.,~a quantity like the second term on the right-hand side of the conclusion of Theorem~\ref{weakharthm}.

\begin{lemma} \label{q<1lem}
Let~$p_1 \in (0, \theta)$ be given, with~$\theta > 1$ as in~\eqref{thetadef}. Let~$w$ be a supersolution of~\eqref{Lv=0} in~$\Sigma \cap B_R$, for some~$R \in (0, R_0 / 2)$, satisfying
\begin{equation} \label{wgedglobal}
w \ge R^{2 s} d + \varepsilon \quad \mbox{in } \Sigma \cap \Omega
\end{equation}
for some~$\varepsilon > 0$. 

Then, there exist two constants~$\gamma_2 > 0$ and~$C > 0$ such that, for every
\begin{equation} \label{pp0bound}
p \in \left(0, \frac{p_1}{\theta} \right]
\end{equation}
and every~$r \in \left[R/2, R \right)$, it holds
\begin{equation} \label{q<1ine}
\Bigg( \dashint_{\Sigma \cap B_{r}} w^{p_1}(x) \, d\sigma(x) \Bigg)^{\frac{1}{p_1}} \le \left\{ C \left( \frac{R}{R - r} \right)^{\gamma_2} \right\}^{\frac{1}{p} - \frac{1}{p_1}} \Bigg( \dashint_{\Sigma \cap B_R} w^{p}(x) \, d\sigma(x) \Bigg)^{\frac{1}{p}}.
\end{equation}
The constant~$\gamma_2$ depends only on~$n$ and~$s$, while~$C$ only on~$n$,~$s$,~$c_\star$,~$C_\star$,~$\Lambda$,~$p_1$, and~$b_*$. Furthermore,
\begin{equation} \label{q<1inebonus}
\inf_{x \in \Sigma \cap B_{R / 2}} \int_{\Sigma \setminus B_R} w(y) K(x, y) \, d\sigma(y) \le \frac{C'}{R^{2 s}} \, \dashint_{\Sigma \cap B_R} w(x) \, d\sigma(x),
\end{equation}
for some constant~$C'$ depending only on~$n$,~$s$,~$c_\star$,~$C_\star$,~$\Lambda$, and~$b_*$.
\end{lemma}
\begin{proof}
By scaling, we may assume~$R = 1$. Let~$1/2 \le \tau < t \le 1$ be two radii and~$\eta \in C^\infty(\R^{n + 1})$ be a cutoff function satisfying~$0 \le \eta \le 1$ in~$\R^{n + 1}$, $\supp(\eta) \subset B_{(2 t + \tau) / 3}$,~$\eta = 1$ in~$B_{(t + 2 \tau) / 3}$, and~$|\nabla \eta| \le 6 / (t - \tau)$ in~$\R^{n + 1}$. Let~$q \in (0, 1)$.

By testing~\eqref{weakdefine} with the non-negative function~$\varphi = \eta^2 w^{-q}$ and taking advantage of~\eqref{Ksymm}, we obtain that
\begin{equation} \label{q<1tech0}
I^{(1)} + 2 I^{(2)} + D \ge 0,
\end{equation}
where
\begin{align*}
I^{(1)} := & \, \int_{\Sigma \cap B_t} \int_{\Sigma \cap B_t} \left( w(x) - w(y) \right) \left( \frac{\eta^2(x)}{w^{q}(x)} - \frac{\eta^2(y)}{w^{q}(y)} \right) d\mu, \\
I^{(2)} := & \, \int_{\Sigma \cap B_t} \frac{\eta^2(x)}{w^{q}(x)} \left( \int_{\Sigma \setminus B_t} \left( w(x) - w(y) \right) K(x, y) \, d\sigma(y) \right) d\sigma(x),
\end{align*}
and
$$
D := 2 \int_{\Sigma} \left( R_0^{- 2 s} b_* w(x) + d \right) \frac{\eta^2(x)}{w^q(x)} \, d\sigma(x).
$$

We begin by dealing with~$I^{(1)}$. Taking advantage of Lemma~\ref{numericlem}, we have that
\begin{align*}
I^{(1)} & \le - \frac{q}{2} \int_{\Sigma \cap B_t} \int_{\Sigma \cap B_t} \left| w^{\frac{1 - q}{2}}(x) - w^{\frac{1 - q}{2}}(y) \right|^2 \min \{ \eta(x), \eta(y) \}^2 \, d\mu \\
& \quad + \frac{4}{q} \int_{\Sigma \cap B_t} \int_{\Sigma \cap B_t} \max \{ w(x), w(y) \}^{1 - q} \left| \eta(x) - \eta(y) \right|^2 \, d\mu.
\end{align*}
On the one hand, by taking advantage of hypotheses~\eqref{Kell}-\eqref{0inOmega}, and applying Corollary~\ref{etasemilem}, we estimate
\begin{align*}
& \int_{\Sigma \cap B_t} \int_{\Sigma \cap B_t} \max \{ w(x), w(y) \}^{1 - q} \left| \eta(x) - \eta(y) \right|^2 \, d\mu \\
& \hspace{60pt} \le 2 \Lambda \int_{\Sigma \cap B_t} w^{1 - q}(x) \left(\int_{\Sigma \cap \Omega} \frac{\left| \eta(x) - \eta(y) \right|^2}{|x - y|^{n + 2 s}} \, d\sigma(y) \right) d\sigma(x) \\
& \hspace{60pt} \le C \, \frac{\| w^{1 - q} \|_{L^1(\Sigma \cap B_t)}}{(t - \tau)^{2 s}},
\end{align*}
where, from here on,~$C \ge 1$ depends only on~$n$,~$s$,~$c_\star$,~$C_\star$,~$\Lambda$, and~$b_*$. On the other hand,~\eqref{Kell} yields
\begin{align*}
& \int_{\Sigma \cap B_t} \int_{\Sigma \cap B_t} \left| w^{\frac{1 - q}{2}}(x) - w^{\frac{1 - q}{2}}(y) \right|^2 \min \{ \eta(x), \eta(y) \}^2 \, d\mu \ge \Lambda^{-1} \left[ w^{\frac{1 - q}{2}} \right]_{H^{s} \left( \Sigma \cap B_{ \frac{t + 2 \tau}{3}} \right)}^2.
\end{align*}
By the last three inequalities, we deduce that
\begin{equation} \label{q<1tech2}
I^{(1)} \le - \frac{q}{C} \left[ w^{\frac{1 - q}{2}} \right]_{H^s \left( \Sigma \cap B_{ \frac{t + 2 \tau}{3}} \right)}^2 + \frac{C q^{-1}}{{(t - \tau)^{2 s}}} \, \| w^{1 - q} \|_{L^1(\Sigma \cap B_t)}.
\end{equation}

We now look at the term~$I^{(2)}$. We split it as~$I^{(2)} = I^{(2 1)} - I^{(2 2)}$, with
\begin{align*}
I^{(2 1)} & := \int_{\Sigma \cap B_t} \eta^2(x) w^{1 - q}(x) \left( \int_{\Sigma \setminus B_t} K(x, y) \, d\sigma(y) \right) d\sigma(x), \\
I^{(2 2)} & := \int_{\Sigma \cap B_t} \frac{\eta^2(x)}{w^{q}(x)} \left( \int_{\Sigma \setminus B_t} w(y) K(x, y) \, d\sigma(y) \right) d\sigma(x).
\end{align*}
To bound~$I^{(2 1)}$, we use~\eqref{Kell}-\eqref{0inOmega} and Lemma~\ref{kerL1overredElem}. We obtain
\begin{align*}
I^{(2 1)} & \le \Lambda \int_{\Sigma \cap B_{(2 t + \tau) / 3}} \eta^2(x) w^{1 - q}(x) \left( \int_{(\Sigma \cap \Omega) \setminus B_{(t - \tau) / 3}(x)} \frac{ d\sigma(y) }{|y - x|^{n + 2 s}} \right) d\sigma(x) \\
& \le \frac{C}{(t - \tau)^{2 s}} \, \| w^{1 - q} \|_{L^1(\Sigma \cap B_t)}.
\end{align*}
On the other hand, using that, by~\eqref{wgedglobal},~$w \ge \varepsilon > 0$ in~$\Sigma \cap \Omega$ and, once again, hypotheses~\eqref{Ksupp} and~\eqref{0inOmega},
\begin{align*}
I^{(2 2)} & \ge \int_{\Sigma \cap B_{(t + 2 \tau) / 3}} w^{-q}(x) \left( \int_{\Sigma \setminus B_t} w(y) K(x, y) \, d\sigma(y) \right) d\sigma(x) \\
& \ge \| w^{- q} \|_{L^1(\Sigma \cap B_\tau)} \left( \inf_{x \in \Sigma \cap B_\tau} \int_{\Sigma \setminus B_t} w(y) K(x, y) \, d\sigma(y) \right).
\end{align*}
The combination of the last two inequalities gives that
$$
I^{(2)} \le C \, \frac{\| w^{1 - q} \|_{L^1(\Sigma \cap B_t)}}{(t - \tau)^{2 s}} - \| w^{- q} \|_{L^1(\Sigma \cap B_\tau)} \left( \inf_{x \in \Sigma \cap B_\tau} \int_{\Sigma \setminus B_t} w(y) K(x, y) \, d\sigma(y) \right).
$$

To estimate~$D$, we simply use~\eqref{wgedglobal} to deduce
$$
D \le 2 (b_* + 1) \int_{\Sigma} \frac{\eta^2(x)}{w^{q - 1}(x)} \, d\sigma(x) \le \frac{C}{(t - \tau)^{2 s}} \, \| w^{1 - q} \|_{L^1(\Sigma \cap B_t)}.
$$

By putting together the last two bounds with~\eqref{q<1tech0} and~\eqref{q<1tech2}, we get
\begin{equation} \label{q<1tech2.5}
\begin{aligned}
& q \left[ w^{\frac{1 - q}{2}} \right]_{H^{s} \left( \Sigma \cap B_{ \frac{t + 2 \tau}{3}} \right)}^2 + \| w^{- q} \|_{L^1(\Sigma \cap B_\tau)} \left( \inf_{x \in \Sigma \cap B_\tau} \int_{\Sigma \setminus B_t} w(y) K(x, y) \, d\sigma(y) \right) \\
& \hspace{250pt} \le \frac{C q^{- 1}}{(t - \tau)^{2 s}} \, \| w^{1 - q} \|_{L^1(\Sigma \cap B_t)}.
\end{aligned}
\end{equation}

At this stage (but not at a further one), we discard the second term on the left-hand side of~\eqref{q<1tech2.5} and we infer, with the aid of inequality~\eqref{sobproharnack}, that
\begin{equation} \label{q<1tech3}
\| w^{1 - q} \|_{L^\theta(\Sigma \cap B_\tau)} \le \frac{C}{q^2 (t - \tau)^{2 s}} \, \| w^{1 - q} \|_{L^1(\Sigma \cap B_t)}
\end{equation}
for every~$1/2 \le \tau < t \le 1$. Let~$p$ and~$p_1$ be as in the statement, and let~$\bar{k}$ be the positive integer for which
\begin{equation} \label{kbardef}
p_1 \theta^{- \bar{k}} \le p < p_1 \theta^{- \bar{k} + 1}.
\end{equation}
We consider the two finite sequences~$\{ p_k \}$ and~$\{ r_k \}$ defined by
$$
p_k := \theta^{-k + 1} p_1 \quad \mbox{and} \quad r_k := r + 2^{k - \bar{k} - 1} (1 - r)
$$
for every integer~$k = 1, \ldots, \bar{k} + 1$. Applying~\eqref{q<1tech3} with~$\tau = r_k$,~$t = r_{k + 1}$,~$q = 1 - p_{k + 1}$ and recalling hypothesis~\eqref{EperCflat}, it is not hard to see that, for every~$k = 1, \ldots, \bar{k}$,
$$
\left( \dashint_{\Sigma \cap B_{r_k}} w^{p_k}(x) \, d\sigma(x) \right)^{\frac{1}{p_k}} \le \lambda_k \left( \dashint_{\Sigma \cap B_{r_{k + 1}}} w^{p_{k + 1}}(x) \, d\sigma(x) \right)^{\frac{1}{p_{k + 1}}},
$$
with
$$
\lambda_k := \left\{ \frac{2^{2 s (\bar{k} - k)} C}{(1 - p_{k + 1})^2 (1 - r)^{2 s}} \right\}^{\frac{1}{p_{k + 1}}}.
$$
Therefore,
$$
\Bigg( \dashint_{\Sigma \cap B_{r_1}} w^{p_1}(x) \, d\sigma(x) \Bigg)^{\frac{1}{p_1}} \le \left( \prod_{k = 1}^{\bar{k}} \lambda_k \right) \Bigg( \dashint_{\Sigma \cap B_1} w^{p_{\bar{k} + 1}}(x) \, d\sigma(x) \Bigg)^{\frac{1}{p_{\bar{k} + 1}}}.
$$
Thus, using~\eqref{EperCflat} once again and, possibly, H\"older's inequality, we obtain that
\begin{equation} \label{q<1tech3.5}
\Bigg( \dashint_{\Sigma \cap B_r} w^{p_1}(x) \, d\sigma(x) \Bigg)^{\frac{1}{p_1}} \le C \left( \prod_{k = 1}^{\bar{k}} \lambda_k \right) \Bigg( \dashint_{\Sigma \cap B_1} w^{p}(x) \, d\sigma(x) \Bigg)^{\frac{1}{p}},
\end{equation}
since~$1/2 \le r \le r_1 \le 1$ and~$p_{\bar{k} + 1} \le p$, by the way~$\bar{k}$ was chosen in~\eqref{kbardef}.

As it can be easily checked,~\eqref{kbardef} and~\eqref{pp0bound} lead to
$$
\frac{1}{p_{\bar{k} + 1}} - \frac{1}{p_1} < \frac{\theta}{p} - \frac{1}{p_1} \le (\theta + 1) \left( \frac{1}{p} - \frac{1}{p_1} \right).
$$
Thanks to this inequality, we have
\begin{equation} \label{q<1tech4}
\sum_{k = 1}^{\bar{k}} \frac{1}{p_{k + 1}} = \frac{1}{p_1} \sum_{k = 1}^{\bar{k}} \theta^{k} = \frac{1}{p_1} \frac{\theta (\theta^{\bar{k}} - 1)}{\theta - 1} = \frac{\theta}{\theta - 1} \left( \frac{1}{p_{\bar{k} + 1}} - \frac{1}{p_1} \right) \le \frac{\gamma_2}{2 s} \left( \frac{1}{p} - \frac{1}{p_1} \right),
\end{equation}
where we set~$\gamma_2 := 2 \theta (\theta + 1) s / (\theta - 1)$. On the other hand, a simple but tedious computation reveals that
$$
2 s \sum_{k = 1}^{\bar{k}} \frac{\bar{k} - k}{p_{k + 1}} \le \frac{2 \theta s}{(\theta - 1)^2} \left( \frac{1}{p_{\bar{k} + 1}} - \frac{1}{p_1} \right) \le \frac{\gamma_2}{\theta - 1} \left( \frac{1}{p} - \frac{1}{p_1} \right).
$$
Here the first inequality can be checked by induction on~$\bar{k}$, for instance. Thanks to this,~\eqref{q<1tech4}, and the inequality~$1 - p_{k + 1} \ge 1 - p_2 = (\theta - p_1)/\theta$ for every~$k = 1, \ldots, \bar{k}$, we have
$$
\prod_{k = 1}^{\bar{k}} \lambda_k \le \left\{ \frac{C \, \theta^2}{\left( \theta - p_1 \right)^2 (1 - r)^{2 s}} \right\}^{\sum\limits_{k = 1}^{\bar{k}} \frac{1}{p_{k + 1}} } 2^{2 s \sum\limits_{k = 1}^{\bar{k}} \frac{\bar{k} - k}{p_{k + 1}} } \le \left\{ \frac{C}{\left( \theta - p_1 \right)^{\frac{\gamma_2}{s}} {(1 - r)}^{\gamma_2}} \right\}^{\frac{1}{p} - \frac{1}{p_1}}.
$$
By combining this with~\eqref{q<1tech3.5}, we arrive at estimate~\eqref{q<1ine}, provided~$C$ is taken to depend also on~$p_1$.

We now go back to~\eqref{q<1tech2.5}. By dropping this time the first term on its left-hand side, and applying it with~$\tau = 1/2$,~$t = 1$, and~$q = 1/2$, we find that
$$
\inf_{x \in \Sigma \cap B_{1 / 2}} \int_{\Sigma \setminus B_1} w(y) K(x, y) \, d\sigma(y) \le C' \, \| w^{1/2} \|_{L^1(\Sigma \cap B_1)}\| w^{- 1/2} \|_{L^1(\Sigma \cap B_{1/2})}^{-1}.
$$
By taking advantage of Lemma~\ref{q>1lem}, hypothesis~\eqref{EperCflat}, and H\"older's inequality, we then obtain
\begin{align*}
\inf_{x \in \Sigma \cap B_{1 / 2}} \int_{\Sigma \setminus B_1} w(y) K(x, y) \, d\sigma(y) & \le C' \, \| w^{1/2} \|_{L^1(\Sigma \cap B_1)} \inf_{\Sigma \cap B_{1 / 4}} w^{1/2} \\
& \le C' \, \| w^{1/2} \|_{L^1(\Sigma \cap B_1)}^2 \le C' \, \| w \|_{L^1(\Sigma \cap B_1)},
\end{align*}
which gives~\eqref{q<1inebonus}.
\end{proof}

We are now in position to put together the last three lemmas and conclude Theorem~\ref{weakharthm}. To do this, we need one more ingredient: an abstract lemma from~\cite{S-C02} that plays the role of the classical John-Nirenberg inequality originally used by Moser in~\cite{M61}. Note that we use this lemma in place of, say,~\cite[Theorem~4]{BG72}, due to the slightly better flexibility of the former. Indeed, the result of~\cite{S-C02} allows for the presence of different exponents ($\gamma_1$ versus~$\gamma_2$) in the constants that govern the estimates of Lemmas~\ref{q>1lem} and~\ref{q<1lem}. See also~\cite{AMPP16} for a variant of this abstract lemma set in a nonlocal parabolic framework.

We include next the statement of this abstract result for the convenience of the reader. Note that, in our setting, the space~$X$ will be the surface~$\Sigma$ endowed with the measure~$\mu = \Haus^n \llcorner \Sigma$ and~$\{ U_\sigma \}$ the intersection of~$\Sigma$ with ambient balls of~$\R^{n + 1}$.

\begin{lemma}[{\cite[Lemma~2.2.6]{S-C02}}] \label{JNlem}
Let~$(X, \mu)$ be a measure space and~$\{ U_\sigma \}_{\sigma \in [1, 2]}$ be a collection of measurable subsets of~$X$ satisfying~$U_{\sigma'} \subseteq U_{\sigma}$ for every~$1 \le \sigma' \le \sigma \le 2$. Let~$C_\sharp > 0$,~$\gamma > 0$, and~$0 < q_0 \le q_1 \le +\infty$.

Let~$f: U_2 \to \R$ be a positive measurable function satisfying
\begin{equation} \label{fcond1}
\| f \|_{L^{q_1}(U_{\sigma'}; \, d\mu)} \le \left\{ \frac{C_\sharp \, \mu(U_2)^{-1}}{(\sigma - \sigma')^\gamma} \right\}^{\frac{1}{q} - \frac{1}{q_1}} \| f \|_{L^q(U_\sigma; \, d\mu)}
\end{equation}
for every~$1 \le \sigma' < \sigma \le 2$ and every~$0 < q \le \min \{ 1, q_1/2 \}$. Assume further that
\begin{equation} \label{fcond2}
\mu \left( \left\{ \log f > \lambda \right\} \right) \le \frac{C_\sharp \, \mu(U_2)}{\lambda}
\quad \mbox{for every } \lambda > 0.
\end{equation}

Then,
$$
\| f \|_{L^{q_1}(U_1; \, d\mu)} \le C \, \mu(U_2)^{\frac{1}{q_1}}
$$
for some constant~$C$ depending only on~$C_\sharp$,~$\gamma$, and~$q_0$.
\end{lemma}

With this in hand, we can complete the proof of the weak Harnack inequality. 

\begin{proof}[Proof of Theorem \ref{weakharthm}]
As usual, we assume without loss of generality that~$R = 1$. Set~$\bar{w} := w + 2^{2 s} d + \varepsilon$, with~$\varepsilon > 0$ small. Notice that~$\bar{w}$ is also a weak supersolution of~\eqref{Lv=0} in~$\Sigma \cap B_{4}$ (with~$b_+$ in place of~$b$) and that it satisfies~$\bar{w} \ge 2^{2 s} d + \varepsilon$ in~$\Sigma \cap \Omega$, thanks to assumption~\eqref{wge0thm}.
 
First of all, we claim that, for every~$p_1 \in (0, \theta)$, it holds
\begin{equation} \label{weakharclaim}
\inf_{\Sigma \cap B_1} \bar{w} \ge c \Bigg( \dashint_{\Sigma \cap B_1} \bar{w}^{\, p_1}(x) \, d\sigma(x) \Bigg)^{\frac{1}{p_1}},
\end{equation}
for some constant~$c \in (0, 1]$ depending only on~$n$,~$s$,~$c_\star$,~$C_\star$,~$\Lambda$,~$b_*$, and~$p_1$.

To prove this claim, we may assume~$p_1 \ge 1$, thanks to H\"older's inequality. Following the proof of Theorem~2.3.1 in~\cite{S-C02}, we apply Lemma~\ref{JNlem} to the functions~$e^{- \kappa} \bar{w}$ and~$e^\kappa \bar{w}^{-1}$, with~$\kappa := (\log \bar{w})_{\Sigma \cap B_2}$,~$\mu = \Haus^n \llcorner \Sigma$, and~$U_\sigma = \Sigma \cap B_\sigma$. We observe that condition~\eqref{fcond2} is fulfilled by both 
these functions, since by Lemma~\ref{loglem} (applied here with~$R = 2$),
\begin{align*}
\Haus^{n} \! \left( \Sigma \cap B_2 \cap \left\{ \left| \log \left( e^{- \kappa} \bar{w} \right) \right| > \lambda \right\} \right) & \le \frac{1}{\lambda} \int_{\Sigma \cap B_2} \left| \log \bar{w} (x) - \kappa \right| d\sigma(x) \\
& \le \frac{C \, \Haus^{n} \left( \Sigma \cap B_2 \right)}{\lambda}
\end{align*}
for every~$\lambda > 0$, where~$C \ge 1$ depends only on~$n$,~$s$,~$c_\star$,~$C_\star$,~$\Lambda$,~$b_*$, and~$p_1$.

Next, we claim that~$e^{- \kappa} \bar{w}$ satisfies~\eqref{fcond1} with~$q_1 = p_1$,~$q_0 = 1$, and~$\gamma = \gamma_2 > 0$ as in the statement of Lemma~\ref{q<1lem} and depending only on~$n$ and~$s$. Indeed, estimate~\eqref{q<1ine}, condition~\eqref{EperCflat}, and the fact that~$\theta \le 2$ ensure the validity of the estimate
\begin{align*}
& \Bigg( \int_{\Sigma \cap B_{r}} |e^{- \kappa} \bar{w}(x)|^{p_1} \, d\sigma(x) \Bigg)^{\frac{1}{p_1}} \\
& \hspace{60pt} \le \left\{ \frac{C \, \Haus^n(\Sigma \cap B_2)^{-1} }{\left( \rho - r \right)^{\gamma_2}} \right\}^{\frac{1}{q} - \frac{1}{p_1}} \Bigg( \int_{\Sigma \cap B_{\rho}} |e^{- \kappa} \bar{w}(x)|^q \, d\sigma(x) \Bigg)^{\frac{1}{q}}
\end{align*}
for every~$1 \le r < \rho \le 2$ and~$q \in (0, p_1 / 2]$. On the other hand, by Lemma~\ref{q>1lem},
$$
\sup_{\Sigma \cap B_r} \left| e^\kappa \bar{w}^{-1} \right| \le \left\{ \frac{C \, \Haus^n(\Sigma \cap B_2)^{-1} }{\left( \rho - r \right)^{\gamma_1}} \right\}^{\frac{1}{q}} \Bigg( \int_{\Sigma \cap B_{\rho}} |e^{\kappa} \bar{w}^{-1}(x)|^{q} \, d\sigma(x) \Bigg)^{\frac{1}{q}}
$$
for every~$1 \le r < \rho \le 2$,~$q \in (0, 1]$, and with~$\gamma_1 > 0$ (as in the statement of Lemma~\ref{q>1lem}) depending only on~$n$ and~$s$. Hence~\eqref{fcond1} is also verified by~$e^\kappa \bar{w}^{-1}$, this time with~$\gamma = \gamma_1$ and~$q_1 = +\infty$.

In light of Lemma~\ref{JNlem}, we infer that
$$
\Bigg( \dashint_{\Sigma \cap B_1} |e^{- \kappa} \bar{w}(x)|^{p_1} \, d\sigma(x) \Bigg)^{\frac{1}{p_1}} \le C \qquad \mbox{and} \qquad \sup_{\Sigma \cap B_1} \left| e^\kappa \bar{w}^{-1} \right| \le C.
$$
The combination of these two inequalities leads us to~\eqref{weakharclaim}.

Next, we notice that, by taking~$p_1 = 1$ in~\eqref{weakharclaim} and recalling estimate~\eqref{q<1inebonus} in Lemma~\ref{q<1lem} (applied with~$R = 1$), it immediately follows that
$$
\inf_{\Sigma \cap B_1} \bar{w} \ge c \, \inf_{x \in \Sigma \cap B_{1/2}} \int_{\Sigma \setminus B_1} \bar{w}(y) K(x, y) \, d\sigma(y).
$$

Since~$\bar{w} = w + 2^{2 s} d + \varepsilon$, the conclusion of Theorem~\ref{weakharthm} (with~$p = p_1 \in (0, \theta)$ and~$\bar{p} = \theta$) can be reached by putting together the last inequality with~\eqref{weakharclaim} and letting~$\varepsilon \rightarrow 0^+$.
\end{proof}

\subsection{The weak Harnack inequality on nonlocal minimal surfaces and graphs}

We now apply the results of the previous subsection to~$\alpha$-minimal surfaces and, in particular,~$\alpha$-minimal graphs. The sets~$\Omega$ and~$\Omega'$ considered in the next result will be later taken to be either all of~$\R^{n + 1}$ or vertical cylinders~$\C_\rho$.

\begin{corollary} \label{weakharoverEcor}
Let~$n \ge 1$,~$\Omega \subseteq \Omega'$ be two open subsets of~$\R^{n + 1}$, and suppose that
$$
\dist(\Omega, \R^{n + 1} \setminus \Omega') \ge \mu_0 \quad \mbox{for some } \mu_0 \in (0, +\infty]
$$
---when~$\Omega' = \R^{n + 1}$, we take~$\mu_0 = +\infty$. For~$\alpha \in (0, 1)$, let~$\partial E \subset \R^{n + 1}$ be an~$\alpha$-minimal surface in~$\Omega'$ and set~$\Sigma := \red E$, with~$\red E$ the reduced boundary of~$E$.

Let~$K: \red E \times \red E \to \R$ be a non-negative kernel satisfying hypotheses~\eqref{Ksymm}-\eqref{Ksupp} for some~$R_0 \in (0, \mu_0 / 2]$,~$\Lambda \ge 1$, and~$s \in (1/2, 1)$. Also assume the validity of~\eqref{0inOmega}.

Let~$b, f: \red E \cap B_{R_0} \to \R$ be two bounded~$\Haus^n$-measurable functions satisfying~\eqref{bgeb*} and~\eqref{fged} for some~$b_*, d \ge 0$. Given a radius~$R \in (0, R_0 / 4)$, let~$w$ be a supersolution of
$$
- \L_K v + b v = f \quad \mbox{in } \red E \cap B_{4 R}
$$
such that~$w \ge 0$ in~$\red E \cap \Omega$.

Then, there exists an exponent~$\bar{p} > 1$ depending only on~$n$ and~$s$, such that for every~$p \in (0, \bar{p})$ it holds
\begin{equation} \label{weakharineonredE}
\begin{aligned}
\inf_{\red E \cap B_{R}} w + R^{2 s} d & \ge c \left\{ \Bigg( \dashint_{\red E \cap B_R} w^{p}(x) \, d\sigma(x) \Bigg)^{\frac{1}{p}} \right. \\
& \quad \left. \rule{0pt}{24pt} + R^{2 s} \inf_{x \in \red E \cap B_{R/2}} \int_{\red E \setminus B_R} w(y) K(x, y) \, d\sigma(y) \right\}
\end{aligned}
\end{equation}
for some constant~$c \in (0, 1]$ depending only on~$n$,~$s$,~$\alpha$,~$\Lambda$,~$b_*$, and~$p$.
\end{corollary}
\begin{proof}
The corollary is an immediate consequence of Theorems~\ref{perimeterboundsthm} and~\ref{weakharthm}. Indeed, by Theorem~\ref{perimeterboundsthm} there exist two constants~$C_\star \ge c_\star > 0$ depending only on~$n$ and~$\alpha$, for which
$$
c_\star \, \rho^n \le \Haus^n(\red E \cap B_\rho(x)) \le C_\star \rho^n \mbox{ for every } x \in \red E \cap \Omega \mbox{ and } \rho \in \left( 0, \frac{\mu_0}{2} \right].
$$
As, by assumption,~$R_0 \le \mu_0 / 2$, the hypotheses of Theorem~\ref{weakharthm} are satisfied. Hence, inequality~\eqref{weakharineonredE} plainly follows.
\end{proof}

Notice that, when~$\Omega = \Omega' = \R^{n + 1}$,~$b = f = 0$, and the kernel~$K$ takes the form~$K(x, y) = |x - y|^{- n - 2 s}$, Corollary~\ref{weakharoverEcor} yields Theorem~\ref{wHthm} of the introduction.

We now present an application of the weak Harnack inequality to~$\alpha$-minimal surfaces in bounded cylinders. This result will be of key importance in the next section, where we will establish the gradient estimate for nonlocal minimal graphs.

\begin{corollary} \label{Jcor}
Let~$n \ge 1$ and~$\alpha \in (0, 1)$. Let~$\partial E$ be an~$\alpha$-minimal surface in the cylinder~$\C_{32 r}$ and suppose that
\begin{equation} \label{redEincappedcyl}
\red E \cap \C_r \subset B_r' \times (- M, M),
\end{equation}
for some positive constants~$r$ and~$M$. Let~$b: \red E \cap \C_{8 r} \to \R$ be a bounded~$\Haus^n$-measurable function satisfying~$b \le r^{- 2 s} b_*$ in~$\red E \cap \C_{8 r}$ for some~$b_* \ge 0$, and~$K$ be the truncated kernel given by
$$
K(x, y) := \frac{\chi_{\C_{16 r}}(x) \chi_{\C_{16 r}}(y)}{|x - y|^{n + 2 s}} \quad \mbox{for all } x, y \in \red E,
$$
for some~$s \in (1/2, 1)$.

Let~$w$ be a non-negative supersolution of~\eqref{Lv=0} in~$\red E \cap \C_{8 r}$, with~$f = 0$. Assume further that, for some constant~$c_\flat > 0$,
\begin{equation} \label{nondegcond0}
\int_{\red E \cap \C_r} w(y) \, d\sigma(y) \ge c_\flat r^n.
\end{equation}

Then,
\begin{equation} \label{infwpower}
\inf_{\red E \cap \C_r} w \ge c \left( 1 + \frac{M}{r} \right)^{- n - 2 s}
\end{equation}
for some constant~$c > 0$ depending only on~$n$,~$s$,~$\alpha$,~$b_*$, and~$c_\flat$.
\end{corollary}

\begin{proof}
Up to a rescaling, we can assume that~$r = 1$. We apply Corollary~\ref{weakharoverEcor} as follows. Take any point~$\bar{x} \in \red E \cap \C_1$. We translate~$\bar{x}$ to the origin and note that, in this new reference frame,~$\partial (E - \bar{x})$ is an~$\alpha$-minimal surface inside the cylinder~$\C_{32}(-\bar{x})$, while~$w(\cdot \, + \bar{x})$ is a non-negative supersolution of~\eqref{Lv=0} in~$\red (E - \bar{x}) \cap \C_8(-\bar{x})$, with~$f = 0$ and kernel~$K_{\bar{x}}(x, y) := \chi_{\C_{16}(-\bar{x})}(x) \chi_{\C_{16}(-\bar{x})}(y) |x - y|^{- n - 2 s}$. Hence, the hypotheses of Corollary~\ref{weakharoverEcor} are satisfied with~$E - \bar{x}$ in place of~$E$,~$\Omega = \C_{16}(-\bar{x})$,~$\Omega' = \C_{32}(-\bar{x})$, and~$R_0 = 4$. Applying~\eqref{weakharineonredE} with~$p = 1$ and~$R = 1/2$, and switching back to the original coordinates, we conclude that
\begin{equation} \label{localtechest0}
\begin{aligned}
\inf_{\red E \cap B_{1/2}(\bar{x})} w & \ge c \left( \dashint_{\red E \cap B_{1/2}(\bar{x})} w(x) \, d\sigma(x) \right. \\
& \quad \left. + \inf_{x \in \red E \cap B_{1 / 4}(\bar{x})} \int_{\left( \red E \cap \C_1 \right) \setminus B_{1/2}(\bar{x})} \frac{w(y)}{|x - y|^{n + 2 s}} \, d\sigma(y) \right)
\end{aligned}
\end{equation}
for a constant~$c > 0$ depending only on~$n$,~$s$,~$\alpha$, and~$b_*$.

In view of Theorem~\ref{perimeterboundsthm},
\begin{equation} \label{localtechest}
\dashint_{\red E \cap B_{1/2}(\bar{x})} w(x) \, d\sigma(x) \ge c \int_{\left( \red E \cap \C_1 \right) \cap B_{1/2}(\bar{x})} w(x) \, d\sigma(x).
\end{equation}
Moreover, by hypothesis~\eqref{redEincappedcyl} and the fact that~$\bar{x} \in \red E \cap \C_1$, it holds
$$
|x - y| \le |x - \bar{x}| + |\bar{x}| + |y| \le 1/4 + 2 \sqrt{1 + M^2} \le 3 (1 + M)
$$
for every~$x \in \red E \cap B_{1 / 4}(\bar{x})$ and~$y \in \red E \cap \C_1$. Thus,
\begin{align*}
& \inf_{x \in \red E \cap B_{1 / 4}(\bar{x})} \int_{\left( \red E \cap \C_1 \right) \setminus B_{1/2}(\bar{x})} \frac{w(y)}{|x - y|^{n + 2 s}} \, d\sigma(y) \\
& \hspace{100pt} \ge \frac{c}{(1 + M)^{n + 2 s}} \int_{\left( \red E \cap \C_1 \right) \setminus B_{1/2}(\bar{x})} w(y) \, d\sigma(y).
\end{align*}
In light of the above estimate,~\eqref{localtechest0},~\eqref{localtechest}, and hypothesis~\eqref{nondegcond0}, we get that
$$
\inf_{\red E \cap B_{1/2}(\bar{x})} w  \ge \frac{c}{(1 + M)^{n + 2 s}} \int_{\red E \cap \C_1} w(y) \, d\sigma(y) \ge \frac{c}{(1 + M)^{n + 2 s}}.
$$
By the arbitrariness of~$\bar{x} \in \red E \cap \C_1$, we deduce that~\eqref{infwpower} holds true.
\end{proof}

Under the hypotheses of the last corollary and with the additional assumption that~$\Sigma$ is connected---which holds, for instance, when~$\Sigma$ is the graph of a continuous function, as in our application---one can obtain the estimate
\begin{equation} \label{infwexp}
\inf_{\red E \cap \C_r} w \ge \exp \left\{- C \left( 1 + \frac{M}{r} \right) \right\}
\end{equation}
with a different proof based on a chaining method appeared,~e.g.,~in~\cite[Corollary~3.2]{DSJ11}. We stress that this technique only makes use of the bound for the first term appearing on the right-hand side of the weak Harnack inequality~\eqref{weakharineonredE}. As a result, it produces the exponential estimate~\eqref{infwexp}, that is weaker than the power-type bound~\eqref{infwpower} in its dependence on~$M$. Arguing as in the proof of Theorem~\ref{localgradestthm} presented in Section~\ref{regsec}, one easily sees that~\eqref{infwexp} yields a gradient bound for nonlocal minimal graphs depending exponentially in the oscillation, analogous to the sharp estimate~\eqref{BDMest} for classical minimal graphs.

We present here below a brief sketch of the argument leading to~\eqref{infwexp}.

First, we cover~$\Sigma := \red E \cap \C_r$ with a family of Euclidean balls~$\B := \{ B^{(j)} \}_{j = 1}^N$ centered on~$\Sigma$ and with radii equal to a small (but universal) fraction of~$r$. Thanks to hypothesis~\eqref{redEincappedcyl}, the cardinality~$N$ can be chosen to be of order~$M / r$.

Take now two balls~$B^{(i)}$ and~$B^{(j)}$ having nonempty intersection. Applying Corollary~\ref{weakharoverEcor}, it is easy to get that
$$
\int_{\red E \cap B^{(i)}} w(x) \, d\sigma(x) \le C \int_{\red E \cap B^{(j)}} w(x) \, d\sigma(x)
$$
for some constant~$C > 1$ depending only on~$n$,~$s$,~$\alpha$, and~$b_*$. As~$\Sigma$ is connected, any two balls~$B^{(i)}$ and~$B^{(j)}$ can be joined by a connected chain of balls in~$\B$. By iterating the above inequality along this chain, we obtain a new estimate in which the constant~$C$ is now replaced by~$C^{1 + M/r}$.

Claim~\eqref{infwexp} follows using~\eqref{weakharineonredE}---with only the first term in its right-hand side---and noticing that hypothesis~\eqref{nondegcond0} ensures the existence of an index~$i$ for which
$$
\int_{\red E \cap B^{(i)}} w(x) \, d\sigma(x) \ge c \left( 1 + \frac{M}{r} \right)^{-1} r^n,
$$
for some~$c \in (0,1]$ depending only on~$n$,~$s$,~$\alpha$,~$b_*$, and~$c_\flat$.

\section{Regularity results} \label{regsec}

\noindent
We combine here the weak Harnack inequality of Section~\ref{weakharsec} with the results on the Jacobi operator presented in Section~\ref{jacsec} to deduce a gradient bound for nonlocal minimal graphs and, as a consequence, their smoothness. More specifically, we prove Theorems~\ref{localgradestthm} and~\ref{globalgradestthm} of the introduction.

\begin{proof}[Proof of Theorem~\ref{localgradestthm}]
Up to a covering argument, we may assume that~$\partial E$ minimizes the~$\alpha$-perimeter in the cylinder~$\C_{64 r}$, instead of~$\C_{2 r}$ as in the statement. Indeed, suppose that we have proved the gradient estimate~\eqref{gradbound} in this case. Going back to the hypotheses of Theorem~\ref{localgradestthm}, we pick a ball~$B_{r/64}'(\bar{x}') \subset B_r'$ such that~$\| \nabla_{\! x'} u \|_{L^\infty(B_r')} / 2 \le \| \nabla_{\! x'} u \|_{L^\infty(B_{r / 64}'(\bar{x}'))}$. Then,~$B'_{64 (r / 64)}(\bar{x}') = B_r'(\bar{x}') \subset B_{2 r}'$ and we may apply the gradient bound in~$B'_{r/64}(\bar{x}')$ obtaining that
$$
\| \nabla_{\! x'} u \|_{L^\infty(B_r')} \le C \left( 1 + \frac{\osc_{B'_{r / 64}(\bar{x}')} u}{r / 64} \right)^{n + 1 + \alpha} \le C \left( 1 + \frac{\osc_{B'_{r}} u}{r} \right)^{n + 1 + \alpha},
$$
as desired.

We now prove~\eqref{gradbound} assuming that~$\partial E$ is~$\alpha$-minimal in~$\C_{64 r}$. By~Theorem~\ref{regsingthm}, we know that~$\partial E \cap \overline{\C_{32 r}}$ is of class~$C^\infty$ outside of a closed singular set~$S$ of Hausdorff dimension at most~$n - 2$. By Proposition~\ref{capsingthm}, we also have
\begin{equation} \label{cap=0}
\begin{cases}
S = \varnothing & \quad \mbox{if } n = 1,\\
\Cap_{\, \Sigma \cap \C_{32 r}, \frac{1 + \alpha}{2}}(S) = 0 & \quad \mbox{if } n \ge 2.
\end{cases}
\end{equation}
Denote by~$\pi': \R^{n + 1} \to \R$ the projection along the~$(n + 1)$-th coordinate direction and let~$S' := \pi'(S \cap \C_r)$. This is a closed set relatively to~$B_r'$, by a compactness argument using that~$u$ is bounded in~$B_r'$. In addition, as~$S$ is at most~$(n - 2)$-dimensional, we have that~$S'$ satisfies
\begin{equation} \label{Ssmall}
\Haus^{d}(S') = 0 \quad \mbox{for all } d > n - 2.
\end{equation}
See,~e.g.,~Corollary~1 in Section~2.4.1 of~\cite{EG92} for a proof of this simple fact.

We can now proceed to prove the theorem. Observe that, up to a vertical translation, we may suppose that~$u(0) = 0$. This means in particular that
\begin{equation} \label{osccontrolsLinfty}
\osc_{B'_{r}} u \ge \| u \|_{L^\infty(B'_{r})}.
\end{equation}
Consider the truncated kernel
$$
K_r(x, y) := \frac{\chi_{\C_{16 r}(x)} \chi_{\C_{16 r}(y)}}{|x - y|^{n + 1 + \alpha}} \quad \mbox{for all } x, y \in \R^{n + 1},
$$
and the corresponding truncated fractional Laplace-type operator given by
$$
\L_r v(x) := \PV \int_{\red E} \left( v(y) - v(x) \right) K_r(x, y) \, dy \quad \mbox{ for } x \in \red E. 
$$
Let~$w := \nu_E^{n + 1}$ be the~$(n + 1)$-th component of the outer unit normal~$\nu_E$ to~$E$, defined on~$\red E$. The function~$w$ is smooth in~$\left( \partial E \setminus S \right) \cap \C_{32 r}$ and given by
\begin{equation} \label{w=1/sqrt}
w(x', u(x')) = \frac{1}{\sqrt{1 + | \nabla_{\! x'} u(x') |^2 }} \quad \mbox{for~} x' \in B'_{r} \setminus S'. 
\end{equation}
In view of Theorem~\ref{Jacintrothm}$\,(ii)$, it holds
$$
- \L_r w(x) + \frac{C}{r^{1 + \alpha}} w(x) \ge 0 \quad \mbox{for~} x \in \left( \partial E \setminus S \right) \cap \C_{8 r},
$$
for some constant~$C > 0$ depending only on~$n$ and~$\alpha$. Recalling~\eqref{cap=0}, this means that~$w$ is a generalized pointwise supersolution of~\eqref{Lv=0} in~$\red E \cap \C_{8 r}$ (in the sense of Definition~\ref{esspointsoldef}), with~$K = K_r$,~$s = (1 + \alpha)/2$,~$f = 0$, and~$b = C r^{- 1 - \alpha}$. Observe that the limit defining the principal value involved in the definition of~$\L_r$ converges uniformly away from~$S$, by the regularity of~$\red E$. Moreover,~$w$ is non-negative in~$\red E \cap \C_{32 r}$, since~$E$ is a global subgraph. As~$\partial E \cap \C_{r}$ is the graph of~$u$ and~\eqref{osccontrolsLinfty} is in force, we see that~\eqref{redEincappedcyl} holds with, say,~$M = 2 \osc_{B_{r}'} u$. Finally, by~\eqref{Ssmall} and~\eqref{w=1/sqrt},
\begin{align*}
\int_{\red E \cap \C_r} w(y) \, d\sigma(y) & = \int_{B'_r \setminus S'} w(y', u(y')) \sqrt{1 + |\nabla_{\! x'} u(y')|^2} \, dy' = \left| B'_r \right| = \left| B'_1 \right| r^n.
\end{align*}
Hence, condition~\eqref{nondegcond0} is satisfied.

As a consequence of these facts, we may apply Corollary~\ref{Jcor} to deduce that
$$
\inf_{x' \in B_r' \setminus S'} w(x', u(x')) \ge c \left( 1 + \frac{\osc_{B_{r}'} u}{r} \right)^{- n - 1 - \alpha}
$$
for some constant~$c \in (0, 1]$ depending only on~$n$ and~$\alpha$. Recalling~\eqref{w=1/sqrt}, we infer that~$u \in W^{1, \infty}(B_r' \setminus S')$ and
$$
\| \nabla_{\! x'} u \|_{L^\infty(B_r' \setminus S')} \le c^{-1} \left( 1 + \frac{\osc_{B'_{r}} u}{r} \right)^{n + 1 + \alpha}.
$$
Since~$S'$ is closed in~$B_r'$ and~$\Haus^{n - 1}(S') = 0$ by~\eqref{Ssmall}, this~$W^{1, \infty}$ bound in~$B_r' \setminus S'$ yields that~$u$ is actually Lipschitz in the whole~$B_r'$ and that estimate~\eqref{gradbound} holds. This is a consequence of a known removability result for the Sobolev space~$W^{1, \infty}$---see, e.g.,~\cite[Theorem~1.2.5]{MP97}. The fact that~$u$ is actually smooth is a consequence of~\cite[Theorem~1.1]{FV17}.
\end{proof}

As noted in~\cite{BG72} for the case of classical minimal surfaces, estimate~\eqref{gradbound} can be improved when dealing with graphs that are minimal in the whole space. In such a situation, the gradient of~$u$ can be bounded in terms of a better power of the oscillation, as claimed in Theorem~\ref{globalgradestthm} of the introduction. Here is a proof of this result.

\begin{proof}[Proof of Theorem~\ref{globalgradestthm}]
By Theorem~\ref{localgradestthm}, we already know that~$u$ is smooth in the whole~$\R^n$. As in the previous proof, we may assume after a vertical translation that
\begin{equation} \label{osccontrolsLinftybis}
\osc_{B'_r} u \ge \| u \|_{L^\infty(B'_r)}.
\end{equation}

We consider the outer unit normal~$\nu_E$ to~$E$ and let~$w := \nu_E^{n + 1}$ be its~$(n + 1)$-th component. We have~\eqref{w=1/sqrt} for every~$x' \in \R^n$ and~$w \ge 0$ on the whole~$\partial E$. Moreover, by Theorem~\ref{Jacintrothm}$\,(i)$ we know that
$$
- \L w(x) \ge 0 \quad \mbox{for~} x \in \partial E,
$$
where
$$
\L v(x) := \PV \int_{\partial E} \frac{v(y) - v(x)}{|x - y|^{n + 1 + \alpha}} \, d\sigma(y).
$$
Therefore,~$w$ is a supersolution of~\eqref{Lv=0} in the whole~$\partial E$, with~$b = f = 0$. Accordingly, we may apply Corollary~\ref{weakharoverEcor} (with~$\Omega = \Omega' = \R^{n + 1}$,~$R_0 = +\infty$, and~$p = 1$) to obtain, for every~$R > 0$,
\begin{equation} \label{infwgec}
\inf_{B_R} w \ge c \, \dashint_{\partial E \cap B_R} w(x) \, d\sigma(x)
\end{equation}
for some~$c \in (0, 1]$ depending only on~$n$ and~$\alpha$.

Let now~$r > 0$ and set
\begin{equation} \label{Rdef}
R := \sqrt{ r^2 + \| u \|^2_{L^\infty(B'_r)} }.
\end{equation}
With this choice we have that~$B'_r \subseteq \pi' \left( \partial E \cap B_R \right)$, where~$\pi': \R^{n + 1} \to \R^{n}$ is the vertical projection. By this,~\eqref{infwgec},~\eqref{w=1/sqrt}, and the perimeter bound~\eqref{perimeterboundabove} of Theorem~\ref{perimeterboundsthm}, we get that
\begin{align*}
\sup_{B'_r} \sqrt{1 + |\nabla_{\! x'} u|^2} & \le c^{-1} \, \Haus^{n}(\partial E \cap B_R) \left( \int_{\partial E \cap B_R} \frac{d\sigma(x)}{\sqrt{1 + |\nabla_{\! x'} u(x')|^2}} \right)^{-1} \\
& = c^{-1} \, \frac{\Haus^{n} \left( \partial E \cap B_R \right)}{\left| \pi' \left( \partial E \cap B_R \right) \right|} \le c^{-1} \left( \frac{R}{r} \right)^{n}
\end{align*}
for some possibly smaller~$c$. Estimate~\eqref{bettergradbound} follows from~\eqref{osccontrolsLinftybis} and~\eqref{Rdef}.
\end{proof}

\section{Rigidity results} \label{rigsec}

\noindent
We prove here two flatness results for entire nonlocal minimal graphs.

Consider the operator~$\mathfrak{H}_\alpha$ defined in~\eqref{Halphadef}, and note that, if~$u$ is smooth at~$x' \in \R^n$, then~$\mathfrak{H}_\alpha u(x')$ is well-defined in the principal value sense. One can see this by adding and subtracting the term
$$
\PV \int_{B'_\rho(x')} G_\alpha \left( \nabla_{\! x'} u(x') \cdot \frac{y' - x'}{|y' - x'|} \right) \, \frac{dy'}{|y' - x'|^{n + \alpha}},
$$
for any given radius~$\rho > 0$, from the definition of~$\mathfrak{H}_\alpha u(x')$. Note that such term vanishes by symmetry---recall that~$G_\alpha$ is odd. Hence, we may write~$\mathfrak{H}_\alpha u(x')$ in the form
\begin{equation} \label{Hualtdef}
\int_{\R^n} \left\{ G_\alpha \left( \frac{u(x') - u(y')}{|x' - y'|} \right) - \chi_{B_\rho'(x')}(y') \, G_\alpha \left( \nabla_{\! x'} u(x') \cdot \frac{x' - y'}{|x' - y'|} \right) \right\} \frac{dy'}{|x' - y'|^{n + \alpha}}.
\end{equation}
From this representation, it easily follows that~$\mathfrak{H}_\alpha u$ is well-defined if~$u$ is~$C^2$ at~$x'$---see,~e.g.,~\cite[Remark~B.2]{BLV16} for a thorough verification of this fact. Note that, unlike the fractional Laplacian, the operator~$\mathfrak{H}_\alpha$ does not require any boundedness or growth assumption at infinity on the function~$u$ to be well-defined. This is the case because~$G_\alpha$ is bounded.

The following is a Liouville-type theorem for globally Lipschitz nonlocal minimal graphs. A more general version of it has been established very recently by Farina~\&~Valdinoci~\cite{FarV17}. Independently of their work, we had found the result below, with a different proof based on the Harnack inequality for integro-differential operators in Euclidean space with kernels bounded above and below by multiples of that of the fractional Laplacian. We stress that our proof, as theirs, does not rely on any of the results presented in the previous sections.

\begin{theorem}[\cite{FarV17}] \label{Liprigthm}
Let~$n \ge 1$ and~$\alpha \in (0, 1)$. Let~$E$ be the subgraph
$$
E = \Big\{ (x', x_{n + 1}) \in \R^{n} \times \R : x_{n + 1} < u(x') \Big\}
$$
of a globally Lipschitz function~$u: \R^n \to \R$, and suppose that~$\partial E$ is an~$\alpha$-minimal surface in~$\R^{n + 1}$.

Then,~$u$ is affine, or equivalently~$\partial E$ is a hyperplane.
\end{theorem}
\begin{proof}
By~\cite[Theorem~1.1]{FV17}, a result on Lipschitz nonlocal minimal surfaces, the function~$u$ is actually smooth. By the uniform Lipschitz assumption it holds
\begin{equation} \label{ugloblip}
| \nabla_{\! x'} u(x') | \le C_0 \quad \mbox{for } x' \in \R^n,
\end{equation}
for some constant~$C_0$. Recalling that~$\mathfrak{H}_\alpha u(x')$ may be written as in~\eqref{Hualtdef}, we have
$$
\int_{\R^n} \left\{ G_\alpha \left( \frac{u(x' + z') - u(x')}{|z'|} \right) - \chi_{B_1'}(z') \, G_\alpha \left( \nabla_{\! x'} u(x') \cdot \frac{z'}{|z'|} \right) \right\} \frac{dz'}{|z'|^{n + \alpha}} = 0.
$$

We differentiate this identity with respect to~$x_i$, for~$i = 1, \ldots, n$. We get
\begin{equation} \label{lioutech1}
\begin{aligned}
& \int_{\R^n} \bigg\{ G_\alpha' \left( \frac{u(x' + z') - u(x')}{|z'|} \right) \big( u_{x_i}(x' + z') - u_{x_i}(x') \big) \\
& \hspace{40pt} - \chi_{B_1'}(z') \, G_\alpha' \left( \nabla_{\! x'} u(x') \cdot \frac{z'}{|z'|} \right) \big( \nabla_{\! x'} u_{x_i}(x') \cdot z' \big) \bigg\} \frac{dz'}{|z'|^{n + 1 + \alpha}} = 0.
\end{aligned}
\end{equation}
Notice now that, since~$G_\alpha'$ is even,
$$
\PV \int_{B_1'} G_\alpha' \left( \nabla_{\! x'} u(x') \cdot \frac{z'}{|z'|} \right) \big( \nabla_{\! x'} u_{x_i}(x') \cdot z' \big) \, \frac{dz'}{|z'|^{n + 1 + \alpha}} = 0.
$$
Accordingly,~\eqref{lioutech1} becomes
$$
\PV \int_{\R^n} G_\alpha' \left( \frac{u(x' + z') - u(x')}{|z'|} \right) \big( u_{x_i}(x' + z') - u_{x_i}(x') \big) \frac{dz'}{|z'|^{n + 1 + \alpha}} = 0.
$$
Switching back to the old variables, we conclude that
$$
\PV \int_{\R^n} \big( u_{x_i}(x') - u_{x_i}(y') \big) \widetilde{K}(x', y') \, dy' = 0 \quad \mbox{for every } x' \in \R^n \mbox{ and } i = 1, \ldots, n,
$$
where we set
$$
\widetilde{K}(x', y') := G_\alpha' \left( \frac{u(x') - u(y')}{|x' - y'|} \right) \frac{1}{|x' - y'|^{n + 1 + \alpha}}.
$$

Let~$\beta_i := \inf_{\R^n} u_{x_i}$, for~$i = 1, \ldots, n$, and define~$v^{(i)}(x') := u_{x_i}(x') - \beta_i$. These are bounded~$C^2$ functions which are non-negative in~$\R^n$ and satisfy
\begin{equation} \label{infvi=0}
\inf_{\R^n} v^{(i)} = 0
\end{equation}
and
\begin{equation} \label{veq}
\PV \int_{\R^n} \big( v^{(i)}(x') - v^{(i)}(y') \big) \widetilde{K}(x', y') \, dy' = 0 \quad \mbox{for~} x' \in \R^n.
\end{equation}
Observe now that the kernel~$\widetilde{K}$ is symmetric (since~$G_\alpha'$ is even) and that it satisfies
$$
\frac{\tilde{c}}{|x' - y'|^{n + 1 + \alpha}} \le \widetilde{K}(x', y') \le \frac{1}{|x' - y'|^{n + 1 + \alpha}} \quad \mbox{for~} x', y' \in \R^n,
$$
for some constant~$\tilde{c} \in (0, 1]$, since, by~\eqref{ugloblip},
$$
1 \ge G_\alpha' \left( \frac{u(x') - u(y')}{|x' - y'|} \right) = \left\{ 1 + \left( \frac{|u(x') - u(y')|}{|x' - y'|} \right)^2 \right\}^{- \frac{n + 1 + \alpha}{2}} \ge \left( 1 + C_0^2 \right)^{- \frac{n + 1 + \alpha}{2}}.
$$

In light of this, equation~\eqref{veq} is uniformly elliptic and thus, by,~e.g.,~\cite[Theorem~2.5]{C17}, each~$v^{(i)}$ satisfies the Harnack inequality. Recall that~$v^{(i)} \ge 0$ in all of~$\R^n$. We deduce that
$$
\sup_{B'_r} v^{(i)} \le C \inf_{B'_r} v^{(i)} \quad \mbox{for every } r > 0 \mbox{ and } i = 1, \ldots, n,
$$
for some constant~$C$ independent of~$r$. Letting now~$r \rightarrow +\infty$ and recalling~\eqref{infvi=0}, we deduce that
$$
0 \le \sup_{\R^n} v^{(i)} \le C \inf_{\R^n} v^{(i)} = 0 \quad \mbox{for every } i = 1, \ldots, n,
$$
i.e., each~$v^{(i)}$ is identically zero. The gradient of~$u$ is thus constant and~$u$ is affine.
\end{proof}

By combining Theorem~\ref{Liprigthm} with our gradient estimate for nonlocal minimal graphs, we can easily prove Theorem~\ref{growthrigthm}.

\begin{proof}[Proof of Theorem~\ref{growthrigthm}]
By Theorem~\ref{globalgradestthm}, we know that~$u \in C^\infty(\R^n)$. Taking advantage of~\eqref{bettergradbound} and~\eqref{growthass}, we infer that
$$
\| \nabla_{\! x'} u \|_{L^\infty(B_r')} \le C \left( 1 + \frac{2 \| u \|_{L^\infty(B_r')}}{r} \right)^n \le C \left( 1 + \frac{2 C (1 + r)}{r} \right)^n \le C
$$
for every~$r \ge 1$ and for some constant~$C$ independent of~$r$. Hence,~$u$ is globally Lipschitz, and the conclusion follows by virtue of Theorem~\ref{Liprigthm}.
\end{proof}

\appendix

\section{Numerical inequalities} \label{auxapp}

\noindent
The next three results have been used in the~Moser iteration performed in Section~\ref{weakharsec}. We begin with the following estimate due to~Kassmann~\cite{Kas09}.

\begin{lemma}[{\cite[Lemma~2.5]{Kas09}}] \label{numericlem2}
Let~$q > 1$. Then, for every~$a, b > 0$ and~$\sigma, \tau \ge 0$, it holds
\begin{align*}
(b - a) \left( \frac{\sigma^{q + 1}}{a^q} - \frac{\tau^{q + 1}}{b^q} \right) & \ge \frac{\sigma \tau}{q - 1} \left\{ \left( \frac{\tau}{b} \right)^{\frac{q - 1}{2}} - \left( \vphantom{\frac{\tau}{b}} \frac{\sigma}{a} \right)^{\frac{q - 1}{2}} \right\}^2 \\
& \quad - \max \left\{ 4, \frac{6 q - 5}{2} \right\} \left( \sigma - \tau \right)^2 \left\{ \left( \vphantom{\frac{\tau}{b}} \frac{\sigma}{a} \right)^{q - 1} + \left( \frac{\tau}{b} \right)^{q - 1} \right\}.
\end{align*}
\end{lemma}

The following simple result appeared, in a stronger version, in~\cite[Lemma~2.6]{Kas09}.

\begin{lemma} \label{lognumericlem}
For every~$a, b > 0$, it holds
\begin{equation} \label{lognumericine}
\left( \log a - \log b \right)^2 \le \frac{(a - b)^2}{a b}.
\end{equation}
\end{lemma}
\begin{proof}
Inequality~\eqref{lognumericine} is equivalent to the fact that
\begin{equation} \label{lognumericinetech}
\psi(t) := \log t - \sqrt{t} + \frac{1}{\sqrt{t}} \le 0 \quad \mbox{for every } t \ge 1.
\end{equation}

We differentiate~$\psi$ and obtain that
$$
\psi'(t) = \frac{1}{t} - \frac{1}{2 \sqrt{t}} - \frac{1}{2 t \sqrt{t}} = - \frac{1}{2 t \sqrt{t}} \left( - 2 \sqrt{t} + t + 1 \right) = - \frac{1}{2 t \sqrt{t}} \left( \sqrt{t} - 1 \right)^2 \le 0
$$
for every~$t \ge 1$. Therefore~$\psi(t) \le \psi(1) = 0$ for every~$t \ge 1$, and~\eqref{lognumericinetech} follows.
\end{proof}

We finish with another numerical inequality in the spirit of the two previous ones.

\begin{lemma} \label{numericlem}
Let~$q \in (0, 1)$. Then, for every~$a, b > 0$ and~$\sigma, \tau \ge 0$, it holds
$$
(a - b) \left( \frac{\sigma^2}{a^q} - \frac{\tau^2}{b^q} \right) \le - \frac{q}{2} \left( a^{\frac{1 - q}{2}} - b^{\frac{1 - q}{2}} \right)^2 \min \{ \sigma, \tau \}^2 + \frac{4}{q} \max \{ a, b \}^{1 - q} \left( \sigma - \tau \right)^2.
$$
\end{lemma}
\begin{proof}
Our argument follows a path similar to the one traced throughout the first part of the proof of~\cite[Lemma~5.1]{DCKP14}.

Clearly, we may assume that~$a > b$. We write
\begin{equation} \label{numerictech1}
(a - b) \left( \frac{\sigma^2}{a^q} - \frac{\tau^2}{b^q} \right) = \frac{a - b}{a^q} \left( \sigma^2 - \frac{a^q}{b^q} \tau^2 \right).
\end{equation}
An application of the weighted Young's inequality yields
$$
\sigma^2 = \tau^2 + (\sigma - \tau)^2 + 2 \tau (\sigma - \tau) \le \left( 1 + \varepsilon \right) \tau^2 + \frac{2}{\varepsilon} (\sigma - \tau)^2
$$
for every~$\varepsilon \in (0, 1]$. By taking advantage of this estimate in~\eqref{numerictech1} with the choice
$$
\varepsilon := \delta \, \frac{a - b}{a} \in (0, 1),
$$
for some~$\delta \in (0, 1)$ to be chosen later, we obtain
\begin{equation} \label{numerictech2}
\begin{aligned}
(a - b) \left( \frac{\sigma^2}{a^q} - \frac{\tau^2}{b^q} \right) & \le \frac{a - b}{a^q} \left\{ \left( 1 + \delta \, \frac{a - b}{a} - \frac{a^q}{b^q} \right) \tau^2 + \frac{2}{\delta} \, \frac{a}{a - b} (\sigma - \tau)^2 \right\} \\
& = \frac{(a - b)^2}{a^{1 + q}} \left\{ \delta - \left( \frac{a^q}{b^q} - 1 \right) \frac{a}{a - b} \right\} \tau^2 + \frac{2}{\delta} \, a^{1 - q} (\sigma - \tau)^2.
\end{aligned}
\end{equation}

We now compute
\begin{equation} \label{numerictech3}
\delta - \left( \frac{a^q}{b^q} - 1 \right) \frac{a}{a - b} = \delta - \left( \frac{a}{b} \right)^q \frac{1 - \left( \frac{b}{a} \right)^q}{1 - \frac{b}{a}} \le \delta - q,
\end{equation}
where the last inequality is a consequence of the fact that~$1 - t^q \ge q (1 - t) t^q$ for every~$t \in [0, 1]$. By choosing~$\delta = q/2$, from inequalities~\eqref{numerictech2} and~\eqref{numerictech3} we get
$$
(a - b) \left( \frac{\sigma^2}{a^q} - \frac{\tau^2}{b^q} \right) \le - \frac{q}{2} \frac{(a - b)^2}{a^{1 + q}} \tau^2 + \frac{4}{q} \, a^{1 - q} (\sigma - \tau)^2.
$$
Then, the conclusion of the lemma easily follows.
\end{proof}

\section*{Acknowledgments}

\noindent
We thank Joaquim Serra for several interesting conversations on the topic of this paper and, in particular, for important advice on the structure of the gradient estimate of Theorem~\ref{localgradestthm}. We also thank the referees for their interesting comments, which greatly contributed to the improvement of the paper.

\newpage

\end{document}